\documentclass[12pt]{amsart}
\usepackage{graphicx}
\usepackage{amsmath}
\usepackage{amsfonts}
\usepackage{amssymb}

\usepackage{setspace}
\usepackage{datetime}
\usepackage[colorlinks,
            linkcolor=black,
            anchorcolor=blue,
            citecolor=black]{hyperref}
\newtheorem{thm}{Theorem}[section]
\newtheorem{cor}[thm]{Corollary}

\newtheorem{lem}[thm]{Lemma}
\newtheorem{prop}[thm]{Proposition}
\theoremstyle{definition}

\theoremstyle{remark}
\newtheorem{rem}[thm]{Remark}
\theoremstyle{conclusion}

\theoremstyle{question}
\numberwithin{equation}{section}

\newcommand{\gee}{\geqslant}

\begin{document}
\title[Conformally invariant systems with mixed order]{Classification of solutions to conformally invariant systems with mixed order and exponentially increasing or nonlocal nonlinearity}

\author{Wei Dai, Guolin Qin}

\address{School of Mathematical Sciences, Beihang University (BUAA), Beijing 100191, P. R. China, and Key Laboratory of Mathematics, Informatics and Behavioral Semantics, Ministry of Education, Beijing 100191, P. R. China}
\email{weidai@buaa.edu.cn}

\address{Institute of Applied Mathematics, Chinese Academy of Sciences, Beijing 100190, P. R. China, and University of Chinese Academy of Sciences, Beijing 100049, P. R. China}
\email{qinguolin18@mails.ucas.ac.cn}

\thanks{Wei Dai is supported by the NNSF of China (No. 12222102 and 11971049) and the Fundamental Research Funds for the Central Universities.}

\begin{abstract}
In this paper, without any assumption on $v$ and under extremely mild assumption $u(x)=O(|x|^{K})$ at $\infty$ for some $K\gg1$ arbitrarily large, we prove classification of solutions to the following conformally invariant system with mixed order and exponentially increasing nonlinearity in $\mathbb{R}^{2}$:
\begin{equation*}\\\begin{cases}
(-\Delta)^{\frac{1}{2}}u(x)=e^{pv(x)}, \qquad x\in\mathbb{R}^{2}, \\
-\Delta v(x)=u^{4}(x), \qquad x\in\mathbb{R}^{2},
\end{cases}\end{equation*}
where $p\in(0,+\infty)$, $u\geq 0$ and satisfies the finite total curvature condition $\int_{\mathbb{R}^{2}}u^{4}(x)\mathrm{d}x<+\infty$. In order to show integral representation formula and crucial asymptotic property for $v$, we derive and use an $\exp^{L}+L\ln L$ inequality, which is itself of independent interest. When $p=\frac{3}{2}$, the system is closely related to single conformally invariant equations $(-\Delta)^{\frac{1}{2}}u=u^{3}$ and $-\Delta v=e^{2v}$ on $\mathbb{R}^{2}$, which have been quite extensively studied (cf. \cite{BF,C,CY,CL,CLL,CLZ} etc). We also derive classification results for nonnegative solutions to conformally invariant system with mixed order and Hartree type nonlocal nonlinearity in $\mathbb{R}^{3}$. Extensions to mixed order conformally invariant systems in $\mathbb{R}^{n}$ with general dimensions $n\geq3$ are also included.
\end{abstract}
\maketitle {\small {\bf Keywords:} Classification of solutions; Conformally invariant; Systems with mixed order; Exponentially increasing nonlinearity; Method of moving spheres; $\exp^{L}+L\ln L$ inequality; Fractional Laplacians; Nonlocal nonlinearity. \\

{\bf 2020 MSC} Primary: 35M30; Secondary: 35A02, 53C18, 35R11.}

\section{Introduction}

\subsection{Conformally invariant systems with mixed order and exponentially increasing nonlinearity in $\mathbb{R}^{2}$}
In this paper, we are mainly concerned with the following conformally invariant system with mixed order and exponentially increasing nonlinearity in $\mathbb{R}^{2}$:
\begin{equation}\label{PDE}\\\begin{cases}
(-\Delta)^{\frac{1}{2}}u(x)=e^{pv(x)}, \qquad x\in\mathbb{R}^{2}, \\ \\
-\Delta v(x)=u^{4}(x), \qquad x\in\mathbb{R}^{2},
\end{cases}\end{equation}
where $p\in(0,+\infty)$, $u\geq0$ and satisfies the finite total curvature condition $\int_{\mathbb{R}^{2}}u^{4}(x)\mathrm{d}x<+\infty$.

\smallskip

We assume $(u,v)$ is a pair of classical solution to the planar system \eqref{PDE} in the sense that $u\in C^{1,\epsilon}_{loc}(\mathbb{R}^{2})\cap \mathcal{L}_{1}(\mathbb{R}^{2})$ with arbitrarily small $\epsilon>0$ and $v\in C^{2}(\mathbb{R}^{2})$. The square root of the Laplacian $(-\Delta)^{\frac{1}{2}}$ is a particular case of general fractional Laplacians $(-\Delta)^{\frac{\alpha}{2}}$ with $\alpha=1$. In $\mathbb{R}^{n}$ with $n\geq1$, for any $u\in C^{[\alpha],\{\alpha\}+\epsilon}_{loc}(\mathbb{R}^{n})\cap\mathcal{L}_{\alpha}(\mathbb{R}^{n})$, the nonlocal operator $(-\Delta)^{\frac{\alpha}{2}}$ ($0<\alpha<2$) is defined by (see \cite{CT,CG,CLL,CLM,DQ,S})
\begin{equation}\label{nonlocal defn}
  (-\Delta)^{\frac{\alpha}{2}}u(x)=C_{n,\alpha} \, P.V.\int_{\mathbb{R}^n}\frac{u(x)-u(y)}{|x-y|^{n+\alpha}}\mathrm{d}y:=C_{n,\alpha}\lim_{\varepsilon\rightarrow0}\int_{|y-x|\geq\varepsilon}\frac{u(x)-u(y)}{|x-y|^{n+\alpha}}\mathrm{d}y,
\end{equation}
where $[\alpha]$ denotes the integer part of $\alpha$, $\{\alpha\}:=\alpha-[\alpha]$, the constant $C_{n,\alpha}=\left(\int_{\mathbb{R}^{n}} \frac{1-\cos \left(2 \pi \zeta_{1}\right)}{|\zeta|^{n+\alpha}} d \zeta\right)^{-1}$ and the (slowly increasing) function space
\begin{equation}\label{0-1-space}
  \mathcal{L}_{\alpha}(\mathbb{R}^{n}):=\left\{u: \mathbb{R}^{n}\rightarrow\mathbb{R}\,\big|\,\int_{\mathbb{R}^{n}}\frac{|u(x)|}{1+|x|^{n+\alpha}}\mathrm{d}x<+\infty\right\}.
\end{equation}

\smallskip

The fractional Laplacians $(-\Delta)^{\frac{\alpha}{2}}$ can also be defined equivalently (see \cite{CLM}) by Caffarelli and Silvestre's extension method (see \cite{CS}) for $u\in C^{[\alpha],\{\alpha\}+\epsilon}_{loc}(\mathbb{R}^{n})\cap\mathcal{L}_{\alpha}(\mathbb{R}^{n})$. For instance, the square root of the Laplacian $(-\Delta)^{\frac{1}{2}}$ can be defined equivalently for any $u\in C^{1,\epsilon}_{loc}(\mathbb{R}^{n})\cap\mathcal{L}_{1}(\mathbb{R}^{n})$ by
\begin{equation}\label{extension}
  (-\Delta)^{\frac{1}{2}}u(x):=-C_{n}\lim_{y\rightarrow0+}\frac{\partial U(x,y)}{\partial y}
  =-C_{n}\lim_{y\rightarrow0+}\int_{\mathbb{R}^{n}}\frac{|x-\xi|^{2}-ny^{2}}{\big(|x-\xi|^{2}+y^{2}\big)^{\frac{n+3}{2}}}u(\xi)d\xi,
\end{equation}
where $U(x,y)$ is the harmonic extension of $u(x)$ in $\mathbb{R}^{n+1}_{+}=\{(x,y)| \, x\in\mathbb{R}^{n}, \, y\geq0\}$. The definition \eqref{nonlocal defn} of the fractional Laplacian $(-\Delta)^{\frac{\alpha}{2}}$ can also be extended further to distributions in the space $\mathcal{L}_{\alpha}(\mathbb{R}^{n})$ by
\begin{equation}\label{distribution}
  \left\langle(-\Delta)^{\frac{\alpha}{2}}u,\phi\right\rangle=\int_{\mathbb{R}^n}u(x)(-\Delta)^{\frac{\alpha}{2}}\phi(x) \mathrm{d}x, \qquad \forall\phi\in C^{\infty}_0(\mathbb{R}^n).
\end{equation}

\smallskip

Throughout this paper, we define $(-\Delta)^{\frac{1}{2}}u$ by definition \eqref{nonlocal defn} and its equivalent definition \eqref{extension} for $u\in C^{1,\epsilon}_{loc}(\mathbb{R}^{2})\cap\mathcal{L}_{1}(\mathbb{R}^{2})$. Due to the nonlocal virtue of $(-\Delta)^{\frac{1}{2}}$, we need the assumption $u\in C^{1,\epsilon}_{loc}(\mathbb{R}^{2})$ with arbitrarily small $\epsilon>0$ (merely $u\in C^{1}$ is not enough) to guarantee that $(-\Delta)^{\frac{1}{2}}u\in C(\mathbb{R}^{2})$ (see \cite{CLM,S}), and hence $u$ is a classical solution to the planar system \eqref{PDE} in the sense that $(-\Delta)^{\frac{1}{2}}u$ is pointwise well defined and continuous in the whole plane $\mathbb{R}^{2}$.

\smallskip

The fractional Laplacian $(-\Delta)^{\frac{\alpha}{2}}$ is a nonlocal integral-differential operator. It can be used to model diverse physical phenomena, such as anomalous diffusion and quasi-geostrophic flows, turbulence and water waves, molecular dynamics, and relativistic quantum mechanics of stars (see \cite{CV,Co} and the references therein). It also has various applications in conformal geometry, probability and finance (see \cite{Be,CT,CG} and the references therein). In particular, the fractional Laplacian can also be understood as the infinitesimal generator of a stable L\'{e}vy process (see \cite{Be}).

\medskip

Being different from higher dimensions $n\geq3$, one should notice that Liouville type theorem for super-harmonic functions (bounded from below) (cf. e.g. Theorem 3.1 in \cite{Farina}) only holds on the plane $\mathbb{R}^{2}$, hence $\mathbb{R}^{2}$ (endowed with the standard flat metric) is a parabolic Riemannian manifold. Euclidean space $\mathbb{R}^{n}$ with $n\gee 3$ (endowed with the standard flat metric) is not a parabolic Riemannian manifold. Indeed, for any $n\gee 3$, the non-constant positive function $u(x):=\Big(\frac{\sqrt{n(n-2)}}{1+|x|^{2}}\Big)^{\frac{n-2}{2}}$ solves the Yamabe equation $-\Delta u=u^{\frac{n+2}{n-2}}$ in $\mathbb{R}^{n}$. Thus the Liouville type theorem for super-harmonic functions (bounded from below) does not hold in $\mathbb{R}^{n}$ with $n\gee 3$.

\medskip

Consider fractional order or higher order geometrically interesting conformally invariant equation of the form
\begin{equation}\label{GPDE}
  (-\Delta)^{\frac{\alpha}{2}}u=u^{\frac{n+\alpha}{n-\alpha}} \qquad \text{in} \,\, \mathbb{R}^{n},
\end{equation}
where $n\geq1$ and $\alpha\in(0,+\infty)$. We say \eqref{GPDE} has subcritical, critical or super-critical order if $\alpha<n$, $\alpha=n$ or $\alpha>n$ respectively. In the special case $n>\alpha=2$, equation \eqref{GPDE} is the the well-known Yamabe problem. In higher order case that $2\leq\alpha\neq n$ is an even integer, equation \eqref{PDE} arises from the conformal metric problems, prescribing $Q$-curvature problems, conformally covariant Paneitz operators and GJMS operators and so on ... (see e.g. \cite{Branson,CY,CL,CL1,CL2,FKT,Gra,GJMS,Juhl,Lin,Li,N,P,WX,Xu,Z} and the references therein). In the fractional order or fractional higher order case that $\alpha\in(0,n)\setminus 2\mathbb{N}$, conformally invariant equation \eqref{GPDE} is closely related to the fractional $Q$-curvature problems and the study of fractional conformally covariant Paneitz and GJMS operators and so on ... (cf. e.g. \cite{CC,CG,JLX1} and the references therein).

\smallskip

The quantitative and qualitative properties of solutions to conformally invariant equations \eqref{GPDE} have been extensively studied. In the special case $n>\alpha=2$, positive $C^2$ smooth solution to the Yamabe equation \eqref{GPDE} has been classified by Gidas, Ni and Nirenberg in \cite{GNN1}, and Caffarelli, Gidas and Spruck in \cite{CGS}. When $n>\alpha=4$, Lin \cite{Lin} proved the classification results for all the positive $C^{4}$ smooth solutions of \eqref{GPDE}. In \cite{WX}, among other things, Wei and Xu classified all the positive $C^{\alpha}$ smooth solutions of \eqref{GPDE} when $\alpha\in(0,n)$ is an even integer. In \cite{CLO}, by developing the method of moving planes in integral forms, Chen, Li and Ou classified all the positive $L^{\frac{2n}{n-\alpha}}_{loc}$ solutions to the equivalent integral equation of the PDE \eqref{GPDE} for general $\alpha\in(0,n)$, as a consequence, they obtained the classification results for positive weak solutions to PDE \eqref{GPDE}, moreover, they also derived classification results for positive $C^{\alpha}$ smooth solutions to \eqref{GPDE} provided $\alpha\in(0,n)$ is an even integer. Subsequently, Chen, Li and Li \cite{CLL} developed a direct method of moving planes for fractional Laplacians $(-\Delta)^{\frac{\alpha}{2}}$ with $0<\alpha<2$ and classified all the $C^{1,1}_{loc}\cap\mathcal{L}_{\alpha}$ positive solutions to the PDE \eqref{GPDE} directly as an application (see also \cite{CLZ} for the direct method of moving spheres for $(-\Delta)^{\frac{\alpha}{2}}$).

\smallskip

One can observe that all the above mentioned classification results of nonnegative classical solutions are focused on the cases that $0<\alpha<2$ or $2\leq \alpha<n$ is an even integer. When $\alpha=3<n$, Dai and Qin \cite{DQ} derived the first classification result of nonnegative classical solutions to third order equation \eqref{GPDE} under weak integrability assumption. Subsequently, by taking full advantage of the Poisson representation formulae for $(-\Delta)^{\frac{\alpha}{2}}$ and introducing the outer-spherical average associated with $(-\Delta)^{\frac{\alpha}{2}}$, among other things, Cao, Dai and Qin \cite{CDQ0} established super poly-harmonic property of nonnegative solutions and hence classified all nonnegative classical solutions to \eqref{GPDE} for any real number $\alpha\in(0,n)$. In the super-critical order cases, for classification results of positive classical solutions to equation \eqref{GPDE} and related IE with negative exponents, please refer to \cite{Li,N,Xu} and the references therein.

\medskip

In the limiting case (or the so-called critical order case) $n=\alpha=2$, by using the method of moving planes, Chen and Li \cite{CL} classified all the $C^{2}$ smooth solutions with finite total curvature of the equation
\begin{equation}\label{0-1}\\\begin{cases}
-\Delta u(x)=e^{2u(x)},  \qquad  x\in\mathbb{R}^{2}, \\ \\
\int_{\mathbb{R}^{2}}e^{2u(x)}\mathrm{d}x<+\infty.
\end{cases}\end{equation}
They proved that there exists some point $x_{0}\in\mathbb{R}^{2}$ and some $\lambda>0$ such that
$$u(x)=\ln\left[\frac{2\lambda}{1+\lambda^{2}|x-x_{0}|^{2}}\right].$$
Equations of type \eqref{0-1} arise from a variety of situations, such as from prescribing Gaussian curvature in geometry and from combustion theory in physics.

\smallskip

Let us briefly review the geometry background of the planar equation \eqref{0-1}. Let $g_{\mathbf{S}^{2}}$ be the standard metric on the unit $2$-sphere $\mathbf{S}^{2}$. If we consider the conformal metric $\hat{g}:=e^{2w}g_{\mathbf{S}^{2}}$, then the Gaussian curvature $K_{\hat{g}}$ satisfies the following PDE:
\begin{equation}\label{a1}
  \Delta_{g_{\mathbf{S}^{2}}}w+K_{\hat{g}}e^{2w}=1 \qquad \text{on} \,\, \mathbf{S}^{2},
\end{equation}
where $\Delta_{g_{\mathbf{S}^{2}}}$ denotes the Laplace-Beltrami operator with respect to the standard metric $g_{\mathbf{S}^{2}}$ on the sphere $\mathbf{S}^{2}$. In particular, if we take $K_{\hat{g}}\equiv1$ in \eqref{a1}, then from the Cartan-Hadamard theorem, we can deduce that $w=12\ln\left|J\phi\right|$, where $J\phi$ denotes the Jacobian of the transformation $\phi$. That is, $\hat{g}:=e^{2w}g_{\mathbf{S}^{2}}$ is the pull back of the standard metric $g_{\mathbf{S}^{2}}$ through some conformal transformation $\phi$ (i.e., $\hat{g}$ is isometric to $g_{\mathbf{S}^{2}}$). Through the stereographic projection $\pi$ from $\mathbf{S}^{2}$ to $\mathbb{R}^{2}$, one can see that equation \eqref{0-1} on $\mathbb{R}^{2}$ is equivalent to the equation \eqref{a1} on $\mathbf{S}^{2}$ with $K_{\hat{g}}\equiv1$.

\smallskip

In general, suppose $(M,g)$ is a smooth compact $n$-dimensional Riemannian manifold, a metrically defined operator $A$ is said to be conformally covariant if and only if
\begin{equation}\label{a2}
  A_{\hat{g}}(\phi)=e^{-bw}A_{g}(e^{aw}\phi)
\end{equation}
for all functions $\phi\in C^{\infty}(M)$, where $\hat{g}:=e^{2w}g$ is a conformal metric of $g$. When $n=2$, the Laplace-Beltrami operator $\Delta_{g}$ is conformally covariant with $a=0$ and $b=2$. When $n\geq 3$, the conformal Laplace-Beltrami operator $L^{n}_{g}:=-\Delta_{g}+\frac{n-2}{4(n-1)}R_{g}$ is conformally covariant with $a=\frac{n-2}{2}$ and $b=\frac{n+2}{2}$. Suppose $(M,g)$ is a smooth compact $4$-dimensional Riemannian manifold, Paneitz \cite{P} extended the Laplace-Beltrami operator $\Delta_{g}$ to the fourth order operator $P^{4}_{g}$ on $(M,g)$ which is defined by
\begin{equation}\label{P1}
  P^{4}_{g}:=\Delta_{g}^{2}-\delta\left[\left(\frac{2}{3}R_{g}g-2Ric_{g}\right)d\right],
\end{equation}
where $\delta$ is the divergent operator, $R_{g}$ is the scalar curvature of $g$ and $Ric_{g}$ is the Ricci curvature of $g$. The Paneitz operator $P^{4}_{g}$ has the conformally covariant property with $a=0$ and $b=4$. In \cite{Branson}, Branson generalized the Paneitz operator $P^{4}_{g}$ to conformally covariant operator $P^{n}_{g}$ on manifolds $(M,g)$ of other dimensions $n\gee3$ with $a=\frac{n-4}{2}$ and $b=\frac{n+4}{2}$.

\smallskip

On general compact Riemannian manifold $(M,g)$ of dimension $n$, the existence of such a conformally covariant operator $P_{n,g}$ with $a=0$ and $b=n$ for even dimensional manifold (i.e., $n=2m$) was first derived by Graham, Jenne, Mason and Sparling in \cite{GJMS}. The authors in \cite{GJMS} discovered the conformally covariant GJMS operator with the principle part $(-\Delta_{g})^{\frac{n}{2}}$  (see also \cite{Branson,CY,Juhl,N}). However, it is only explicitly known for the Euclidean space $\mathbb{R}^{n}$ with standard metric $g$ and hence for the $n$-sphere $\mathbf{S}^{n}$ with standard metric $g_{\mathbf{S}^{n}}$. The explicit formula for $P_{n,g_{\mathbf{S}^{n}}}$ on $\mathbf{S}^{n}$ with general integer $n\in\mathbb{N}^{+}$ is given by (cf. \cite{Branson,CY,GJMS,Juhl,N}):
\begin{equation}\label{GJMS-2}
P_{n,g_{\mathbf{S}^{n}}}\left(\cdot\right)=\prod_{k=1}^{\frac{n}{2}}\left[-\Delta_{g_{\mathbf{S}^{n}}}+\left(\frac{n}{2}-k\right)\left(\frac{n}{2}+k-1\right)\right]\left(\cdot\right), \qquad \text{if} \,\, n \,\, \text{is even},
\end{equation}
\begin{equation}\label{GJMS-1}
P_{n,g_{\mathbf{S}^{n}}}\left(\cdot\right)=\left[-\Delta_{g_{\mathbf{S}^{n}}}+\frac{(n-1)^{2}}{4}\right]^{\frac{1}{2}}
\prod_{k=1}^{\frac{n-1}{2}}\left[-\Delta_{g_{\mathbf{S}^{n}}}+\frac{(n-1)^{2}}{4}-k^{2}\right]\left(\cdot\right), \qquad \text{if} \,\, n \,\, \text{is odd}.
\end{equation}

\smallskip

On $(\mathbf{S}^{n},g_{\mathbf{S}^{n}})$, if we change the standard metric $g_{\mathbf{S}^{n}}$ to its conformal metric $\hat{g}:=e^{2w}g_{\mathbf{S}^{n}}$ for some smooth function $w$ on the $n$-sphere $\mathbf{S}^{n}$, since the GJMS operator $P_{n,g}$ is conformlly covariant, it turns out that there exists some scalar curvature quantity $Q_{n,g}$ of order $n$ such that
\begin{equation}\label{a3}
  -P_{n,g_{\mathbf{S}^{n}}}(w)+Q_{n,\hat{g}}e^{nw}=Q_{n,g_{\mathbf{S}^{n}}} \qquad \text{on} \,\, \mathbf{S}^{n}.
\end{equation}
When the metric $\hat{g}$ is isometric to the standard metric $g_{\mathbf{S}^{n}}$, then $Q_{n,\hat{g}}=Q_{n,g_{\mathbf{S}^{n}}}=(n-1)!$, and hence \eqref{a3} becomes
\begin{equation}\label{a4}
  -P_{n,g_{\mathbf{S}^{n}}}(w)+(n-1)!e^{nw}=(n-1)! \qquad \text{on} \,\, \mathbf{S}^{n}.
\end{equation}

\smallskip

We reformulate the equation \eqref{a4} on $\mathbb{R}^{n}$ by applying the stereographic projection. Let us denote by $\pi: \, \mathbf{S}^{n}\rightarrow \mathbb{R}^{n}$ the stereographic projection which maps the south pole on $\mathbf{S}^{n}$ to $\infty$. That is, for any $\zeta=(\zeta_{1},\cdots,\zeta_{n+1})\in\mathbf{S}^{n}\subset\mathbb{R}^{n+1}$ and $x=\pi(\zeta)=(x_{1},\cdots,x_{n})\in\mathbb{R}^{n}$, then it holds $\zeta_{k}=\frac{2x_{k}}{1+|x|^{2}}$ for $1\leq k\leq n$ and $\zeta_{n+1}=\frac{1-|x|^{2}}{1+|x|^{2}}$. Suppose $w$ is a smooth function on $\mathbf{S}^{n}$, define the function $u(x):=\phi(x)+w(\zeta)$ for any $x\in\mathbb{R}^{n}$, where $\zeta:=\pi^{-1}(x)$ and $\phi(x):=\ln\left[\frac{2}{1+|x|^{2}}\right]=\ln\left|J_{\pi^{-1}}\right|$. Since the GJMS operator $P_{n,g_{\mathbf{S}^{n}}}$ is the pull back under $\pi$ of the operator $(-\Delta)^{\frac{n}{2}}$ on $\mathbb{R}^{n}$ (see Theorem 3.3 in \cite{Branson1}), $w$ satisfies the equation \eqref{a4} on $\mathbf{S}^{n}$ if and only if the function $u$ satisfies
\begin{equation}\label{a5}
  (-\Delta)^{\frac{n}{2}}u=(n-1)!e^{nu} \qquad \text{in} \,\, \mathbb{R}^{n}.
\end{equation}

\smallskip

In \cite{CY}, for general integer $n$, Chang and Yang classified the $C^{n}$ smooth solutions to the critical order equations \eqref{a5} under decay conditions near infinity
\begin{equation}\label{a0}
  u(x)=\ln\left[\frac{2}{1+|x|^2}\right]+w\left(\zeta(x)\right)
\end{equation}
for some smooth function $w$ defined on $\mathbf{S}^n$. When $n=\alpha=4$, Lin \cite{Lin} proved the classification results for all the $C^{4}$ smooth solutions of
\begin{equation}\label{0-2}\\\begin{cases}
\Delta^{2}u(x)=6e^{4u(x)}, \,\,\,\,\,\,\,\, x\in\mathbb{R}^{4}, \\ \\
\int_{\mathbb{R}^{4}}e^{4u(x)}\mathrm{d}x<+\infty, \,\,\,\,\,\, u(x)=o\left(|x|^{2}\right) \,\,\,\, \text{as} \,\,\,\, |x|\rightarrow+\infty.
\end{cases}\end{equation}
When $n=\alpha$ is an even integer, Wei and Xu \cite{WX} classified the $C^{n}$ smooth solutions of \eqref{a5} with finite total curvature $\int_{\mathbb{R}^{n}}e^{nu(x)}\mathrm{d}x<+\infty$ under the assumption $u(x)=o\left(|x|^{2}\right)$ as $|x|\rightarrow+\infty$. Zhu \cite{Z} classified all the classical solutions with finite total curvature of the problem
\begin{equation}\label{0-4}\\\begin{cases}
(-\Delta)^{\frac{3}{2}}u(x)=2e^{3u(x)}, \,\,\,\,\,\,\,\, x\in\mathbb{R}^{3}, \\ \\
\int_{\mathbb{R}^{3}}e^{3u(x)}\mathrm{d}x<+\infty, \,\,\,\,\,\, u(x)=o(|x|^{2}) \,\,\,\, \text{as} \,\,\,\, |x|\rightarrow+\infty.
\end{cases}\end{equation}
The equation \eqref{0-4} can also be regarded as the following system with mixed order:
\begin{equation}\label{0-4s}\\\begin{cases}
(-\Delta)^{\frac{1}{2}}u(x)=2e^{3v(x)}, \,\,\,\,\,\,\,\, x\in\mathbb{R}^{3}, \\
-\Delta v(x)=u(x), \,\,\,\,\,\,\,\, x\in\mathbb{R}^{3}, \\
\int_{\mathbb{R}^{3}}e^{3v(x)}\mathrm{d}x<+\infty, \,\,\,\,\,\, v(x)=o\left(|x|^{2}\right) \,\,\,\, \text{as} \,\,\,\, |x|\rightarrow+\infty.
\end{cases}\end{equation}
One should note that the planar system \eqref{PDE} has higher degree of nonlinearity than \eqref{0-4s}. Recently, Yu \cite{Yu} classified $(u,v)\in C^{2}(\mathbb{R}^{4})\times C^{4}(\mathbb{R}^{4})$ to the following conformally invariant system
\begin{equation}\label{Yu-s}\\\begin{cases}
-\Delta u(x)=e^{3v(x)}, \,\,\quad  u(x)>0, \qquad  x\in\mathbb{R}^{4}, \\
\Delta^{2}v(x)=u^{4}(x), \qquad\,\,\,  x\in\mathbb{R}^{4}, \\
\int_{\mathbb{R}^{4}}u^{4}(x)\mathrm{d}x<+\infty, \quad \int_{\mathbb{R}^{4}}e^{3v(x)}\mathrm{d}x<+\infty, \,\,\,\,\,\, v(x)=o\left(|x|^{2}\right) \,\,\,\, \text{as} \,\,\,\, |x|\rightarrow+\infty.
\end{cases}\end{equation}
For more literatures on the quantitative and qualitative properties of solutions to fractional order or higher order conformally invariant PDE and IE problems, please refer to \cite{BKN,BF,CT,C,CQ,CDQ,CL,CL2,DQ0,DQ3,DQ4,Fall,Farina,FK,FKT,FLS,JLX1,LZ1,LZ} and the references therein.

\medskip

In this paper, by using the method of moving spheres, we classify all the classical solutions $(u,v)$ to the conformally invariant planar system \eqref{PDE} with mixed order and exponentially increasing nonlinearity. One can observe that, if we assume the relationship $u=e^{\frac{v}{2}}$ on $\mathbb{R}^{2}$ between $u$ and $v$ in the following two conformally invariant equations:
\begin{equation}\label{a6}
  (-\Delta)^{\frac{1}{2}}u=u^{3} \qquad \text{and} \qquad -\Delta v=e^{2v} \qquad \text{in} \,\, \mathbb{R}^{2}
\end{equation}
with the finite total curvature $\int_{\mathbb{R}^{2}}e^{2v(x)}\mathrm{d}x<+\infty$, the resulting system is
\begin{equation}\label{PDE+}\\\begin{cases}
(-\Delta)^{\frac{1}{2}}u(x)=e^{\frac{3}{2}v(x)}, \qquad x\in\mathbb{R}^{2}, \\ \\
-\Delta v(x)=u^{4}(x), \qquad x\in\mathbb{R}^{2}
\end{cases}\end{equation}
with the finite total curvature $\int_{\mathbb{R}^{2}}u^{4}(x)\mathrm{d}x<+\infty$, i.e., system \eqref{PDE} with $p=\frac{3}{2}$. For more literatures on the classification of solutions and Liouville type theorems for various PDE and IE problems via the methods of moving planes or spheres and the method of scaling spheres, please refer to \cite{CGS,CDQ0,CDZ,CQ,CY,CDQ,CL,CL1,CL0,CL2,CLL,CLO,CLZ,DHL,DLQ,DQ,DQ0,DQ3,DQ4,DGZ,DZ,FKT,GNN1,JLX,JLX1,Lin,Li,LZ1,LZ,NN,Pa,Serrin,WX,Xu,Yu,Z} and the references therein.

\smallskip

Our main classification result for system \eqref{PDE} is the following theorem.
\begin{thm}\label{thm0}
Assume $p\in(0,+\infty)$ and $(u,v)$ is a pair of classical solutions to the planar system \eqref{PDE} such that $u\geq0$ and $\int_{\mathbb{R}^{2}}u^{4}(x)\mathrm{d}x<+\infty$. Suppose there exists some $K\gg1$ arbitrarily large such that $u(x)=O\left(|x|^{K}\right)$ as $|x|\rightarrow+\infty$, then $(u,v)$ must take the unique form:
\begin{equation}\label{conclu1}
 u(x)=\left(\frac{6}{p}\right)^{\frac{1}{4}}\left(\frac{\mu}{1+\mu^{2}|x-x_{0}|^{2}}\right)^{\frac{1}{2}} , \qquad v(x)=\frac{3}{2p}\ln\left[\frac{\left(\frac{6}{p}\right)^{\frac{1}{6}}\mu}{1+\mu^2|x-x_{0}|^{2}}\right]
\end{equation}
for some $\mu>0$ and some $x_{0}\in\mathbb{R}^{2}$, and
\begin{equation}\label{conclu2}
  \int_{\mathbb{R}^{2}}u^{4}(x)\mathrm{d}x=\frac{6\pi}{p} \qquad \text{and} \qquad \int_{\mathbb{R}^{2}}e^{pv(x)}\mathrm{d}x=\left(\frac{6}{p}\right)^{\frac{1}{4}}\frac{2\pi}{\sqrt{\mu}}.
\end{equation}
\end{thm}
\begin{rem}\label{rem0}
One should note that, in Theorem \ref{thm0}, we do not need any assumption on $v$. The assumption ``$u(x)=O\left(|x|^{K}\right)$ at $\infty$ for some $K\gg1$ arbitrarily large" is an extremely mild condition, which is much weaker than the condition ``$u$ is bounded from above" and any other assumptions (such as $u(x)=o(|x|^2)$ at $\infty$ and \eqref{a0}) in previous literatures on higher order equations or systems mentioned above. In fact, the necessary condition for us to define $(-\Delta)^{\frac{1}{2}}u$ is $u\in\mathcal{L}_{1}$ (i.e., $\frac{u}{1+|x|^{3}}\in L^{1}(\mathbb{R}^{2})$), which already indicates that $u$ grows slowly and must has strictly less than linear growth at $\infty$ in the sense of integral. In addition to \eqref{conclu2} in Theorem \ref{thm0}, by direct calculations, one can also find that
\begin{equation}\label{a13}
  \int_{\mathbb{R}^{2}}e^{\frac{4}{3}pv(x)}\mathrm{d}x=\left(\frac{6}{p}\right)^{\frac{1}{3}}\pi.
\end{equation}
\end{rem}

\begin{rem}\label{rem4}
For general $s\in(0,1)$, through similar arguments as in the proof of Theorem \ref{thm0}, we can also classify all classical solutions $(u,v)$ to
\begin{equation}\label{gPDE-a}\\\begin{cases}
(-\Delta)^{s}u(x)=e^{pv(x)}, \qquad x\in\mathbb{R}^{2}, \\ \\
-\Delta v(x)=u^{\frac{2}{1-s}}(x), \qquad x\in\mathbb{R}^{2}.
\end{cases}\end{equation}
We focus particularly on the case $s=\frac{1}{2}$ in this paper for the sake of simplicity and leave the classification result for generalized problem \eqref{gPDE-a} to interested readers.
\end{rem}

\smallskip

We would like to mention some key ideas and main ingredients in our proof of Theorem \ref{thm0}.

\smallskip

First, from the finite total curvature condition $\int_{\mathbb{R}^{2}}u^{4}(x)\mathrm{d}x<+\infty$, we can derive the integral representation formula for $u$ (see Lemma \ref{lem0}), that is,
\begin{equation}\label{e1}
  u(x)=\frac{1}{2\pi}\int_{\mathbb{R}^{2}}\frac{1}{|x-y|}e^{pv(y)}\mathrm{d}y,
\end{equation}
and hence $|x|^{-1}e^{pv}\in L^{1}(\mathbb{R}^{2})$. Combining this with the assumptions $u(x)=O\left(|x|^{K}\right)$ at $\infty$ for some $K\gg1$ arbitrarily large and $u^{4}\in L^{1}(\mathbb{R}^{2})$, by elliptic estimates and proving an $\exp^{L}+L\ln L$ inequality, we get $v^{+}=o(|x|^{\delta})$ at $\infty$ for arbitrarily small $\delta>0$ (see Corollary \ref{cor0}) and the asymptotic property $\lim\limits_{|x|\rightarrow+\infty}\frac{\zeta(x)}{\ln|x|}=-\alpha$, where $v^{+}:=\max\{v,0\}$, $\zeta(x):=\frac{1}{2\pi}\int_{\mathbb{R}^{2}}\ln\left[\frac{|y|}{|x-y|}\right]u^{4}(y)\mathrm{d}y$ and $\alpha:=\frac{1}{2\pi}\int_{\mathbb{R}^{2}}u^{4}(x)\mathrm{d}x$. The $\exp^{L}+L\ln L$ inequality can be regarded as a limiting form of Young's inequality or H\"{o}lder inequality, which implies that
\begin{equation}\label{eq-a}
\begin{aligned}
  \int_{\Omega}|f(x)g(x)|\mathrm{d}x &\leq \int_{\Omega}\left(e^{|f(x)|}-|f(x)|-1\right)\mathrm{d}x+\int_{\Omega}|g(x)|\ln\left(|g(x)|+1\right)\mathrm{d}x \\
  &=:\left|f\right|_{\exp^{L}(\Omega)}+\left\|g\right\|_{L\ln L(\Omega)},
\end{aligned}
\end{equation}
where $f\in \exp^{L}(\Omega)$ and $g\in L\ln L(\Omega)$ (see Lemma \ref{limiting inequality}, Remarks \ref{rem1} and \ref{rem2}). By exploiting the $\exp^{L}+L\ln L$ inequality, we get the following estimate on integral with logarithmic singularity:
\begin{equation}\label{eq-b}
  \begin{aligned}
  \int_{B_{1}(x)}\ln\left(\frac{1}{|x-y|}\right)u^{4}(y)\mathrm{d}y&\leq \int_{B_{1}(x)}\frac{1}{|x-y|}\mathrm{d}y+\int_{B_{1}(x)}u^{4}(y)\ln\left(u^{4}(y)+1\right)\mathrm{d}y \\
  &\leq 2\pi+\left[\max_{|y-x|\leq1}\ln\left(u^{4}(y)+1\right)\right]\int_{B_{1}(x)}u^{4}(y)\mathrm{d}y,
  \end{aligned}
\end{equation}
which is crucial in deriving the integral representation formula and asymptotic property for $v$.

Based on these properties, by Liouville type results in a corollary of Lemma 3.3 in Lin \cite{Lin} (see Lemma \ref{lem2} and Corollary \ref{cor1}), we can deduce the integral representation formula for $v$, that is,
\begin{equation}\label{e2}
  v(x)=\frac{1}{2\pi}\int_{\mathbb{R}^{2}}\ln\left[\frac{|y|}{|x-y|}\right]u^{4}(y)\mathrm{d}y+\gamma
\end{equation}
for some constant $\gamma\in\mathbb{R}$, and hence the crucial asymptotic behavior $\lim\limits_{|x|\rightarrow+\infty}\frac{v(x)}{\ln|x|}=-\alpha$ (see Lemma \ref{lem1}). The asymptotic behavior of $v$ implies that $\alpha\geq\frac{1}{p}$, and furthermore, $\beta:=\frac{1}{2\pi}\int_{\mathbb{R}^{2}}e^{pv(x)}\mathrm{d}x<+\infty$ and the asymptotic behavior $\lim\limits_{|x|\rightarrow+\infty}|x|u(x)=\beta$ provided that $\alpha>\frac{2}{p}$ (see Corollary \ref{cor2}).

\smallskip

Next, by making use of these properties, we can apply the method of moving spheres to the IE system for $(u,v)$ consisting of \eqref{e1} and \eqref{e2}. For any $x_{0}\in\mathbb{R}^{2}$, we first prove that, $u_{x_{0},\lambda}\leq u$ and $v_{x_{0},\lambda}\leq v$ in $B_{\lambda}(x_{0})\setminus\{x_{0}\}$ for $\lambda\in(0,+\infty)$ sufficiently large if $\alpha\geq\frac{3}{p}$, $u_{x_{0},\lambda}\geq u$ and $v_{x_{0},\lambda}\geq v$ in $B_{\lambda}(x_{0})\setminus\{x_{0}\}$ for $\lambda\in(0,+\infty)$ sufficiently small if $\frac{1}{p}\leq\alpha\leq\frac{3}{p}$ (see \eqref{m0} for definitions of the Kelvin transforms $u_{x_{0},\lambda}$ and $v_{x_{0},\lambda}$). Then, for any $x_{0}\in\mathbb{R}^{2}$, we show the limiting radius $\lambda_{x_{0}}=0$ if $\alpha>\frac{3}{p}$ and $\lambda_{x_{0}}=+\infty$ if $\alpha<\frac{3}{p}$ (see \eqref{m19} and \eqref{m31} for definitions of the limiting radius $\lambda_{x_{0}}$), and hence derive a contradiction from Lemma 11.2 in \cite{LZ1} (see Lemma \ref{lem10}), the finite total curvature condition and the system \eqref{PDE}. Finally, we must have $\alpha=\frac{3}{p}$ and hence $u_{x_{0},\lambda}\equiv u$ and $v_{x_{0},\lambda}\equiv v$ in $\mathbb{R}^{2}\setminus\{x_{0}\}$ for any $x_{0}\in\mathbb{R}^{2}$ and some $\lambda>0$ depending on $x_{0}$. As a consequence, Lemma 11.1 in \cite{LZ1} (see Lemma \ref{lem10}) and the asymptotic properties of $(u,v)$ yield the desired classification results in Theorem \ref{thm0}.

\medskip

\subsection{Conformally invariant systems with mixed order and Hartree type nonlocal nonlinearity in $\mathbb{R}^{3}$}

We also investigate the following conformally invariant system with mixed order and Hartree type nonlocal nonlinearity in $\mathbb{R}^{3}$:
\begin{equation}\label{PDEH}
\begin{cases}(-\Delta)^{\frac{1}{2}} u(x)=\left(\frac{1}{|x|^{\sigma}}\ast v^{6-\sigma}\right)v^{4-\sigma}(x) & \text { in } \,\, \mathbb{R}^{3}, \\ \\
-\Delta v(x)=u^{\frac{5}{2}}(x) & \text {in} \,\, \mathbb{R}^{3}, \end{cases}
\end{equation}
where $\sigma\in(0,3)$, $u\geq0$ and $v\geq0$. We assume $(u,v)$ is a pair of classical solution to the system \eqref{PDEH} in the sense that $u\in C^{1,\epsilon}_{loc}(\mathbb{R}^{3})\cap \mathcal{L}_{1}(\mathbb{R}^{3})$ with arbitrarily small $\epsilon>0$ and $v\in C^{2}(\mathbb{R}^{3})$.

\smallskip

Consider nonnegative solutions to the following physically interesting static Schr\"{o}dinger-Hartree-Maxwell type equations involving higher-order or higher-order fractional Laplacians
\begin{equation}\label{PDES}
(-\Delta)^{s}u(x)=\left(\frac{1}{|x|^{\sigma}}\ast u^{p}\right)u^{q}(x) \,\,\,\,\,\,\,\,\,\,\,\, \text{in} \,\,\, \mathbb{R}^{n},
\end{equation}
where $n\geq1$, $0<s:=m+\frac{\alpha}{2}<\frac{n}{2}$, $m\geq0$ is an integer, $0<\alpha\leq2$, $0<\sigma<n$, $0<p\leq\frac{2n-\sigma}{n-2s}$ and $0<q\leq\frac{n+2s-\sigma}{n-2s}$. When $\sigma=4s$, $p=2$, $0<q\leq1$, \eqref{PDE} is called static Schr\"{o}dinger-Hartree type equations. When $\sigma=n-2s$, $p=\frac{n+2s}{n-2s}$, $0<q\leq\frac{4s}{n-2s}$, \eqref{PDE} is known as static Schr\"{o}dinger-Maxwell type equations. When $p=\frac{2n-\sigma}{n-2s}$ and $q=\frac{n+2s-\sigma}{n-2s}$, we say equation \eqref{PDES} is conformally invariant or has critical growth (on nonlinearity).

\smallskip

One should observe that both the fractional Laplacians $(-\Delta)^{\frac{\alpha}{2}}$ and the convolution type nonlinearity are nonlocal in equation \eqref{PDES} and system \eqref{PDEH}, which is quite different from most of the known results in previous literature. PDEs of type \eqref{PDES} arise in the Hartree-Fock theory of the nonlinear Schr\"{o}dinger equations (see \cite{LS}). The solution $u$ to problem \eqref{PDES} is also a ground state or a stationary solution to the following dynamic Schr\"{o}dinger-Hartree equation
\begin{equation}\label{Hartree}
i\partial_{t}u+(-\Delta)^{m+\frac{\alpha}{2}}u=\left(\frac{1}{|x|^{\sigma}}\ast |u|^{p}\right)|u|^{q-1}u, \qquad (t,x)\in\mathbb{R}\times\mathbb{R}^{n}
\end{equation}
involving higher-order or higher-order fractional Laplacians. The higher order and fractional order Schr\"{o}dinger-Hartree equations have many interesting applications in the quantum theory of large systems of non-relativistic bosonic atoms and molecules and the theory of laser propagation in medium (see, e.g. \cite{FL,LMZ} and the references therein).

\smallskip

The qualitative properties of solutions to fractional order or higher order Hartree or Choquard type equations have been extensively studied, for instance, see \cite{AGSY,CD,CDZ,CL3,CW,DFHQW,DFQ,DL,DLQ,DQ,Lei,Lieb,Liu,MZ,MS,MS1} and the references therein. In \cite{DLQ}, Dai, Liu and Qin proved the super poly-harmonic properties of nonnegative classical solutions by using the outer-spherical average associated with $(-\Delta)^{\frac{\alpha}{2}}$ and classified all nonnegative classical solutions to Schr\"{o}dinger-Hartree-Maxwell type equations \eqref{PDES} in the full range $s:=m+\frac{\alpha}{2}\in(0,\frac{n}{2})$, $m\geq0$ is an integer, $0<\alpha\leq2$, $0<\sigma<n$, $0<p\leq\frac{2n-\sigma}{n-2s}$ and $0<q\leq\frac{n+2s-\sigma}{n-2s}$. In critical and super-critical order cases (i.e., $\frac{n}{2}\leq s:=m+\frac{\alpha}{2}<+\infty$ and $p,q\in(0,+\infty)$), they also derived Liouville type theorem.

\smallskip

In this paper, by using the method of moving spheres, we classify all the classical solutions $(u,v)$ to the conformally invariant 3D system \eqref{PDE} with mixed order and Hartree type nonlocal nonlinearity. One should observe that, if we assume the relationship $u=v^{2}$ on $\mathbb{R}^{3}$ between the nonnegative solutions $u$ and $v$ in the following two conformally invariant equations:
\begin{equation}\label{a7}
  (-\Delta)^{\frac{1}{2}}u(x)=\left(\frac{1}{|x|^{2}}\ast u^{2}\right)u(x) \qquad \text{and} \qquad -\Delta v(x)=v^{5}(x) \qquad \text{in} \,\, \mathbb{R}^{3},
\end{equation}
the resulting system is
\begin{equation}\label{PDEH+}
\begin{cases}(-\Delta)^{\frac{1}{2}}u(x)=\left(\frac{1}{|x|^{2}}\ast v^{4}\right)v^{2}(x) & \text { in } \,\, \mathbb{R}^{3}, \\ \\
-\Delta v(x)=u^{\frac{5}{2}}(x) & \text { in } \,\, \mathbb{R}^{3}, \end{cases}
\end{equation}
i.e., system \eqref{PDEH} with $\sigma=2$.

\smallskip

Our classification result for system \eqref{PDEH} is the following theorem.
\begin{thm}\label{thm1}
Assume $\sigma\in(0,3)$ and $(u,v)$ is a pair of nonnegative classical solution to the 3D system \eqref{PDEH}. Then we have, either $(u,v)\equiv(0,0)$, or $(u,v)$ must take the unique form:
\begin{equation}\label{conclu-1}
 u(x)=\left[\frac{2\times 3^{2(5-\sigma)}}{I\left(\frac{\sigma}{2}\right)}\right]^{\frac{1}{24-5\sigma}}\frac{\mu}{1+\mu^{2}|x-\bar{x}|^2},
\end{equation}
\begin{equation}\label{conclu-2}
v(x)=\left[\frac{3\times 2^{\frac{5}{2}}}{I\left(\frac{\sigma}{2}\right)^{\frac{5}{2}}}\right]^{\frac{1}{24-5\sigma}}{\left(\frac{\mu}{1+\mu^{2}|x-\bar{x}|^2}\right)}^{\frac{1}{2}}
\end{equation}
for some $\mu>0$ and some $\bar{x}\in\mathbb{R}^{3}$, where $I(\gamma):=\frac{\pi^{\frac{3}{2}}\Gamma\left(\frac{3-2\gamma}{2}\right)}{\Gamma(3-\gamma)}$ for any $0<\gamma<\frac{3}{2}$.
\end{thm}

\medskip

\subsection{Extensions to mixed order conformally invariant systems in $\mathbb{R}^{n}$ with general $n\geq3$}

We also consider the following mixed order conformally invariant system in $\mathbb{R}^{n}$:
\begin{equation}\label{gPDE}\\\begin{cases}
(-\Delta)^{\frac{1}{2}}u(x)=v^{\frac{n+1}{n-2}}(x), \qquad x\in\mathbb{R}^{n}, \\ \\
-\Delta v(x)=u^{\frac{n+2}{n-1}}(x), \qquad x\in\mathbb{R}^{n},
\end{cases}\end{equation}
where $n\geq3$ and $u,v\geq0$. We assume $(u,v)$ is a pair of classical solution to the system \eqref{gPDE} in the sense that $u\in C^{1,\epsilon}_{loc}(\mathbb{R}^{n})\cap \mathcal{L}_{1}(\mathbb{R}^{n})$ with arbitrarily small $\epsilon>0$ and $v\in C^{2}(\mathbb{R}^{n})$.

\medskip

One should observe that, if we assume the relationship $u^{n-2}=v^{n-1}$ in $\mathbb{R}^{n}$ between the nonnegative solutions $u$ and $v$ in the following two conformally invariant equations:
\begin{equation}\label{a20}
  (-\Delta)^{\frac{1}{2}}u(x)=u^{\frac{n+1}{n-1}}(x) \qquad \text{and} \qquad -\Delta v(x)=v^{\frac{n+2}{n-2}}(x) \qquad \text{in} \,\, \mathbb{R}^{n},
\end{equation}
the resulting system is the system \eqref{gPDE}.

\medskip

By applying similar but simpler arguments as in the proof of Theorem \ref{thm1}, we can derive the following classification result for the system \eqref{gPDE}.
\begin{thm}\label{thmg}
Assume $n\geq3$ and $(u,v)$ is a pair of nonnegative classical solution to the system \eqref{gPDE}. Then we have, either $(u,v)\equiv(0,0)$, or $(u,v)$ must take the unique form:
\begin{equation}\label{conclu-1g}
 u(x)=\left[\frac{\left[n(n-1)(n-2)\right]^{\frac{n+1}{3}}}{n-1}\right]^{\frac{n-1}{2n}}{\left(\frac{\mu}{1+\mu^{2}|x-\bar{x}|^2}\right)}^{\frac{n-1}{2}},
\end{equation}
\begin{equation}\label{conclu-2g}
v(x)=\frac{1}{n(n-2)}\left[\frac{\left[n(n-1)(n-2)\right]^{\frac{n+1}{3}}}{n-1}\right]^{\frac{n+2}{2n}}{\left(\frac{\mu}{1+\mu^{2}|x-\bar{x}|^2}\right)}^{\frac{n-2}{2}}
\end{equation}
for some $\mu>0$ and some $\bar{x}\in\mathbb{R}^{n}$.
\end{thm}

\begin{rem}\label{rem3}
The proof of Theorem \ref{thmg} is similar to the proof of Theorem \ref{thm1}, so we leave the details to interested readers.
\end{rem}

The rest of our paper is organized as follows. In section 2, we will carry out our proof of Theorem \ref{thm0}. Section 3 is devoted to proving our Theorem \ref{thm1}.

\smallskip

In what follows, we will use $C$ to denote a general positive constant that may depend on $p$ and $\sigma$, and whose value may differ from line to line.

\section{Proof of Theorem \ref{thm0}}

In this section, we classify all the classical solutions $(u,v)$ to the planar system \eqref{PDE} and hence carry out our proof of Theorem \ref{thm0}.

\medskip

Assume $(u,v)$ is a pair of classical solution to the planar system \eqref{PDE} with $u\geq0$. One should observe that, if $(u,v)$ solve the system \eqref{PDE} for any given $p\in(0,+\infty)$, then $\tilde{u}:=p^{\frac{1}{4}}u$ and $\tilde{v}:=pv+\frac{1}{4}\ln p$ solve \eqref{PDE} with $p=1$. Due to this observation, we will take $p=1$ in \eqref{PDE} hereafter in section 2 in order to make the notation lighter.

\medskip

We first prove the following integral representation formula for $u(x)$.
\begin{lem}\label{lem0}
Assume $\int_{\mathbb{R}^{2}}u^{4}(x)\mathrm{d}x<+\infty$. Then we have, for any $x\in\mathbb{R}^{2}$,
\begin{equation}\label{1}
  u(x)=\frac{1}{2\pi}\int_{\mathbb{R}^{2}}\frac{1}{|x-y|}e^{v(y)}\mathrm{d}y.
\end{equation}
Consequently, $u>0$ in $\mathbb{R}^{2}$, $u(x)\geq\frac{c}{|x|}$ for some constant $c>0$ and $|x|$ large enough, and
\begin{equation}\label{15}
  \int_{\mathbb{R}^{2}}\frac{e^{v(x)}}{|x|}\mathrm{d}x<+\infty.
\end{equation}
\end{lem}
\begin{proof}
For arbitrary $R>0$, let
\begin{equation}\label{2}
\eta_R(x)=\int_{B_R(0)}G_R(x,y)e^{v(y)}\mathrm{d}y,
\end{equation}
where Green's function for $(-\Delta)^{\frac{1}{2}}$ on $B_R(0)\subset\mathbb{R}^{2}$ is given by
\begin{equation}\label{3}
G_R(x,y):=\frac{C_{0}}{|x-y|}\int_{0}^{\frac{t_{R}}{s_{R}}}\frac{1}{(1+b)\sqrt{b}}db
\,\,\,\,\,\,\,\,\, \text{if} \,\, x,y\in B_{R}(0)
\end{equation}
with $s_{R}:=\frac{|x-y|^{2}}{R^{2}}$, $t_{R}:=\left(1-\frac{|x|^{2}}{R^{2}}\right)\left(1-\frac{|y|^{2}}{R^{2}}\right)$, $C_{0}:=\frac{1}{2\pi}\left[\int_{0}^{+\infty}\frac{1}{(1+b)\sqrt{b}}db\right]^{-1}$, and $G_{R}(x,y)=0$ if $x$ or $y\in\mathbb{R}^{2}\setminus B_{R}(0)$ (see \cite{K}). Then, we can derive that $\eta_{R}\in C^{1,\epsilon}_{loc}\left(B_{R}(0)\right)\cap C(\mathbb{R}^{2})\cap\mathcal{L}_{1}(\mathbb{R}^{2})$ satisfies
\begin{equation}\label{4}\\\begin{cases}
(-\Delta)^{\frac{1}{2}}\eta_R(x)=e^{v(x)}, \qquad x\in B_R(0),\\
\eta_R(x)=0,\ \ \ \ \ \ \ x\in \mathbb{R}^{2}\setminus B_R(0).
\end{cases}\end{equation}
Let $w_R(x):=u(x)-\eta_R(x)\in C^{1,\epsilon}_{loc}\left(B_{R}(0)\right)\cap C(\mathbb{R}^{2})\cap\mathcal{L}_{1}(\mathbb{R}^{2})$. By system \eqref{PDE} and \eqref{4}, we have $w_R\in C^{1,\epsilon}_{loc}\left(B_{R}(0)\right)\cap C(\mathbb{R}^{2})\cap\mathcal{L}_{1}(\mathbb{R}^{2})$ and satisfies
\begin{equation}\label{5}\\\begin{cases}
(-\Delta)^{\frac{1}{2}}w_R(x)=0, \qquad x\in B_R(0),\\
w_R(x)\geq0, \qquad x\in \mathbb{R}^{2}\setminus B_R(0).
\end{cases}\end{equation}

Now we need the following maximum principle for fractional Laplacians.
\begin{lem}\label{max}(Maximum principle, \cite{CLL,S})
Let $\Omega$ be a bounded domain in $\mathbb{R}^{n}$, $n\geq2$ and $0<\alpha<2$. Assume that $u\in\mathcal{L}_{\alpha}\cap C^{[\alpha],\{\alpha\}+\epsilon}_{loc}(\Omega)$ with arbitrarily small $\epsilon>0$ and is l.s.c. on $\overline{\Omega}$. If $(-\Delta)^{\frac{\alpha}{2}}u\geq 0$ in $\Omega$ and $u\geq 0$ in $\mathbb{R}^n\setminus\Omega$, then $u\geq 0$ in $\mathbb{R}^n$. Moreover, if $u=0$ at some point in $\Omega$, then $u=0$ a.e. in $\mathbb{R}^{n}$. These conclusions also hold for unbounded domain $\Omega$ if we assume further that
\[\liminf_{|x|\rightarrow\infty}u(x)\geq0.\]
\end{lem}
By Lemma \ref{max}, we deduce from \eqref{5} that for any $R>0$,
\begin{equation}\label{6}
  w_R(x)=u(x)-\eta_{R}(x)\geq0, \qquad \forall \,\, x\in\mathbb{R}^{2}.
\end{equation}
Now, for each fixed $x\in\mathbb{R}^{2}$, letting $R\rightarrow\infty$ in \eqref{6}, we have
\begin{equation}\label{7}
u(x)\geq \frac{1}{2\pi}\int_{\mathbb{R}^{2}}\frac{1}{|x-y|}e^{v(y)}\mathrm{d}y=:\eta(x)>0.
\end{equation}
Taking $x=0$ in \eqref{7}, we get that $v$ satisfies the following integrability
\begin{equation}\label{8}
  \int_{\mathbb{R}^{2}}\frac{e^{v(y)}}{|y|}\mathrm{d}y\leq 2\pi u(0)<+\infty.
\end{equation}
One can observe that $\eta\in C^{1,\epsilon}_{loc}(\mathbb{R}^{2})\cap\mathcal{L}_{1}(\mathbb{R}^{2})$ is a solution of
\begin{equation}\label{9}
(-\Delta)^{\frac{1}{2}}\eta(x)=e^{v(x)},  \qquad \forall \, x\in \mathbb{R}^{2}.
\end{equation}
Define $w(x)=u(x)-\eta(x)$, then by system \eqref{PDE} and \eqref{9}, we have $w\in C^{1,\epsilon}_{loc}(\mathbb{R}^{2})\cap\mathcal{L}_{1}(\mathbb{R}^{2})$ and satisfies
\begin{equation}\label{10}\\\begin{cases}
(-\Delta)^{\frac{1}{2}}w(x)=0, \,\quad \,\,  x\in \mathbb{R}^{2},\\
w(x)\geq0, \,\,\quad \,  x\in \mathbb{R}^{2}.
\end{cases}\end{equation}

Now we need the following Liouville type theorem for $\alpha$-harmonic functions in $\mathbb{R}^{n}$ with $n\geq 2$.
\begin{lem}\label{Liouville}(Liouville theorem, \cite{BKN})
Assume $n\geq2$ and $0<\alpha<2$. Let $u$ be a strong solution of
\begin{equation*}\\\begin{cases}
(-\Delta)^{\frac{\alpha}{2}}u(x)=0, \,\,\,\,\,\,\,\, x\in\mathbb{R}^{n}, \\
u(x)\geq0, \,\,\,\,\,\,\, x\in\mathbb{R}^{n},
\end{cases}\end{equation*}
then $u\equiv C\geq0$.
\end{lem}
For the proof of Lemma \ref{Liouville}, please refer to \cite{BKN}, see also e.g. \cite{CS,Fall}.

From Lemma \ref{Liouville}, we get $w(x)=u(x)-\eta(x)\equiv C\geq0$. Thus, we have proved that
\begin{equation}\label{11}
  u(x)=\frac{1}{2\pi}\int_{\mathbb{R}^{2}}\frac{1}{|x-y|}e^{v(y)}\mathrm{d}y+C>C\geq0.
\end{equation}
Now, by the finite total curvature condition $\int_{\mathbb{R}^{2}}u^{4}(x)\mathrm{d}x<+\infty$, we get
\begin{equation}\label{12}
\int_{\mathbb{R}^{2}}C^{4}\mathrm{d}x<\int_{\mathbb{R}^{2}}u^{4}(x)\mathrm{d}x<+\infty,
\end{equation}
from which we can infer immediately that $C=0$. Therefore, we arrived at
\begin{equation}\label{13}
  u(x)=\frac{1}{2\pi}\int_{\mathbb{R}^{2}}\frac{1}{|x-y|}e^{v(y)}\mathrm{d}y,
\end{equation}
that is, $u$ satisfies the integral equation \eqref{1}.

In addition, from the integral representation formula \eqref{1} for $u$, we get, for any $|x|$ sufficiently large,
\begin{eqnarray}\label{31}
  && u(x)=\frac{1}{2\pi}\int_{\mathbb{R}^{2}}\frac{1}{|x-y|}e^{v(y)}\mathrm{d}y\geq\frac{1}{2\pi}\int_{1\leq |y|<\frac{|x|}{2}}\frac{1}{|x-y|}e^{v(y)}\mathrm{d}y \\
 \nonumber && \qquad \geq \frac{1}{3\pi|x|}\int_{1\leq |y|<\frac{|x|}{2}}\frac{e^{v(y)}}{|y|}\mathrm{d}y\geq \frac{1}{6\pi|x|}\int_{|y|\geq1}\frac{e^{v(y)}}{|y|}\mathrm{d}y=:\frac{c}{|x|}.
\end{eqnarray}
This finishes our proof of Lemma \ref{lem0}.
\end{proof}

From Lemma \ref{lem0}, we can get immediately the following corollary.
\begin{cor}\label{cor0}
Assume $\int_{\mathbb{R}^{2}}u^{4}(x)\mathrm{d}x<+\infty$ and $u=O\left(|x|^{K}\right)$ at $\infty$ for some $K\gg1$ arbitrarily large. Then we have, for any $\delta>0$ small,
\begin{equation}\label{14}
  v^{+}(x)=o\left(|x|^{\delta}\right) \qquad \text{as} \,\,\, |x|\rightarrow+\infty,
\end{equation}
where $v^{+}(x):=\max\{v(x),0\}$ for all $x\in\mathbb{R}^{2}$.
\end{cor}
\begin{proof}
From \eqref{15}, we infer that, for any (bounded or unbounded) domain $\Omega\subseteq\mathbb{R}^{2}$ and $k\in\mathbb{N}^{+}$,
\begin{equation}\label{16}
  \frac{1}{k!}\int_{\Omega}\frac{\left[v^{+}(x)\right]^{k}}{|x|}\mathrm{d}x\leq\int_{\Omega}\frac{e^{v(x)}}{|x|}\mathrm{d}x\leq\int_{\mathbb{R}^{2}}\frac{e^{v(x)}}{|x|}\mathrm{d}x<+\infty.
\end{equation}
For any $x$ such that $|x|$ sufficiently large, by the condition $u=O\left(|x|^{K}\right)$ at $\infty$ for some $K\gg1$ arbitrarily large, \eqref{16} and standard elliptic estimates, we have, for arbitrary $0<\varepsilon<\frac{1}{16K^{2}}$ small and arbitrary $k\in\left(\frac{1}{\sqrt{\varepsilon}},+\infty\right)\cap\mathbb{N}^{+}$ large,
\begin{eqnarray}\label{17}
   && v^{+}(x)\leq \|v^{+}\|_{L^{\infty}\left(B_{\frac{1}{2}}(x)\right)}\leq C\left\{\|v^{+}\|_{L^{1}(B_{1}(x))}+\|u^{4}\|_{L^{1+\varepsilon}(B_{1}(x))}\right\} \\
 \nonumber &&\qquad \,\,\, \leq C\left\{\pi\|v^{+}\|_{L^{k}(B_{1}(x))}
 +\left(\max_{|y-x|\leq1}u(y)\right)^{\frac{4\varepsilon}{1+\varepsilon}}\left(\int_{B_{1}(x)}u^{4}(y)\mathrm{d}y\right)^{\frac{1}{1+\varepsilon}}\right\} \\
 \nonumber &&\qquad\,\,\, \leq C\left\{\left(\int_{B_{1}(x)}\frac{e^{v(y)}}{|y|}\mathrm{d}y\right)^{\frac{1}{k}}|x|^{\frac{1}{k}}
 +O\left(|x|^{\frac{4K\varepsilon}{1+\varepsilon}}\right)\left(\int_{B_{1}(x)}u^{4}(y)\mathrm{d}y\right)^{\frac{1}{1+\varepsilon}}\right\} \\
 \nonumber &&\qquad\,\,\, \leq o\left(|x|^{\frac{1}{k}}\right)+o\left(|x|^{\frac{4K\varepsilon}{1+\varepsilon}}\right)=o\left[|x|^{\sqrt{\varepsilon}}\right],
\end{eqnarray}
where $C$ is a positive constant independent of $x$. This finishes our proof of Corollary \ref{cor0}.
\end{proof}

Since $\int_{\mathbb{R}^{2}}u^{4}(x)\mathrm{d}x<+\infty$, we can define
\begin{equation}\label{0}
  \zeta(x):=\frac{1}{2\pi}\int_{\mathbb{R}^{2}}\ln\left[\frac{|y|}{|x-y|}\right]u^{4}(y)\mathrm{d}y, \qquad \forall \, x\in\mathbb{R}^{2}.
\end{equation}
We can prove the following integral representation formula and asymptotic property for $v$.
\begin{lem}\label{lem1}
Assume $\int_{\mathbb{R}^{2}}u^{4}(x)\mathrm{d}x<+\infty$ and $u=O\left(|x|^{K}\right)$ at $\infty$ for some $K\gg1$ arbitrarily large. Then we have
\begin{equation}\label{21}
  v(x)=\zeta(x)+\gamma:=\frac{1}{2\pi}\int_{\mathbb{R}^{2}}\ln\left[\frac{|y|}{|x-y|}\right]u^{4}(y)\mathrm{d}y+\gamma, \qquad \forall \, x\in\mathbb{R}^{2},
\end{equation}
where $\gamma\in\mathbb{R}$ is a constant. Moreover,
\begin{equation}\label{22}
  \lim_{|x|\rightarrow+\infty}\frac{v(x)}{\ln |x|}=-\alpha,
\end{equation}
where $\alpha:=\frac{1}{2\pi}\int_{\mathbb{R}^{2}}u^{4}(y)\mathrm{d}y\in(0,+\infty)$.
\end{lem}
\begin{proof}
We will first prove the following asymptotic property:
\begin{equation}\label{23}
  \lim_{|x|\rightarrow+\infty}\frac{\zeta(x)}{\ln |x|}=-\alpha.
\end{equation}
To this end, we only need to show that
\begin{equation}\label{24}
  \lim_{|x|\rightarrow+\infty}\int_{\mathbb{R}^{2}}\frac{\ln(|x-y|)-\ln |y|-\ln |x|}{\ln |x|}u^{4}(y)\mathrm{d}y=0.
\end{equation}

\smallskip

We need the following useful $\exp^{L}+L\ln L$ inequality, which is itself of independent interest and can be regarded as a limiting form of Young's inequality or H\"{o}lder inequality.
\begin{lem}[$\exp^{L}+L\ln L$ inequality]\label{limiting inequality}
Assume $n\geq1$ and $\Omega\subseteq\mathbb{R}^{n}$ is a bounded or unbounded domain. Suppose $f\in \exp^{L}(\Omega)$ and $g\in L\ln L(\Omega)$, then we have $fg\in L^{1}(\Omega)$ and
\begin{equation}\label{eL-LlnL}
\begin{aligned}
  \int_{\Omega}|f(x)g(x)|\mathrm{d}x &\leq \int_{\Omega}\left(e^{|f(x)|}-|f(x)|-1\right)\mathrm{d}x+\int_{\Omega}|g(x)|\ln\left(|g(x)|+1\right)\mathrm{d}x \\
  &=:\left|f\right|_{\exp^{L}(\Omega)}+\left\|g\right\|_{L\ln L(\Omega)},
\end{aligned}
\end{equation}
where the spaces $\exp^{L}(\Omega):=\left\{f\mid \, f: \, \Omega\rightarrow\mathbb{C} \,\, \text{measurable}, \, \int_{\Omega}\left(e^{|f(x)|}-|f(x)|-1\right)\mathrm{d}x<+\infty\right\}$ and $L\ln L(\Omega):=\left\{g\mid \, g: \, \Omega\rightarrow\mathbb{C} \,\, \text{measurable}, \, \int_{\Omega}|g(x)|\ln\left(|g(x)|+1\right)\mathrm{d}x<+\infty\right\}$.
\end{lem}
\begin{proof}
In order to prove Lemma \ref{limiting inequality}, we only need to show the following elementary inequality:
\begin{equation}\label{inequality}
  ab\leq \left[e^{a}-a-1\right]+b\ln(b+1), \qquad \forall \, a,\, b\geq0.
\end{equation}
Indeed, since $x=\ln(1+y)$ is the inverse function of $y=e^{x}-1$, one has
\begin{equation}\label{Young-limit}
ab\leq\int_{0}^{a}\left(e^{x}-1\right)\mathrm{d}x+\int_{0}^{b}\ln(1+y)\mathrm{d}y=\left[e^{a}-a-1\right]+\left[\left(b+1\right)\ln\left(b+1\right)-b\right]
\end{equation}
with the equality attained if and only if $e^{a}-1=b$. The inequality \eqref{inequality} now follows immediately from \eqref{Young-limit} and the fact $\ln\left(b+1\right)\leq b$ for any $b\in\mathbb{R}$. This finishes our proof of Lemma \ref{limiting inequality}.
\end{proof}

\begin{rem}\label{rem1}
For $a,b\geq0$, Young's inequality states $ab\leq\frac{a^{p}}{p}+\frac{b^{q}}{q}$ with $\frac{1}{p}+\frac{1}{q}=1$ and $1<p<+\infty$. The inequality \eqref{Young-limit} can be regarded as the limiting form of Young's inequality as $p\rightarrow+\infty$. If $a,b\geq0$, then \eqref{Young-limit} implies $ab\leq\left[e^{a}-a-1\right]+b\ln\left(b+1\right)\leq e^{a}+b\ln b$. If $a,b\geq1$, then \eqref{Young-limit} yields $ab\leq e^{a-1}+b\ln b$ with the equality attained if and only if $e^{a-1}=b$.
\end{rem}

\begin{rem}\label{rem2}
The inequality \eqref{eL-LlnL} in Lemma \ref{limiting inequality} can be regarded as the limiting form of H\"{o}lder's inequality as $q\rightarrow1+$ (i.e., $p\rightarrow+\infty$). In fact, assume $n\geq1$, $0\in\Omega\subseteq\mathbb{R}^{n}$ is a bounded domain and $f: \, \Omega\rightarrow\mathbb{C}$ is measurable. H\"{o}lder's inequality implies $\left(\ln |x|\right)f\in L^{1}(\Omega)$ provided $f\in L^{q}(\Omega)$ for some $1<q\leq+\infty$. However, there are counter-examples in the endpoint case $q=1$. For instance, $f=|x|^{-n}\left[\ln|x|\right]^{-2}\in L^{1}(\Omega)$ but $\left(\ln |x|\right)f\not\in L^{1}(\Omega)$. Nevertheless, if $f\in L\ln L(\Omega)$, i.e., $\int_{\Omega}|f(x)|\ln\left(|f(x)|+1\right)\mathrm{d}x<+\infty$, then Lemma \ref{limiting inequality} implies $\left(\ln |x|\right)f\in L^{1}(\Omega)$. Indeed, by the $\exp^{L}+L\ln L$ type inequality \eqref{eL-LlnL}, we have
\begin{equation}\label{a19}
  \int_{\Omega}\big|\ln |x|f(x)\big|\mathrm{d}x\leq \int_{\Omega}\left(|x|^{-\frac{1}{2}}+|x|^{\frac{1}{2}}\right)\mathrm{d}x+4\int_{\Omega}|f(x)|\ln\left(|f(x)|+1\right)\mathrm{d}x<+\infty.
\end{equation}
\end{rem}

\medskip

By using the $\exp^{L}+L\ln L$ inequality \eqref{eL-LlnL} in Lemma \ref{limiting inequality}, we get
\begin{equation}\label{a18}
  \begin{aligned}
  \int_{B_{1}(x)}\ln\left(\frac{1}{|x-y|}\right)u^{4}(y)\mathrm{d}y&\leq \int_{B_{1}(x)}\frac{1}{|x-y|}\mathrm{d}y+\int_{B_{1}(x)}u^{4}(y)\ln\left(u^{4}(y)+1\right)\mathrm{d}y \\
  &\leq 2\pi+\left[\max_{|y-x|\leq1}\ln\left(u^{4}(y)+1\right)\right]\int_{B_{1}(x)}u^{4}(y)\mathrm{d}y.
  \end{aligned}
\end{equation}
As a consequence, by \eqref{a18}, the conditions $\int_{\mathbb{R}^{2}}u^{4}(x)\mathrm{d}x<+\infty$ and $u=O\left(|x|^{K}\right)$ at $\infty$ for some $K\gg1$ arbitrarily large, we have, for any $|x|\geq e^{2}$ large enough,
\begin{eqnarray}\label{25}
   && \quad \left|\int_{\mathbb{R}^{2}}\frac{\ln(|x-y|)-\ln |y|-\ln |x|}{\ln |x|}u^{4}(y)\mathrm{d}y\right| \\
  \nonumber &&\leq 3\int_{B_{1}(x)}u^{4}(y)\mathrm{d}y+\frac{2\pi}{\ln|x|}+\frac{O\left(4K\ln|x|\right)}{\ln|x|}\int_{B_{1}(x)}u^{4}(y)\mathrm{d}y \\
 \nonumber && \quad +\frac{\max\limits_{|y|\leq\ln |x|}\left|\ln\left(\frac{|x-y|}{|x|}\right)\right|}{\ln |x|}\int_{|y|<\ln |x|}u^{4}(y)\mathrm{d}y+\frac{1}{\ln |x|}\int_{|y|<\ln |x|}\left|\ln |y|\right|u^{4}(y)\mathrm{d}y \\
 \nonumber && \quad +\sup_{\substack{|y-x|\geq 1 \\ |y|\geq\ln |x|}}\frac{\left|\ln(|x-y|)-\ln |y|-\ln |x|\right|}{\ln |x|}\int_{|y|\geq\ln |x|}u^{4}(y)\mathrm{d}y \\
 \nonumber &&\leq o_{|x|}(1)+\frac{2\pi}{\ln|x|}+\frac{\ln 2}{\ln |x|}\int_{\mathbb{R}^{2}}u^{4}(x)\mathrm{d}x+\frac{1}{\ln |x|}\int_{|y|<1}\ln\left(\frac{1}{|y|}\right)u^{4}(y)\mathrm{d}y \\
 \nonumber &&\quad +\frac{\ln\left(\ln |x|\right)}{\ln |x|}\int_{\mathbb{R}^{2}}u^{4}(y)\mathrm{d}y+\left(2+\frac{\ln 2}{\ln |x|}\right)\int_{|y|\geq\ln |x|}u^{4}(y)\mathrm{d}y
 =o_{|x|}(1),
\end{eqnarray}
where we have use the fact $1>\frac{1}{|x|}+\frac{1}{|y|}\geq\frac{|x-y|}{|x|\cdot|y|}\geq\frac{1}{2|x|^{2}}$ for any $|y-x|\geq 1$ and $|y|\geq\ln |x|$. By letting $|x|\rightarrow+\infty$ in \eqref{25}, we obtain
\begin{equation}\label{26}
  \lim_{|x|\rightarrow+\infty}\int_{\mathbb{R}^{2}}\frac{\ln(|x-y|)-\ln |y|-\ln |x|}{\ln |x|}u^{4}(y)\mathrm{d}y=0,
\end{equation}
and hence \eqref{23} holds.

\smallskip

Next, we aim to show \eqref{21}. We need the following Lemma from Lin \cite{Lin}.
\begin{lem}[Lemma 3.3 in \cite{Lin}]\label{lem2}
Assume $n\geq2$. Suppose that $w$ is a harmonic function in $\mathbb{R}^{n}$ such that $e^{w-c|x|^{2}}\in L^{1}(\mathbb{R}^{n})$ for some constant $c>0$. Then $w$ is a polynomial of degree at most $2$.
\end{lem}

From Lemma \ref{lem2}, we can derive the following corollary.
\begin{cor}\label{cor1}
Assume $n\geq2$. Suppose that $w$ is a harmonic function in $\mathbb{R}^{n}$. Then we have \\
(i) If $w^{+}=O(|x|^{2})$ at $\infty$, then $w$ is a polynomial of degree at most $2$. \\
(ii) If $w^{+}=o(|x|^{2})$ at $\infty$, then $w$ is a polynomial of degree at most $1$. \\
(iii) If $w^{+}=o(|x|)$ at $\infty$, then $w\equiv C$ in $\mathbb{R}^{n}$ for some constant $C$.
\end{cor}
\begin{proof}
Conclusion (i) is a direct consequence of Lemma \ref{lem2}, we omit the details. For (ii), if $w^{+}=o(|x|^{2})$ at $\infty$, then $w$ must take the form
\begin{equation}\label{29}
  w(x)=a_{0}+\sum_{k=1}^{n}a_{k}x_{k}+\sum_{k=1}^{n}b_{k}x_{k}^{2},
\end{equation}
where $b_{k}\leq0$. If there is some $b_{k_{0}}<0$, then $\Delta w(x)\leq 2b_{k_{0}}<0$, which is absurd since $w$ is harmonic. Thus we must have $b_{k}=0$ for every $k=1,\cdots,n$ and hence $w$ is a polynomial of order at most $1$. As to (iii), if $w^{+}=o(|x|)$ at $\infty$, then we also have $a_{k}=0$ for every $k=1,\cdots,n$ and hence $w\equiv a_{0}$ in $\mathbb{R}^{n}$. This finishes our proof of Corollary \ref{cor1}.
\end{proof}

Note that
\begin{equation}\label{28}
  -\Delta\left(v-\zeta\right)(x)=0, \qquad \forall \, x\in\mathbb{R}^{2}.
\end{equation}
From \eqref{23}, we infer that $\zeta=O(\ln |x|)$ at $\infty$. Since Corollary \ref{cor0} implies that $v^{+}=o\left(|x|^{\delta}\right)$ at $\infty$ for arbitrary $\delta>0$ small, Corollary \ref{cor1} (iii) yields immediately that, for some constant $\gamma\in\mathbb{R}$,
\begin{equation}\label{27}
  v(x)-\zeta(x)\equiv\gamma, \qquad \forall \, x\in\mathbb{R}^{2}.
\end{equation}
Thus the integral representation formula \eqref{21} for $v$ holds. The asymptotic property \eqref{22} follows immediately from \eqref{21} and \eqref{23}. This completes our proof of Lemma \ref{lem1}.
\end{proof}

As a consequence of Lemma \ref{lem1}, we have the following corollary.
\begin{cor}\label{cor2}
Assume $\int_{\mathbb{R}^{2}}u^{4}(x)\mathrm{d}x<+\infty$ and $u=O\left(|x|^{K}\right)$ at $\infty$ for some $K\gg1$ arbitrarily large. Then we have, for arbitrarily small $\delta>0$,
\begin{equation}\label{18}
  \lim_{|x|\rightarrow+\infty}\frac{e^{v(x)}}{|x|^{-\alpha-\delta}}=+\infty \qquad \text{and} \qquad \lim_{|x|\rightarrow+\infty}\frac{e^{v(x)}}{|x|^{-\alpha+\delta}}=0.
\end{equation}
Consequently,
\begin{equation}\label{19}
  \alpha:=\frac{1}{2\pi}\int_{\mathbb{R}^{2}}u^{4}(x)\mathrm{d}x\geq1.
\end{equation}
Moreover, if $\alpha>2$, then
\begin{equation}\label{20}
  \beta:=\frac{1}{2\pi}\int_{\mathbb{R}^{2}}e^{v(x)}\mathrm{d}x<+\infty,
\end{equation}
\begin{equation}\label{33}
  \lim_{|x|\rightarrow+\infty}|x|u(x)=\beta.
\end{equation}
\end{cor}
\begin{proof}
The asymptotic property \eqref{22} in Lemma \ref{lem1} implies that $v(x)=-\alpha\ln |x|+o(\ln |x|)$ at $\infty$. Therefore, as $|x|\rightarrow+\infty$,
\begin{equation}\label{30}
  e^{v(x)}=e^{-\alpha\ln |x|+o(\ln |x|)}=|x|^{-\alpha}e^{o(\ln |x|)}.
\end{equation}
Therefore, we have, for arbitrarily small $\delta>0$,
\begin{equation}\label{32}
  \lim_{|x|\rightarrow+\infty}\frac{e^{v(x)}}{|x|^{-\alpha-\delta}}=+\infty \qquad \text{and} \qquad \lim_{|x|\rightarrow+\infty}\frac{e^{v(x)}}{|x|^{-\alpha+\delta}}=0.
\end{equation}
By \eqref{32}, one can easily verify that the integrability $\int_{\mathbb{R}^{2}}\frac{e^{v(x)}}{|x|}\mathrm{d}x<+\infty$ derived in \eqref{15} in Lemma \ref{lem0} implies that $\alpha\geq 1$.

If we assume $\alpha>2$, it follows immediately from the asymptotic property \eqref{18} that $\beta:=\frac{1}{2\pi}\int_{\mathbb{R}^{2}}e^{v(x)}\mathrm{d}x<+\infty$. Take $\delta:=\frac{\alpha-2}{2}$, then \eqref{18} implies that, there exists a $R_{0}\geq1$ large enough such that
\begin{equation}\label{35}
  e^{v(x)}\leq \frac{1}{|x|^{\frac{\alpha+2}{2}}}, \qquad \forall \,\, |x|\geq R_{0}.
\end{equation}
In order to prove \eqref{33}, by the integral representation formula \eqref{1} for $u$, we only need to show
\begin{equation}\label{34}
  \lim_{|x|\rightarrow+\infty}\int_{\mathbb{R}^{2}}\frac{|x|-|x-y|}{|x-y|}e^{v(y)}\mathrm{d}y=0.
\end{equation}
Indeed, by \eqref{20} and \eqref{35}, we have, for any $|x|>2R_{0}$ sufficiently large,
\begin{eqnarray}\label{36}
  &&\quad \left|\int_{\mathbb{R}^{2}}\frac{|x|-|x-y|}{|x-y|}e^{v(y)}\mathrm{d}y\right| \\
 \nonumber &&\leq \int_{|y-x|<\frac{|x|}{2}}\frac{1}{|x-y||y|^{\frac{\alpha}{2}}}\mathrm{d}y+3\int_{\substack{|y-x|\geq\frac{|x|}{2} \\ |y|\geq\frac{|x|}{2}}}e^{v(y)}\mathrm{d}y\\
 \nonumber && \quad +\frac{2}{|x|}\int_{|y|<R_{0}}|y|e^{v(y)}\mathrm{d}y+\frac{2}{|x|}\int_{R_{0}\leq|y|<\frac{|x|}{2}}\frac{1}{|y|^{\frac{\alpha}{2}}}\mathrm{d}y \\
 \nonumber &&\leq \frac{2^{\frac{\alpha}{2}}\pi}{|x|^{\frac{\alpha-2}{2}}}+o_{|x|}(1)+\frac{2}{|x|}\int_{|y|<R_{0}}|y|e^{v(y)}\mathrm{d}y+\frac{4\pi}{|x|}\tau(x)=o_{|x|}(1),
\end{eqnarray}
where $\tau(x):=\frac{2}{4-\alpha}\left(\frac{|x|}{2}\right)^{\frac{4-\alpha}{2}}$ if $2<\alpha<4$, $\tau(x):=\ln\left(\frac{|x|}{2}\right)$ if $\alpha=4$ and $\tau(x):=\frac{2}{\alpha-4}R_{0}^{\frac{4-\alpha}{2}}$ if $\alpha>4$. Hence the asymptotic property \eqref{33} holds. This completes our proof of Corollary \ref{cor2}.
\end{proof}

\medskip

We have proved that classical solution $(u,v)$ to the PDE system \eqref{PDE} also solves the following IE system:
\begin{equation}\label{IE}\\\begin{cases}
u(x)=\frac{1}{2\pi}\int_{\mathbb{R}^{2}}\frac{1}{|x-y|}e^{v(y)}\mathrm{d}y, \qquad x\in\mathbb{R}^{2}, \\ \\
v(x)=\frac{1}{2\pi}\int_{\mathbb{R}^{2}}\ln\Big[\frac{|y|}{|x-y|}\Big]u^{4}(y)\mathrm{d}y+\gamma, \qquad x\in\mathbb{R}^{2},
\end{cases}\end{equation}
where $\gamma\in\mathbb{R}$ is a constant. Next, we will apply the method of moving spheres to show that $\alpha=3$ and derive the classification of $(u,v)$.

To this end, for arbitrarily given $x_{0}\in\mathbb{R}^{2}$ and any $\lambda>0$, we define the Kelvin transforms of $(u,v)$ centered at $x_{0}$ by
\begin{equation}\label{m0}
u_{x_{0},\lambda}(x):=\frac{\lambda}{|x-x_{0}|}u\left(x^{x_{0},\lambda}\right), \qquad v_{x_{0},\lambda}(x)=v\left(x^{x_{0},\lambda}\right)+3\ln\frac{\lambda}{|x-x_{0}|}
\end{equation}
for arbitrary $x\in\mathbb{R}^{2}\setminus\{x_{0}\}$, where $x^{x_{0},\lambda}:=\frac{\lambda^{2}(x-x_{0})}{|x-x_{0}|^{2}}+x_{0}$.

Now, we will carry out the the method of moving spheres to the IE system \eqref{IE} with respect to arbitrarily given point $x_{0}\in\mathbb{R}^{2}$. Let $\lambda>0$ be an arbitrary positive real number. Define $w^{u}_{x_{0},\lambda}(x):=u_{x_{0},\lambda}(x)-u(x)$ and $w^{v}_{x_{0},\lambda}(x):=v_{x_{0},\lambda}(x)-v(x)$ for any $x\in\mathbb{R}^{2}\setminus\{x_{0}\}$. We start moving the sphere $S_{\lambda}(x_{0}):=\{x\in\mathbb{R}^{2}\mid \, |x-x_{0}|=\lambda\}$ from near $\lambda=0$ or $\lambda=+\infty$, until its limiting position. Therefore, the moving sphere process can be divided into two steps.

\medskip

In what follows, the two different cases $\alpha\geq3$ and $1\leq \alpha\leq3$ will be discussed separately.

\medskip

\emph{Step 1}. Start moving the circle $S_{\lambda}(x_{0})$ from near $\lambda=0$ or $\lambda=+\infty$.

\medskip

\noindent \emph{Case (i)} $\alpha\geq3$. We will show that, for $\lambda>0$ sufficiently large,
\begin{equation}\label{m1}
w^{u}_{x_{0},\lambda}(x)\leq 0 \quad \text{and} \quad w^{v}_{x_{0},\lambda}(x)\leq 0, \qquad \forall x \in B_{\lambda}(x_{0})\setminus\{x_{0}\}.
\end{equation}
That is, we start moving the sphere $S_{\lambda}(x_{0}):=\{x\in\mathbb{R}^{2}\mid \, |x-x_{0}|=\lambda\}$ from near $\lambda=+\infty$ towards the point $x_{0}$ such that \eqref{m1} holds.

Define
\begin{equation}\label{m2}
B_{\lambda,u}^{+}(x_{0}):=\left\{x \in B_{\lambda}(x_{0})\setminus\{x_{0}\} \mid w^{u}_{x_{0},\lambda}(x)>0\right\},
\end{equation}
\begin{equation}\label{m2'}
B_{\lambda,v}^{+}(x_{0}):=\left\{x \in B_{\lambda}(x_{0})\setminus\{x_{0}\} \mid w^{v}_{x_{0},\lambda}(x)>0\right\}.
\end{equation}
We will show that, for $\lambda>0$ sufficiently large,
\begin{equation}\label{m3}
B_{\lambda,u}^{+}(x_{0})=B_{\lambda,v}^{+}(x_{0})=\emptyset.
\end{equation}

Since $(u,v)$ solves the IE system \eqref{IE}, through direct calculations, we get, for any $\lambda>0$ and all $x\in\mathbb{R}^{2}$,
\begin{equation}\label{m4}
\begin{aligned}
u(x)=&\frac{1}{2\pi}\int_{B_{\lambda}(x_{0})}\frac{e^{v(y)}}{|x-y|}\mathrm{d}y
+\frac{1}{2\pi}\int_{B_{\lambda}(x_{0})}\frac{e^{v(y^{x_{0},\lambda})}}{\left|x-y^{x_{0},\lambda}\right|}\left(\frac{\lambda}{|y-x_{0}|}\right)^{4}\mathrm{d}y \\
=&\frac{1}{2\pi}\int_{B_{\lambda}(x_{0})}\frac{e^{v(y)}}{|x-y|}\mathrm{d}y
+\frac{1}{2\pi}\int_{B_{\lambda}(x_{0})}\frac{e^{v_{x_{0},\lambda}(y)}}{\left|\frac{\lambda(x-x_{0})}{|x-x_{0}|}-\frac{|x-x_{0}|}{\lambda}(y-x_{0})\right|}\mathrm{d}y,
\end{aligned}
\end{equation}
\begin{equation}\label{m4'}
\begin{aligned}
v(x)=& \frac{1}{2\pi}\int_{B_{\lambda}(x_{0})} \ln\left[\frac{|y|}{|x-y|}\right]u^{4}(y)\mathrm{d}y \\
&+\frac{1}{2\pi}\int_{B_{\lambda}(x_{0})} \ln\left[\frac{|y^{x_{0},\lambda}|}{\left|x-y^{x_{0},\lambda}\right|}\right]u^{4}(y^{x_{0},\lambda})\left(\frac{\lambda}{|y-x_{0}|}\right)^{4}\mathrm{d}y+\gamma \\
=& \frac{1}{2\pi}\int_{B_{\lambda}(x_{0})} \ln\left[\frac{|y|}{|x-y|}\right]u^{4}(y)\mathrm{d}y
+\frac{1}{2\pi}\int_{B_{\lambda}(x_{0})} \ln\left[\frac{|y^{x_{0},\lambda}|}{\left|x-y^{x_{0},\lambda}\right|}\right]u_{x_{0},\lambda}^{4}(y)\mathrm{d}y+\gamma.
\end{aligned}
\end{equation}
Consequently, we deduce from \eqref{m4} and \eqref{m4'} that, for arbitrarily given $x_{0}\in \mathbb{R}^{2}$, any $\lambda>0$ and all $x\in\mathbb{R}^{2}\setminus\{x_{0}\}$,
\begin{equation}\label{m5}
\begin{aligned}
u_{x_{0},\lambda}(x)=&\frac{1}{2\pi}\frac{\lambda}{|x-x_{0}|}\int_{B_{\lambda}(x_{0})}\frac{e^{v(y)}}{|x^{x_{0},\lambda}-y|}\mathrm{d}y \\
&+\frac{1}{2\pi}\frac{\lambda}{|x-x_{0}|}\int_{B_{\lambda}(x_{0})}\frac{e^{v(y^{x_{0},\lambda})}}{\left|x^{x_{0},\lambda}-y^{x_{0},\lambda}\right|}\left(\frac{\lambda}{|y-x_{0}|}\right)^{4}\mathrm{d}y \\
=&\frac{1}{2\pi}\int_{B_{\lambda}(x_{0})}\frac{e^{v(y)}}{\left|\frac{\lambda(x-x_{0})}{|x-x_{0}|}-\frac{|x-x_{0}|}{\lambda}(y-x_{0})\right|}\mathrm{d}y
+\frac{1}{2\pi}\int_{B_{\lambda}(x_{0})}\frac{e^{v_{x_{0},\lambda}}(y)}{|x-y|}\mathrm{d}y,
\end{aligned}
\end{equation}
\begin{equation}\label{m5'}
\begin{aligned}
v_{x_{0},\lambda}(x)=& \frac{1}{2\pi}\int_{B_{\lambda}(x_{0})} \ln\left[\frac{|y|}{|x^{x_{0},\lambda}-y|}\right]u^{4}(y)\mathrm{d}y \\
& +\frac{1}{2\pi}\int_{B_{\lambda}(x_{0})} \ln\left[\frac{|y^{x_{0},\lambda}|}{\left|x^{x_{0},\lambda}-y^{x_{0},\lambda}\right|}\right]u_{x_{0},\lambda}^{4}(y)\mathrm{d}y+\gamma+3\ln\frac{\lambda}{|x-x_{0}|}.
\end{aligned}
\end{equation}
It follows from \eqref{m4}, \eqref{m4'}, \eqref{m5} and \eqref{m5'} that, for any $x\in B_{\lambda}(x_{0})\setminus\{x_{0}\}$,
\begin{equation}\label{m6}
\begin{aligned}
& w^{u}_{x_{0},\lambda}(x)=u_{x_{0},\lambda}(x)-u(x) \\
=&\frac{1}{2\pi}\int_{B_{\lambda}(x_{0})}\left[\frac{1}{|x-y|}-\frac{1}{\left|\frac{\lambda(x-x_{0})}{|x-x_{0}|}-\frac{|x-x_{0}|}{\lambda}(y-x_{0})\right|}\right]
\left(e^{v_{x_{0},\lambda}(y)}-e^{v(y)}\right)\mathrm{d}y,
\end{aligned}
\end{equation}
\begin{equation}\label{m6'}
\begin{aligned}
w^{v}_{x_{0},\lambda}(x)=&v_{x_{0},\lambda}(x)-v(x)=\frac{1}{2\pi}\int_{B_{\lambda}(x_{0})} \ln\left[\frac{|x-y|}{|x^{x_{0},\lambda}-y|}\right]u^{4}(y)\mathrm{d}y \\
&+\frac{1}{2\pi}\int_{B_{\lambda}(x_{0})} \ln\left[\frac{\left|x-y^{x_{0},\lambda}\right|}{\left|x^{x_{0},\lambda}-y^{x_{0},\lambda}\right|}\right]u_{x_{0},\lambda}^{4}(y)\mathrm{d}y
+3\ln\frac{\lambda}{|x-x_{0}|} \\
=&\frac{1}{2\pi}\int_{B_{\lambda}(x_{0})}\ln\left[\frac{\left|\frac{\lambda(x-x_{0})}{|x-x_{0}|}-\frac{|x-x_{0}|}{\lambda}(y-x_{0})\right|}{\left|x-y\right|}\right]
\left(u_{x_{0},\lambda}^{4}(y)-u^{4}(y)\right)\mathrm{d}y \\
& +\left(3-\alpha\right)\ln\frac{\lambda}{|x-x_{0}|},
\end{aligned}
\end{equation}
where $\alpha:=\frac{1}{2\pi}\int_{\mathbb{R}^{2}}u^{4}(x)\mathrm{d}x\in\left[1,+\infty\right)$.

From \eqref{m6} and the mean value theorem, one can derive that, for any $x\in B^{+}_{\lambda,u}(x_{0})$,
\begin{equation}\label{m7}
\begin{aligned}
0&<w^{u}_{x_{0},\lambda}(x)=u_{x_{0},\lambda}(x)-u(x) \\
&\leq \frac{1}{2\pi}\int_{B^{+}_{\lambda,v}(x_{0})}\left[\frac{1}{|x-y|}-\frac{1}{\left|\frac{\lambda(x-x_{0})}{|x-x_{0}|}-\frac{|x-x_{0}|}{\lambda}(y-x_{0})\right|}\right]
\left(e^{v_{x_{0},\lambda}(y)}-e^{v(y)}\right)\mathrm{d}y \\
&\leq \frac{1}{2\pi}\int_{B^{+}_{\lambda,v}(x_{0})}\frac{1}{|x-y|}e^{\xi_{x_{0},\lambda}(y)}w^{v}_{x_{0},\lambda}(y)\mathrm{d}y \\
&\leq \frac{1}{2\pi}\int_{B^{+}_{\lambda,v}(x_{0})}\frac{e^{v_{x_{0},\lambda}(y)}}{|x-y|}w^{v}_{x_{0},\lambda}(y)\mathrm{d}y,
\end{aligned}
\end{equation}
where $v(y)<\xi_{x_{0},\lambda}(y)<v_{x_{0},\lambda}(y)$ for any $y\in B^{+}_{\lambda,v}(x_{0})$. Here we have used the fact that $|x-y|<\left|\frac{\lambda (x-x_0)}{|x-x_0|}-\frac{|x-x_0|(y-x_0)}{\lambda}\right|$ for any $x,y\in B_\lambda(x_0)\setminus\{x_{0}\}$.

\smallskip

By direct calculations, one can obtain that, for any $x, y\in B_{\lambda}(x_{0})\setminus\{x_{0}\}$ and $x\neq y$,
\begin{equation}\label{m8}
\begin{aligned}
0<2\ln\left[\frac{\left|\frac{\lambda(x-x_{0})}{|x-x_{0}|}-\frac{|x-x_{0}|}{\lambda}(y-x_{0})\right|}{\left|x-y\right|}\right] &=\ln \left[1+\frac{\left(\lambda-\frac{|x-x_{0}|^{2}}{\lambda}\right)\left(\lambda-\frac{|y-x_{0}|^{2}}{\lambda}\right)}{|x-y|^{2}}\right] \\
& \leq \ln \left(1+\frac{\lambda^{2}}{|x-y|^{2}}\right).
\end{aligned}
\end{equation}
It is obvious that, for arbitrary $\varepsilon>0$,
\begin{equation}\label{m9}
\ln (1+t)=o\left(t^{\varepsilon}\right) \,\,\quad\,\, \text { as } \, t \rightarrow+\infty.
\end{equation}
This implies, for any given $\varepsilon>0$, there exists a $\delta(\varepsilon)>0$ such that
\begin{equation}\label{m10}
\ln (1+t) \leq t^{\varepsilon}, \qquad \forall \, t>\frac{1}{\delta(\varepsilon)^{2}}.
\end{equation}
Therefore, from \eqref{m8} and \eqref{m10}, we deduce that
\begin{equation}\label{m11}
\ln\left[\frac{\left|\frac{\lambda(x-x_{0})}{|x-x_{0}|}-\frac{|x-x_{0}|}{\lambda}(y-x_{0})\right|}{\left|x-y\right|}\right]
\leq \frac{1}{2}\frac{\lambda^{2 \varepsilon}}{|x-y|^{2 \varepsilon}}, \qquad \forall \, 0<|x|,|y|<\lambda, \,\,|x-y|<\lambda \delta(\varepsilon),
\end{equation}
\begin{equation}\label{m12}
\ln\left[\frac{\left|\frac{\lambda(x-x_{0})}{|x-x_{0}|}-\frac{|x-x_{0}|}{\lambda}(y-x_{0})\right|}{\left|x-y\right|}\right]
\leq \frac{\ln \left[1+\frac{1}{\delta(\varepsilon)^{2}}\right]}{2}, \quad \forall \, 0<|x|,|y|<\lambda, \,\, |x-y| \geq \lambda \delta(\varepsilon).
\end{equation}
Due to $\alpha\geq3$, from \eqref{m6'}, \eqref{m11}, \eqref{m12} and the mean value theorem, one can derive that, for any given $\varepsilon>0$ small and all $x\in B^{+}_{\lambda,v}(x_{0})$,
\begin{equation}\label{m13}
\begin{aligned}
0&<w^{v}_{x_{0},\lambda}(x)=v_{x_{0},\lambda}(x)-v(x) \\
&\leq\frac{1}{2\pi}\int_{B^{+}_{\lambda,u}(x_{0})}\ln\left[\frac{\left|\frac{\lambda(x-x_{0})}{|x-x_{0}|}-\frac{|x-x_{0}|}{\lambda}(y-x_{0})\right|}{\left|x-y\right|}\right]
\left(u_{x_{0},\lambda}^{4}(y)-u^{4}(y)\right)\mathrm{d}y \\
&\leq \frac{1}{\pi}\int_{B^{+}_{\lambda,u}(x_{0})\cap B_{\lambda \delta(\varepsilon)}(x)} \frac{\lambda^{2 \varepsilon}}{|x-y|^{2 \varepsilon}}\eta_{x_{0},\lambda}^{3}(y)w^{u}_{x_{0},\lambda}(y)\mathrm{d}y \\
&\quad +\frac{\ln \left[1+\frac{1}{\delta(\varepsilon)^{2}}\right]}{\pi}\int_{B^{+}_{\lambda,u}(x_{0})\setminus B_{\lambda \delta(\varepsilon)}(x)}\eta_{x_{0},\lambda}^{3}(y)w^{u}_{x_{0},\lambda}(y)\mathrm{d}y \\
&\leq \frac{1}{\pi}\int_{B^{+}_{\lambda,u}(x_{0})\cap B_{\lambda \delta(\varepsilon)}(x)} \frac{\lambda^{2 \varepsilon}}{|x-y|^{2 \varepsilon}}u_{x_{0},\lambda}^{3}(y)w^{u}_{x_{0},\lambda}(y)\mathrm{d}y \\
&\quad +C\left(\delta(\varepsilon)\right)\int_{B^{+}_{\lambda,u}(x_{0})\setminus B_{\lambda \delta(\varepsilon)}(x)}u_{x_{0},\lambda}^{3}(y)w^{u}_{x_{0},\lambda}(y)\mathrm{d}y,
\end{aligned}
\end{equation}
where $u(y)<\eta_{x_{0},\lambda}(y)<u_{x_{0},\lambda}(y)$ for any $y\in B^{+}_{\lambda,u}(x_{0})$.

\medskip

Now we need the following Hardy-Littlewood-Sobolev inequality (cf. e.g. \cite{FL1,FL,Lieb}, see also \cite{DHL,DGZ,DZ,NN}).
\begin{lem}[Hardy-Littlewood-Sobolev inequality]\label{lem3}
Let $n\geq1$, $0<s<n$, and $1<p<q<+\infty$ be such that $s+\frac{n}{q}=\frac{n}{p}$. Then we have
\begin{equation}
\left\|\int_{\mathbb{R}^{n}} \frac{f(y)}{|x-y|^{n-s}} d y\right\|_{L^{q}\left(\mathbb{R}^{n}\right)} \leq C_{n, s, p, q}\|f\|_{L^{p}\left(\mathbb{R}^{n}\right)}
\end{equation}
for all $f\in L^{p}(\mathbb{R}^{n})$.
\end{lem}

By Hardy-Littlewood-Sobolev inequality, H\"{o}lder inequality, \eqref{m7} and \eqref{m13}, we have, for any given $\varepsilon\in (0,\frac{1}{6})$ small enough (to be chosen later) and any $\frac{1}{\varepsilon}<q<+\infty$,
\begin{equation}\label{m14}
\begin{aligned}
&\left\|w^{u}_{x_{0},\lambda}\right\|_{L^{r}\left(B_{\lambda,u}^{+}(x_{0})\right)}
\leq C\left\|e^{v_{x_{0},\lambda}} \, w^{v}_{x_{0},\lambda}\right\|_{L^{\frac{6q}{5q+6}}\left(B_{\lambda,v}^{+}(x_{0})\right)} \\
\leq & C\left\|e^{v_{x_{0},\lambda}}\right\|_{L^{\frac{6}{5}}\left(B_{\lambda,v}^{+}(x_{0})\right)}
\left\|w^{v}_{x_{0},\lambda}\right\|_{L^{q}\left(B_{\lambda,v}^{+}(x_{0})\right)} \\
\leq & C\lambda^{\frac{1}{3}}\left(\int_{|x-x_{0}|\geq\lambda}\frac{e^{\frac{6}{5}v(x)}}{|x-x_{0}|^{\frac{2}{5}}}\mathrm{d}x\right)^{\frac{5}{6}}
\left\|w^{v}_{x_{0},\lambda}\right\|_{L^{q}\left(B_{\lambda,v}^{+}(x_{0})\right)} \\
\leq & C\lambda^{\frac{1}{3}}\left(\int_{\mathbb{R}^{2}}\frac{e^{\frac{6}{5}v(x)}}{|x-x_{0}|^{\frac{2}{5}}}\mathrm{d}x\right)^{\frac{5}{6}}
\left\|w^{v}_{x_{0},\lambda}\right\|_{L^{q}\left(B_{\lambda,v}^{+}(x_{0})\right)},
\end{aligned}
\end{equation}
where $r=\frac{3q}{q+3}\in\left(\frac{3}{1+3\varepsilon},3\right)$, and
\begin{equation}\label{m14'}
\begin{aligned}
&\left\|w^{v}_{x_{0},\lambda}\right\|_{L^{q}\left(B_{\lambda,v}^{+}(x_{0})\right)} \\
\leq & C \lambda^{2 \varepsilon}\left\|u_{x_{0},\lambda}^{3} \, w^{u}_{x_{0},\lambda}\right\|_{L^{\frac{q}{1+(1-\varepsilon) q}}\left(B_{\lambda,u}^{+}(x_{0})\right)} \\
&+C(\delta(\varepsilon))\left|B_{\lambda,v}^{+}(x_{0})\right|^{\frac{1}{q}}\int_{B_{\lambda,u}^{+}(x_{0})}u_{x_{0},\lambda}^{3}(y)w^{u}_{x_{0},\lambda}(y)\mathrm{d}y \\
\leq & C \lambda^{2 \varepsilon}\left\|u_{x_{0},\lambda}\right\|^{3}_{L^{\frac{9}{2-3\varepsilon}}\left(B_{\lambda,u}^{+}(x_{0})\right)}
\left\|w^{u}_{x_{0},\lambda}\right\|_{L^{r}\left(B_{\lambda,u}^{+}(x_{0})\right)} \\
&+C(\delta(\varepsilon))\left|B_{\lambda,v}^{+}(x_{0})\right|^{\frac{1}{q}}\left\|u_{x_{0},\lambda}\right\|^{3}_{L^{\frac{3r}{r-1}}\left(B_{\lambda,u}^{+}(x_{0})\right)}
\left\|w^{u}_{x_{0},\lambda}\right\|_{L^{r}\left(B_{\lambda,u}^{+}(x_{0})\right)} \\
\leq & C \lambda^{-\frac{1}{3}-2\varepsilon}
\left(\int_{\mathbb{R}^{2}}|x-x_{0}|^{\frac{1+12\varepsilon}{2-3\epsilon}}u^{\frac{9}{2-3\epsilon}}(x)\mathrm{d}x\right)^{\frac{2-3\varepsilon}{3}}
\left\|w^{u}_{x_{0},\lambda}\right\|_{L^{r}\left(B_{\lambda,u}^{+}(x_{0})\right)} \\
&+ C(\delta(\varepsilon))\lambda^{-\frac{1}{3}-\frac{2}{q}}
\left(\int_{\mathbb{R}^{2}}|x-x_{0}|^{\frac{q+12}{2q-3}}u^{\frac{9q}{2q-3}}(x)\mathrm{d}x\right)^{\frac{2q-3}{3q}}
\left\|w^{u}_{x_{0},\lambda}\right\|_{L^{r}\left(B_{\lambda,u}^{+}(x_{0})\right)}.
\end{aligned}
\end{equation}
Since $\alpha\geq3$, by the asymptotic properties of $(u,v)$ in \eqref{18} and \eqref{33} in Corollary \ref{cor2}, we deduce that, for $\varepsilon\in(0,\frac{1}{6})$ sufficiently small and hence $q\in\left(\frac{1}{\varepsilon},+\infty\right)$ large enough, $\frac{e^{\frac{6}{5}v}}{|x-x_{0}|^{\frac{2}{5}}}\in L^{1}(\mathbb{R}^{2})$, $|x-x_{0}|^{\frac{1+12\varepsilon}{2-3\epsilon}}u^{\frac{9}{2-3\epsilon}}\in L^{1}(\mathbb{R}^{2})$ and $|x-x_{0}|^{\frac{q+12}{2q-3}}u^{\frac{9q}{2q-3}}\in L^{1}(\mathbb{R}^{2})$. Therefore, \eqref{m14} and \eqref{m14'} yields that
\begin{equation}\label{m15}
  \left\|w^{u}_{x_{0},\lambda}\right\|_{L^{r}\left(B_{\lambda,u}^{+}(x_{0})\right)}\leq C_{q,\varepsilon}\max\left\{\lambda^{-2\varepsilon},\lambda^{-\frac{2}{q}}\right\}\left\|w^{u}_{x_{0},\lambda}\right\|_{L^{r}\left(B_{\lambda,u}^{+}(x_{0})\right)},
\end{equation}
\begin{equation}\label{m15'}
  \left\|w^{v}_{x_{0},\lambda}\right\|_{L^{q}\left(B_{\lambda,v}^{+}(x_{0})\right)}\leq C_{q,\varepsilon}\max\left\{\lambda^{-2\varepsilon},\lambda^{-\frac{2}{q}}\right\}\left\|w^{v}_{x_{0},\lambda}\right\|_{L^{q}\left(B_{\lambda,v}^{+}(x_{0})\right)}.
\end{equation}
As a consequence, if we choose $\varepsilon\in\left(0,\frac{1}{6}\right)$ sufficiently small and hence $q\in\left(\frac{1}{\varepsilon},+\infty\right)$ large enough, there exists a $\Lambda_{0}>0$ large enough such that
\begin{equation}\label{m16}
 \left\|w^{u}_{x_{0},\lambda}\right\|_{L^{r}\left(B_{\lambda,u}^{+}\right)}\leq \frac{1}{2}\left\|w^{u}_{x_{0},\lambda}\right\|_{L^{r}\left(B_{\lambda,u}^{+}\right)}, \qquad \left\|w^{v}_{x_{0},\lambda}\right\|_{L^{q}\left(B_{\lambda,v}^{+}\right)}\leq\frac{1}{2}
 \left\|w^{v}_{x_{0},\lambda}\right\|_{L^{q}\left(B_{\lambda,v}^{+}\right)}
\end{equation}
for all $\Lambda_{0}\leq\lambda<+\infty$. By \eqref{m16}, we arrive at
\begin{equation}\label{m17}
\left\|w^{u}_{x_{0},\lambda}\right\|_{L^{r}\left(B_{\lambda,u}^{+}(x_{0})\right)}=\left\|w^{v}_{x_{0},\lambda}\right\|_{L^{q}\left(B_{\lambda,v}^{+}(x_{0})\right)}=0,
\end{equation}
which means $B_{\lambda,u}^{+}(x_{0})=B_{\lambda,v}^{+}(x_{0})=\emptyset$ for any $\Lambda_{0}\leq\lambda<+\infty$. Therefore, we have proved that, for all $\Lambda_{0}\leq\lambda<+\infty$,
\begin{equation}\label{m18}
w^{u}_{x_{0},\lambda}(x)\leq0, \quad w^{v}_{x_{0},\lambda}(x)\leq 0, \qquad \forall \, x \in B_{\lambda}(x_{0})\setminus\{x_{0}\}.
\end{equation}
This completes Step 1 for the case $\alpha\geq3$.

\medskip

\noindent \emph{Case (ii)} $1\leq\alpha\leq3$. We will show that, for $\lambda>0$ sufficiently small,
\begin{equation}\label{m1-}
w^{u}_{x_{0},\lambda}(x)\geq 0 \quad \text{and} \quad w^{v}_{x_{0},\lambda}(x)\geq 0, \qquad \forall x \in B_{\lambda}(x_{0})\setminus\{x_{0}\}.
\end{equation}
That is, we start moving the sphere $S_{\lambda}(x_{0}):=\{x\in\mathbb{R}^{2}\mid \, |x-x_{0}|=\lambda\}$ from near the point $x_{0}$ outward such that \eqref{m1-} holds.

Define
\begin{equation}\label{m2-}
B_{\lambda,u}^{-}(x_{0}):=\left\{x \in B_{\lambda}(x_{0})\setminus\{x_{0}\} \mid w^{u}_{x_{0},\lambda}(x)<0\right\},
\end{equation}
\begin{equation}\label{m2'-}
B_{\lambda,v}^{-}(x_{0}):=\left\{x \in B_{\lambda}(x_{0})\setminus\{x_{0}\} \mid w^{v}_{x_{0},\lambda}(x)<0\right\}.
\end{equation}
We will show that, for $\lambda>0$ sufficiently small,
\begin{equation}\label{m3-}
B_{\lambda,u}^{-}(x_{0})=B_{\lambda,v}^{-}(x_{0})=\emptyset.
\end{equation}

From \eqref{m6} and the mean value theorem, one can derive that, for any $x\in B^{-}_{\lambda,u}(x_{0})$,
\begin{equation}\label{m7-}
\begin{aligned}
0&>w^{u}_{x_{0},\lambda}(x)=u_{x_{0},\lambda}(x)-u(x) \\
&\geq \frac{1}{2\pi}\int_{B^{-}_{\lambda,v}(x_{0})}\left[\frac{1}{|x-y|}-\frac{1}{\left|\frac{\lambda(x-x_{0})}{|x-x_{0}|}-\frac{|x-x_{0}|}{\lambda}(y-x_{0})\right|}\right]
\left(e^{v_{x_{0},\lambda}(y)}-e^{v(y)}\right)\mathrm{d}y \\
&\geq \frac{1}{2\pi}\int_{B^{-}_{\lambda,v}(x_{0})}\frac{1}{|x-y|}e^{\bar{\xi}_{x_{0},\lambda}(y)}w^{v}_{x_{0},\lambda}(y)\mathrm{d}y \\
&\geq \frac{1}{2\pi}\int_{B^{-}_{\lambda,v}(x_{0})}\frac{e^{v(y)}}{|x-y|}w^{v}_{x_{0},\lambda}(y)\mathrm{d}y,
\end{aligned}
\end{equation}
where $v_{x_{0},\lambda}(y)<\bar{\xi}_{x_{0},\lambda}(y)<v(y)$ for any $y\in B^{-}_{\lambda,v}(x_{0})$. Due to $1\leq\alpha\leq3$, from \eqref{m6'}, \eqref{m11}, \eqref{m12} and the mean value theorem, one can derive that, for any given $\varepsilon>0$ small and all $x\in B^{-}_{\lambda,v}(x_{0})$,
\begin{equation}\label{m13-}
\begin{aligned}
0&>w^{v}_{x_{0},\lambda}(x)=v_{x_{0},\lambda}(x)-v(x) \\
&\geq\frac{1}{2\pi}\int_{B^{-}_{\lambda,u}(x_{0})}\ln\left[\frac{\left|\frac{\lambda(x-x_{0})}{|x-x_{0}|}-\frac{|x-x_{0}|}{\lambda}(y-x_{0})\right|}{\left|x-y\right|}\right]
\left(u_{x_{0},\lambda}^{4}(y)-u^{4}(y)\right)\mathrm{d}y \\
&\geq \frac{1}{\pi}\int_{B^{-}_{\lambda,u}(x_{0})\cap B_{\lambda \delta(\varepsilon)}(x)} \frac{\lambda^{2 \varepsilon}}{|x-y|^{2 \varepsilon}}\bar{\eta}_{x_{0},\lambda}^{3}(y)w^{u}_{x_{0},\lambda}(y)\mathrm{d}y \\
&\quad +\frac{\ln \left[1+\frac{1}{\delta(\varepsilon)^{2}}\right]}{\pi}\int_{B^{-}_{\lambda,u}(x_{0})\setminus B_{\lambda \delta(\varepsilon)}(x)}\bar{\eta}_{x_{0},\lambda}^{3}(y)w^{u}_{x_{0},\lambda}(y)\mathrm{d}y \\
&\geq \frac{1}{\pi}\int_{B^{-}_{\lambda,u}(x_{0})\cap B_{\lambda \delta(\varepsilon)}(x)} \frac{\lambda^{2 \varepsilon}}{|x-y|^{2 \varepsilon}}u^{3}(y)w^{u}_{x_{0},\lambda}(y)\mathrm{d}y \\
&\quad +C\left(\delta(\varepsilon)\right)\int_{B^{-}_{\lambda,u}(x_{0})\setminus B_{\lambda\delta(\varepsilon)}(x)}u^{3}(y)w^{u}_{x_{0},\lambda}(y)\mathrm{d}y,
\end{aligned}
\end{equation}
where $u_{x_{0},\lambda}(y)<\bar{\eta}_{x_{0},\lambda}(y)<u(y)$ for any $y\in B^{-}_{\lambda,u}(x_{0})$.

Now we choose $\varepsilon=\frac{1}{9}\in\left(0,\frac{1}{6}\right)$ small enough and $q=12\in\left(\frac{1}{\varepsilon},+\infty\right)$. Since $u\in L^{\infty}_{loc}(\mathbb{R}^{2})$ and $v\in L^{\infty}_{loc}(\mathbb{R}^{2})$, from Hardy-Littlewood-Sobolev inequality, H\"{o}lder inequality, \eqref{m7-} and \eqref{m13-}, we derive that
\begin{equation}\label{m14-}
\begin{aligned}
\left\|w^{u}_{x_{0},\lambda}\right\|_{L^{\frac{12}{5}}\left(B_{\lambda,u}^{-}(x_{0})\right)}
\leq& C\left\|e^{v} \, w^{v}_{x_{0},\lambda}\right\|_{L^{\frac{12}{11}}\left(B_{\lambda,v}^{-}(x_{0})\right)} \\
\leq & C\left\|e^{v}\right\|_{L^{\frac{6}{5}}\left(B_{\lambda,v}^{-}(x_{0})\right)}
\left\|w^{v}_{x_{0},\lambda}\right\|_{L^{12}\left(B_{\lambda,v}^{-}(x_{0})\right)} \\
\leq & C\lambda^{\frac{5}{3}}\left\|w^{v}_{x_{0},\lambda}\right\|_{L^{12}\left(B_{\lambda,v}^{-}(x_{0})\right)}
\end{aligned}
\end{equation}
and
\begin{equation}\label{m14'-}
\begin{aligned}
&\left\|w^{v}_{x_{0},\lambda}\right\|_{L^{12}\left(B_{\lambda,v}^{-}(x_{0})\right)} \\
\leq & C \lambda^{\frac{2}{9}}\left\|u^{3} \, w^{u}_{x_{0},\lambda}\right\|_{L^{\frac{36}{35}}\left(B_{\lambda,u}^{-}(x_{0})\right)}
+C\left|B_{\lambda,v}^{-}(x_{0})\right|^{\frac{1}{12}}\int_{B_{\lambda,u}^{-}(x_{0})}u^{3}(y)w^{u}_{x_{0},\lambda}(y)\mathrm{d}y \\
\leq & C \lambda^{\frac{2}{9}}\left\|u\right\|^{3}_{L^{\frac{27}{5}}\left(B_{\lambda,u}^{-}(x_{0})\right)}
\left\|w^{u}_{x_{0},\lambda}\right\|_{L^{\frac{12}{5}}\left(B_{\lambda,u}^{-}(x_{0})\right)} \\
&+C\left|B_{\lambda,v}^{-}(x_{0})\right|^{\frac{1}{12}}\left\|u\right\|^{3}_{L^{\frac{36}{7}}\left(B_{\lambda,u}^{-}(x_{0})\right)}
\left\|w^{u}_{x_{0},\lambda}\right\|_{L^{\frac{12}{5}}\left(B_{\lambda,u}^{-}(x_{0})\right)} \\
\leq & C \lambda^{\frac{4}{3}}\left\|w^{u}_{x_{0},\lambda}\right\|_{L^{\frac{12}{5}}\left(B_{\lambda,u}^{-}(x_{0})\right)}.
\end{aligned}
\end{equation}
Consequently, \eqref{m14-} and \eqref{m14'-} yields that
\begin{equation}\label{m15-}
  \left\|w^{u}_{x_{0},\lambda}\right\|_{L^{\frac{12}{5}}\left(B_{\lambda,u}^{-}(x_{0})\right)}\leq C\lambda^{3}\left\|w^{u}_{x_{0},\lambda}\right\|_{L^{\frac{12}{5}}\left(B_{\lambda,u}^{-}(x_{0})\right)},
\end{equation}
\begin{equation}\label{m15'-}
  \left\|w^{v}_{x_{0},\lambda}\right\|_{L^{12}\left(B_{\lambda,v}^{-}(x_{0})\right)}\leq C\lambda^{3}\left\|w^{v}_{x_{0},\lambda}\right\|_{L^{12}\left(B_{\lambda,v}^{-}(x_{0})\right)}.
\end{equation}
As a consequence, there exists a $\epsilon_{0}>0$ small enough, such that
\begin{equation}\label{m16-}
 \left\|w^{u}_{x_{0},\lambda}\right\|_{L^{\frac{12}{5}}\left(B_{\lambda,u}^{-}\right)}\leq \frac{1}{2}\left\|w^{u}_{x_{0},\lambda}\right\|_{L^{\frac{12}{5}\left(B_{\lambda,u}^{-}\right)}}, \qquad \left\|w^{v}_{x_{0},\lambda}\right\|_{L^{12}\left(B_{\lambda,v}^{-}\right)}\leq\frac{1}{2}
 \left\|w^{v}_{x_{0},\lambda}\right\|_{L^{12}\left(B_{\lambda,v}^{-}\right)}
\end{equation}
for all $0<\lambda\leq\epsilon_{0}$. By \eqref{m16-}, we arrive at
\begin{equation}\label{m17-}
\left\|w^{u}_{x_{0},\lambda}\right\|_{L^{\frac{12}{5}}\left(B_{\lambda,u}^{-}(x_{0})\right)}=\left\|w^{v}_{x_{0},\lambda}\right\|_{L^{12}\left(B_{\lambda,v}^{-}(x_{0})\right)}=0,
\end{equation}
which means $B_{\lambda,u}^{-}(x_{0})=B_{\lambda,v}^{-}(x_{0})=\emptyset$ for any $0<\lambda\leq\epsilon_{0}$. Therefore, we have proved that, for all $0<\lambda\leq\epsilon_{0}$,
\begin{equation}\label{m18-}
w^{u}_{x_{0},\lambda}(x)\geq0, \quad w^{v}_{x_{0},\lambda}(x)\geq 0, \qquad \forall \, x \in B_{\lambda}(x_{0})\setminus\{x_{0}\}.
\end{equation}
This completes Step 1 for the case $1\leq\alpha\leq3$.

\bigskip

Step 2. Moving the sphere $S_{\lambda}$ outward or inward until the limiting position.

\medskip

In what follows, we will derive contradictions in both the cases $\alpha>3$ and $1\leq \alpha<3$, and hence we must have $\alpha=3$.

\medskip

\noindent\emph{Case (i)} $\alpha>3$.
Step 1 provides a starting point to carry out the method of moving spheres for any given center $x_0\in \mathbb{R}^{2}$. Next we will continuously decrease the radius $\lambda$ as long as \eqref{m1} holds. For arbitrarily given center $x_0$, the critical scale $\lambda_{x_0}$ is defined by
\begin{equation}\label{m19}
\lambda_{x_0}:=\inf\left\{\lambda\in(0,+\infty) \mid w^{u}_{x_0,\mu}\leq 0, \, w^{v}_{x_0,\mu}\leq 0 \,\, \text{in} \,\, B_\mu(x_0)\setminus\{x_0\}, \,\, \forall \, \lambda\leq\mu<+\infty\right\}.
\end{equation}

\smallskip

From Step 1, we know that $\lambda_{x_{0}}$ is well-defined and $0\leq\lambda_{x_{0}}<+\infty$ for any $x_{0}\in\mathbb{R}^{2}$. We first show that, in the case $\alpha>3$, it must hold $\lambda_{x_0}=0$. Suppose on the contrary that $\lambda_{x_0}>0$, we will prove that the the sphere can be moved a bit further, which contradicts the definition of $\lambda_{x_{0}}$.

\smallskip

By the definition of $\lambda_{x_{0}}$, we have $w^{u}_{x_{0}, \lambda_{x_{0}}}\leq 0$ and $w^{v}_{x_{0},\lambda_{x_{0}}}\leq 0$ in $B_{\lambda_{x_{0}}}(x_{0}) \setminus\{x_{0}\}$. Then, we infer from \eqref{m6'} that, for any $x\in B_{\lambda_{x_{0}}}(x_{0})\setminus\{x_{0}\}$, it holds
\begin{equation}\label{m20}
\begin{aligned} w^{v}_{x_{0},\lambda_{x_0}}(x)
=&\frac{1}{2\pi}\int_{B_{\lambda_{x_0}}(x_{0})}\ln\left[\frac{\left|\frac{\lambda_{x_0}(x-x_{0})}{|x-x_{0}|}-\frac{|x-x_{0}|}{\lambda_{x_0}}(y-x_{0})\right|}{\left|x-y\right|}\right]
		\left(u_{x_{0},\lambda_{x_0}}^{4}(y)-u^{4}(y)\right)\mathrm{d}y \\
		& +\left(3-\alpha\right)\ln\frac{\lambda_{x_0}}{|x-x_{0}|}\\
		\leq&\left(3-\alpha\right)\ln\frac{\lambda_{x_0}}{|x-x_{0}|}<0,
\end{aligned}
\end{equation}
which combined with \eqref{m6}, further implies that
\begin{equation}\label{m21}
\begin{aligned}
w^{u}_{x_{0},\lambda_{x_0}}(x)=&\frac{1}{2\pi}\int_{B_{\lambda_{x_0}}(x_{0})}\left[\frac{1}{|x-y|}-\frac{1}{\left|\frac{\lambda_{x_0}(x-x_{0})}{|x-x_{0}|}-\frac{|x-x_{0}|}{\lambda_{x_0}}(y-x_{0})\right|}\right] \\ &\times\left(e^{v_{x_{0},\lambda_{x_0}}(y)}-e^{v(y)}\right)\mathrm{d}y<0.
\end{aligned}
\end{equation}

Now choose $\delta_1>0$ sufficiently small, which will be determined later. Define the narrow region
\begin{equation}\label{m22}
	A_{\delta_1}:=\left\{x\in \mathbb{R}^2 \,|\,0<|x-x_0|<\delta_1 \,\, \text{or} \,\, \lambda_{x_0}-\delta_1<|x-x_0|<\lambda_{x_0}\right\}\subset B_{\lambda_{x_{0}}}(x_{0})\setminus\{x_{0}\}.
\end{equation}	
Since that $w^u_{x_0, \lambda_{x_0}}$ and $w^v_{x_0, \lambda_{x_0}}$ are continuous in $\mathbb{R}^{2}\setminus\{x_0\}$ and $A_{\delta_1}^c:=\left(B_{\lambda_{x_0}}(x_0)\setminus\{x_{0}\}\right)\setminus A_{\delta_1} $ is a compact subset, there exists a positive constant $C_0>0$ such that
\begin{equation}\label{m23}
	w^u_{x_0, \lambda_{x_0}}(x)<-C_0 \quad \text{and} \quad w^v_{x_0, \lambda_{x_0}}(x)<-C_0, \quad\quad \forall \,\, x\in A_{\delta_1}^c.
\end{equation}
By continuity, we can choose $\delta_2>0$ sufficiently small such that, for any $\lambda\in [\lambda_{x_0}-\delta_2,\,\lambda_{x_0}]$,
\begin{equation}\label{m24}
	w^{u}_{x_0, \lambda}(x)<-\frac{C_0}{2} \quad \text{and} \quad w^{v}_{x_0,\lambda}(x)<-\frac{C_0}{2}, \quad\quad \forall \,\, x\in A_{\delta_1}^c.
\end{equation}
Hence we must have
\begin{equation}\label{m25}
\begin{aligned}
	B_{\lambda,u}^{+}\cup B_{\lambda,v}^{+}&\subset\left(B_{\lambda}(x_0)\setminus\{x_{0}\}\right)\setminus A^{c}_{\delta_1} \\
&=\left\{x\in \mathbb{R}^2\,|\,0<|x-x_0|<\delta_1 \,\, \text{or} \,\, \lambda_{x_0}-\delta_1<|x-x_0|<\lambda\right\}
\end{aligned}
\end{equation}
for any $\lambda\in [\lambda_{x_0}-\delta_2,\,\lambda_{x_0}]$. By \eqref{m14} and \eqref{m14'}, we conclude that, for $\varepsilon\in\left(0,\frac{1}{6}\right)$ small and $q\in\left(\frac{1}{\varepsilon},+\infty\right)$ large,
\begin{equation}\label{m26}
	\begin{aligned}
		\left\|w^{u}_{x_{0},\lambda}\right\|_{L^{r}\left(B_{\lambda,u}^{+}(x_{0})\right)}		
		\leq&  C\left\|e^{v_{x_{0},\lambda}}\right\|_{L^{\frac{6}{5}}\left(B_{\lambda,v}^{+}(x_{0})\right)}
		\left\|w^{v}_{x_{0},\lambda}\right\|_{L^{q}\left(B_{\lambda,v}^{+}(x_{0})\right)} \\
		\leq& C_{\lambda,\varepsilon}\left\|e^{v_{x_{0},\lambda}}\right\|_{L^{\frac{6}{5}}\left(B_{\lambda,v}^{+}(x_{0})\right)}
\left\|w^{u}_{x_{0},\lambda}\right\|_{L^{r}\left(B_{\lambda,u}^{+}(x_{0})\right)} \\
	&\times\left(\left\|u_{x_{0},\lambda}\right\|^{3}_{L^{\frac{9}{2-3\varepsilon}}\left(B_{\lambda,u}^{+}(x_{0})\right)}
	+\left|B_{\lambda,v}^{+}(x_{0})\right|^{\frac{1}{q}}\left\|u_{x_{0},\lambda}\right\|^{3}_{L^{\frac{3r}{r-1}}\left(B_{\lambda,u}^{+}(x_{0})\right)}\right),
	\end{aligned}
\end{equation}
and
\begin{equation}\label{m27}
	\begin{aligned}
		\left\|w^{v}_{x_{0},\lambda}\right\|_{L^{q}\left(B_{\lambda,v}^{+}(x_{0})\right)}\leq & C_{\lambda,\varepsilon}\left\|e^{v_{x_{0},\lambda}}\right\|_{L^{\frac{6}{5}}\left(B_{\lambda,v}^{+}(x_{0})\right)}
\left\|w^{v}_{x_{0},\lambda}\right\|_{L^{q}\left(B_{\lambda,v}^{+}(x_{0})\right)} \\
&\times \left(\left\|u_{x_{0},\lambda}\right\|^{3}_{L^{\frac{9}{2-3\varepsilon}}\left(B_{\lambda,u}^{+}(x_{0})\right)}
+\left|B_{\lambda,v}^{+}(x_{0})\right|^{\frac{1}{q}}\left\|u_{x_{0},\lambda}\right\|^{3}_{L^{\frac{3r}{r-1}}\left(B_{\lambda,u}^{+}(x_{0})\right)}\right),
	\end{aligned}
\end{equation}
where $r=\frac{3q}{q+3}$. Since $\alpha\geq3$, we can infer from the asymptotic properties of $(u,v)$ in \eqref{18} and \eqref{33} in Corollary \ref{cor2} that, for $\varepsilon\in\left(0,\frac{1}{6}\right)$ sufficiently small and hence $q\in\left(\frac{1}{\varepsilon},+\infty\right)$ large enough, $e^{v_{x_{0},\lambda}}\in L^{\frac{6}{5}}(\mathbb{R}^{2})$ and $u_{x_{0},\lambda}\in L^{\frac{9}{2-3\varepsilon}}\cap L^{\frac{9q}{2q-3}}(\mathbb{R}^{2})$. Therefore, by \eqref{m25}, we can choose $\delta_1$ and $\delta_2$ sufficiently small such that, for any $\lambda\in [\lambda_{x_0}-\delta_2,\,\lambda_{x_0}]$,
\begin{equation}\label{m28}
C_{\lambda,\varepsilon}\left\|e^{v_{x_{0},\lambda}}\right\|_{L^{\frac{6}{5}}\left(B_{\lambda,v}^{+}(x_{0})\right)}\left[\left\|u_{x_{0},\lambda}\right\|^{3}_{L^{\frac{9}{2-3\varepsilon}}\left(B_{\lambda,u}^{+}\right)}
+\left|B_{\lambda,v}^{+}\right|^{\frac{1}{q}}\left\|u_{x_{0},\lambda}\right\|^{3}_{L^{\frac{3r}{r-1}}\left(B_{\lambda,u}^{+}\right)}\right]<\frac{1}{2}.
\end{equation}
Combining this with \eqref{m26} and \eqref{m27}, we obtain that, for any $\lambda\in[\lambda_{x_0}-\delta_2,\,\lambda_{x_0}]$,
\begin{equation}\label{m29} \left\|w^{u}_{x_{0},\lambda}\right\|_{L^{r}\left(B_{\lambda,u}^{+}(x_{0})\right)}=\left\|w^{v}_{x_{0},\lambda}\right\|_{L^{q}\left(B_{\lambda,v}^{+}(x_{0})\right)}=0,
\end{equation}
and hence
\begin{equation}\label{m30}
	w^{u}_{x_{0},\lambda}(x)\leq 0 \quad \text{and} \quad w^{v}_{x_{0},\lambda}(x)\leq 0,\quad\quad \forall \,\, x\in B_{\lambda}(x_0)\setminus\{x_0\}.
\end{equation}
This contradicts the definition of $\lambda_{x_0}$. Therefore, we must have $\lambda_{x_0}=0$ for any $x_0\in\mathbb R^2$.

\medskip

In order to derive a contradiction, we also need the following calculus Lemma (see Lemma 11.1 and Lemma 11.2 in \cite{LZ1}, see also \cite{Li,LZ,Xu}).
\begin{lem}[Lemma 11.1 and Lemma 11.2 in \cite{LZ1}]\label{lem10}
Let $n\geq 1$, $\nu\in\mathbb{R}$ and $u\in C^{1}({\mathbb{R}^n})$. For every $x_0\in \mathbb{R}^n$ and $\lambda>0$, define $u_{x_0,\lambda}(x):={\left(\frac{\lambda}{|x-x_0|}\right)}^{\nu}u\left(\frac{\lambda^2(x-x_0)}{|x-x_0|^2}+x_0\right)$ for any $x\in\mathbb{R}^{n}\setminus\{x_{0}\}$. Then, we have\\
	\emph{(i)} If for every $x_0\in \mathbb{R}^n$, there exists a $0<\lambda_{x_0}<+\infty$ such that
	\begin{equation*}
		u_{x_0,\lambda_{x_0}}(x)=u(x), \qquad \forall \,\, x\in\mathbb{R}^n\setminus\{x_0\},
	\end{equation*}
	then for some $C\in\mathbb{R}$, $\mu>0$ and $\bar x\in \mathbb{R}^n$,
	\begin{equation*}
		u(x)=C{\left(\frac{\mu}{1+\mu^{2}|x-\bar x|^2}\right)}^{\frac{\nu}{2}}.
	\end{equation*}
	\emph{(ii)} If for every $x_0\in \mathbb{R}^n$ and any $0<\lambda<+\infty$,
	\begin{equation*}
		u_{x_0,\lambda}(x)\geq u(x), \qquad \forall \,\, x\in B_\lambda(x_0)\setminus\{x_0\},
	\end{equation*}
	then $u\equiv C$ for some constant $C\in\mathbb{R}$.
\end{lem}
\begin{rem}\label{rem11}
In Lemma 11.1 and Lemma 11.2 of \cite{LZ1}, Li and Zhang have proved Lemma \ref{lem10} for $\nu>0$. Nevertheless, their methods can also be applied to show Lemma \ref{lem10} in the cases $\nu\leq0$, see \cite{Li,LZ,Xu}.
\end{rem}

From the conclusion (ii) in Lemma \ref{lem10} (replacing $u$ by $-u$ therein), we deduce that $u\equiv C$ for some constant $C$. Since $u^4\in L^1{(\mathbb R^2)}$, we must have $u\equiv0$. However, by the first equation in the system \eqref{PDE}, we have
$$0=e^{v(x)}>0 \qquad \text{in} \,\, \mathbb{R}^{2}.$$
This is a contradiction and hence $\alpha>3$ is impossible.

\medskip

\noindent\emph{Case (ii)} $1\leq \alpha<3$. In this case, for arbitrarily given center $x_0$, the critical scale $\lambda_{x_0}$ is defined by
\begin{equation}\label{m31}
	\lambda_{x_0}:=\sup\left\{\lambda>0\mid \, w^u_{x_0,\mu}\geq 0, \, w^v_{x_0,\mu}\geq 0 \,\, \text{in} \,\, B_\mu(x_0)\setminus\{x_0\}, \,\, \forall \, 0<\mu\leq\lambda\right\}.
\end{equation}
Step 1 yields that $\lambda_{x_{0}}$ is well-defined and $0<\lambda_{x_{0}}\leq+\infty$ for any $x_{0}\in\mathbb{R}^{2}$. We will show that $\lambda_{x_0}=+\infty$, which will lead to a contradiction again as in \emph{Case (i)} $\alpha>3$. Suppose on the contrary that $\lambda_{x_0}<+\infty$, we will prove that the sphere can be moved a bit further, which contradicts the definition of $\lambda_{x_0}$.

\smallskip

By the definition of $\lambda_{x_0}$, we have $w^u_{x_0, \lambda_{x_0}}\geq 0$ and $w^v_{x_0,\lambda_{x_0}}\geq 0$ in $B_{\lambda_{x_0}}(x_0) \setminus\{x_0\}$. Then, we infer from \eqref{m6'} that, for any $x\in B_{\lambda_{x_{0}}}(x_{0})\setminus\{x_{0}\}$, it holds
\begin{equation}\label{m32}
	\begin{aligned}
		w^{v}_{x_{0},\lambda_{x_0}}(x) =&\frac{1}{2\pi}\int_{B_{\lambda_{x_0}}(x_{0})}\ln\left[\frac{\left|\frac{\lambda_{x_0}(x-x_{0})}{|x-x_{0}|}-\frac{|x-x_{0}|}{\lambda_{x_0}}(y-x_{0})\right|}{\left|x-y\right|}\right]
		\left(u_{x_{0},\lambda_{x_0}}^{4}(y)-u^{4}(y)\right)\mathrm{d}y \\
		& +\left(3-\alpha\right)\ln\frac{\lambda_{x_0}}{|x-x_{0}|}\\
		\geq&\left(3-\alpha\right)\ln\frac{\lambda_{x_0}}{|x-x_{0}|}>0,
	\end{aligned}
\end{equation}
which combined with \eqref{m6}, further implies that
\begin{equation}\label{m33}
\begin{aligned}
	w^{u}_{x_{0},\lambda_{x_0}}(x) =&\frac{1}{2\pi}\int_{B_{\lambda_{x_0}}(x_{0})}\left[\frac{1}{|x-y|}-\frac{1}{\left|\frac{\lambda_{x_0}(x-x_{0})}{|x-x_{0}|}-\frac{|x-x_{0}|}{\lambda_{x_0}}(y-x_{0})\right|}\right] \\
&\times \left(e^{v_{x_{0},\lambda_{x_0}}(y)}-e^{v(y)}\right)\mathrm{d}y>0.
\end{aligned}
\end{equation}

Now choose $\delta_1>0$ sufficiently small, which will be determined later. Define the narrow region $A_{\delta_1}\subset B_{\lambda_{x_{0}}}(x_{0})\setminus\{x_{0}\}$ by \eqref{m22}. Since $w^u_{x_0, \lambda_{x_0}}$ and $w^v_{x_0, \lambda_{x_0}}$ are continuous in $\mathbb{R}^{2}\setminus\{x_0\}$ and $A_{\delta_1}^c:=\left(B_{\lambda_{x_0}}(x_0)\setminus\{x_0\}\right)\setminus A_{\delta_1}$ is a compact subset, there exists a positive constant $C_1>0$ such that
\begin{equation}\label{m34}
	w^u_{x_0, \lambda_{x_0}}(x)>C_1 \quad \text{and} \quad w^v_{x_0, \lambda_{x_0}}(x)>C_1, \quad\quad \forall \,\, x\in A_{\delta_1}^c.
\end{equation}
By continuity, we can choose $\delta_2>0$ sufficiently small, such that, for any $\lambda\in [\lambda_{x_0},\,\lambda_{x_0}+\delta_2]$, there holds
\begin{equation}\label{m35}
	w^u_{x_0, \lambda}(x)>\frac{C_1}{2} \quad \text{and} \quad w^v_{x_0, \lambda}(x)>\frac{C_1}{2}, \quad\quad \forall \,\, x\in A_{\delta_1}^c.
\end{equation}
Hence we must have
\begin{equation}\label{m36}
\begin{aligned}
	B_{\lambda,u}^{-}\cup B_{\lambda,v}^{-}&\subset \left(B_{\lambda}(x_0)\setminus\{x_0\}\right)\setminus A^{c}_{\delta_1} \\
&=\left\{x\in \mathbb{R}^2\,|\,0<|x-x_0|<\delta_1\,\, \text{or}\,\, \lambda_{x_0}-\delta_1<|x-x_0|<\lambda\right\}
\end{aligned}
\end{equation}
for any $\lambda\in [\lambda_{x_0},\,\lambda_{x_0}+\delta_2]$. By \eqref{m14-} and \eqref{m14'-}, we conclude that
\begin{equation}\label{m37}
	\begin{aligned}
		\left\|w^{u}_{x_{0},\lambda}\right\|_{L^{\frac{12}{5}}\left(B_{\lambda,u}^{-}(x_{0})\right)}
		\leq& C\left\|e^{v}\right\|_{L^{\frac{6}{5}}\left(B_{\lambda,v}^{-}(x_{0})\right)}
		\left\|w^{v}_{x_{0},\lambda}\right\|_{L^{12}\left(B_{\lambda,v}^{-}(x_{0})\right)} \\
		\leq& C_{\lambda}\left\|e^{v}\right\|_{L^{\frac{6}{5}}\left(B_{\lambda,v}^{-}(x_{0})\right)}
\left\|w^{u}_{x_{0},\lambda}\right\|_{L^{\frac{12}{5}}\left(B_{\lambda,u}^{-}(x_{0})\right)} \\
&\times \left(\left\|u\right\|^{3}_{L^{\frac{27}{5}}\left(B_{\lambda,u}^{-}(x_{0})\right)}
+\left|B_{\lambda,v}^{-}(x_{0})\right|^{\frac{1}{12}}\left\|u\right\|^{3}_{L^{\frac{36}{7}}\left(B_{\lambda,u}^{-}(x_{0})\right)}\right),
\end{aligned}
\end{equation}
and
\begin{equation}\label{m38}
	\begin{aligned}
	\left\|w^{v}_{x_{0},\lambda}\right\|_{L^{12}\left(B_{\lambda,v}^{-}(x_{0})\right)}
 \leq& C_{\lambda}\left\|e^{v}\right\|_{L^{\frac{6}{5}}\left(B_{\lambda,v}^{-}(x_{0})\right)}
 \left\|w^{v}_{x_{0},\lambda}\right\|_{L^{12}\left(B_{\lambda,v}^{-}(x_{0})\right)} \\
 &\times \left(\left\|u\right\|^{3}_{L^{\frac{27}{5}}\left(B_{\lambda,u}^{-}(x_{0})\right)}
 +\left|B_{\lambda,v}^{-}(x_{0})\right|^{\frac{1}{12}}\left\|u\right\|^{3}_{L^{\frac{36}{7}}\left(B_{\lambda,u}^{-}(x_{0})\right)}\right).
	\end{aligned}
\end{equation}
By the local boundedness of $u$ and $v$, we can choose $\delta_1$ and $\delta_2$ sufficiently small such that, for any $\lambda\in [\lambda_{x_0},\,\lambda_{x_0}+\delta_2]$,
\begin{equation}\label{m39} C_{\lambda}\left\|e^{v}\right\|_{L^{\frac{6}{5}}\left(B_{\lambda,v}^{-}(x_{0})\right)}\left(\left\|u\right\|^{3}_{L^{\frac{27}{5}}\left(B_{\lambda,u}^{-}(x_{0})\right)}
+\left|B_{\lambda,v}^{-}(x_{0})\right|^{\frac{1}{12}}\left\|u\right\|^{3}_{L^{\frac{36}{7}}\left(B_{\lambda,u}^{-}(x_{0})\right)}\right)<\frac{1}{2}.
\end{equation}
Then, by \eqref{m37} and \eqref{m38}, we obtain that, for any $\lambda\in[\lambda_{x_0},\,\lambda_{x_0}+\delta_2]$,
\begin{equation}\label{m40}	\left\|w^{u}_{x_{0},\lambda}\right\|_{L^{\frac{12}{5}}\left(B_{\lambda,u}^{-}(x_{0})\right)}
=\left\|w^{v}_{x_{0},\lambda}\right\|_{L^{12}\left(B_{\lambda,v}^{-}(x_{0})\right)}=0,
\end{equation}
and hence
\begin{equation}\label{m41}
	w^{u}_{x_{0},\lambda}(x)\geq 0 \quad \text{and} \quad w^{v}_{x_{0},\lambda}(x)\geq 0,\quad\quad \forall \,\, x\in B_{\lambda}(x_0)\setminus\{x_0\}.
\end{equation}
This contradicts the definition of $\lambda_{x_0}$. Therefore, we must have $\lambda_{x_0}=+\infty$ for any $x_0\in\mathbb R^2$.

\medskip

From the conclusion (ii) in Lemma \ref{lem10} and the fact that $u^4\in L^1{(\mathbb R^2)}$, we must have $u\equiv0$, which will lead to a contradiction again by the first equation in system \eqref{PDE}. Hence, the case $1\leq \alpha<3$ can not happen.

\medskip

From the contradictions derived in both Cases (i) and (ii), we conclude that
\begin{equation}\label{a8}
  \alpha:=\frac{1}{2\pi}\int_{\mathbb{R}^{2}}u^{4}(x)\mathrm{d}x=3.
\end{equation}
In this case, we deduce from Step 1 that, for $\lambda>0$ large,
 \begin{equation}\label{a9}
 	w^{u}_{x_{0},\lambda}(x)\leq 0 \quad \text{and} \quad w^{v}_{x_{0},\lambda}(x)\leq 0, \qquad \forall x \in B_{\lambda}(x_{0})\setminus\{x_{0}\};
 \end{equation}
while for $\lambda>0$ small,
\begin{equation}\label{a10}
	w^{u}_{x_{0},\lambda}(x)\geq 0 \quad \text{and} \quad w^{v}_{x_{0},\lambda}(x)\geq 0, \qquad \forall x \in B_{\lambda}(x_{0})\setminus\{x_{0}\}.
\end{equation}
If the critical scale (defined in \eqref{m31}) $\lambda_{x_{0}}<+\infty$, then we must have $w^{u}_{x_{0},\lambda_{x_{0}}}=w^{v}_{x_{0},\lambda_{x_{0}}}=0$ in $B_{\lambda_{x_{0}}}(x_{0})\setminus\{x_{0}\}$, or else the sphere $S_{\lambda}$ can be moved a bit further such that \eqref{a10} still hold (see \emph{Case (ii)} or \emph{Case (i)} in Step 2), which contradicts the definition \eqref{m31} of $\lambda_{x_{0}}$. If the critical scale (defined in \eqref{m31}) $\lambda_{x_{0}}=+\infty$, it follows from \eqref{a9} that $w^{u}_{x_{0},\lambda}=w^{v}_{x_{0},\lambda}=0$ in $B_{\lambda}(x_{0})\setminus\{x_{0}\}$ for $\lambda$ sufficiently large. As a consequence, for arbitrary $x_0\in\mathbb R^2$, there exists a $\lambda>0$ (depending on $x_{0}$) such that
\begin{equation}\label{m42}
	w^{u}_{x_{0},\lambda}(x)=0 \quad \text{and} \quad w^{v}_{x_{0},\lambda}(x)=0, \qquad \forall x \in B_{\lambda}(x_{0})\setminus\{x_{0}\}.
\end{equation}
Then, we infer from the conclusion (i) in Lemma \ref{lem10} that, for some $C\in\mathbb{R}$, $\mu>0$ and $x_0\in \mathbb{R}^2$, $u$ must be of the form
\begin{equation}\label{m43}
	u(x)=C{\left(\frac{\mu}{1+\mu^{2}|x-x_0|^2}\right)}^{\frac{1}{2}}, \qquad \forall \, x\in\mathbb{R}^{2},
\end{equation}
which combined with the first equation in system \eqref{PDE} and the asymptotic behavior \eqref{22} imply
\begin{equation}\label{m44}
	v(x)=\frac{3}{2}\ln\left[\frac{C'\mu}{1+\mu^{2}|x-x_0|^2}\right], \qquad \forall \, x\in\mathbb{R}^{2}.
\end{equation}
Then, by the formula \eqref{a8} and direct calculations, we can obtain that $C=6^{\frac{1}{4}}$ and hence
\begin{equation}\label{a12}
  u(x)=6^{\frac{1}{4}}\left(\frac{\mu}{1+\mu^{2}|x-x_{0}|^{2}}\right)^{\frac{1}{2}}, \qquad \forall \, x\in\mathbb{R}^{2}.
\end{equation}
Moreover, by the asymptotic behavior \eqref{33} in Corollary \ref{cor2}, one has
\begin{equation}\label{a11}
  \beta:=\frac{1}{2\pi}\int_{\mathbb{R}^{2}}e^{v(x)}\mathrm{d}x=6^{\frac{1}{4}}\frac{1}{\sqrt{\mu}}.
\end{equation}
Combining \eqref{m44} and \eqref{a11} yields that $C'=6^{\frac{1}{6}}$ and hence
\begin{equation}\label{m45}
 v(x)=\frac{3}{2}\ln\left[\frac{6^{\frac{1}{6}}\mu}{1+\mu^2|x-x_{0}|^{2}}\right], \qquad \forall \, x\in\mathbb{R}^{2}.
\end{equation}
In addition, one can verify by calculations that $(u,v)$ given by \eqref{a12} and \eqref{m45} is indeed a pair of solutions to the IE system \eqref{IE} and hence a solution to the PDE system \eqref{PDE}. This concludes our proof of Theorem \ref{thm0}.

\section{Proof of Theorem \ref{thm1}}

In this section, we carry out our proof of Theorem \ref{thm1} and derive the classification of nonnegative classical solutions $(u,v)$ to the system \eqref{PDEH} with mixed order and Hartree type nonlocal nonlinearity in $\mathbb{R}^{3}$.	

\medskip

Through entirely similar arguments as in the proof of Lemma \ref{lem0}, by using Green's functions for $(-\Delta)^{\frac{1}{2}}$ and $-\Delta$ on balls $B_{R}(0)$, one can deduce from the Liouville theorems and maximum principles for $(-\Delta)^{\frac{1}{2}}$ and $-\Delta$ that nonnegative solution $(u,v)$ of PDE system \eqref{PDEH} also solves the following IE system
\begin{equation}\label{6-2}
\begin{cases}u(x)=\frac{1}{2 \pi^{2}} \int_{\mathbb{R}^{3}} \frac{v^{4-\sigma}(y) P(y)}{|x-y|^{2}} d y & \text { in } \mathbb{R}^{3}, \\ \\ v(x)=\frac{1}{4 \pi} \int_{\mathbb{R}^{3}} \frac{u^{\frac{5}{2}}(y)}{|x-y|} d y & \text { in } \mathbb{R}^{3},\end{cases}
\end{equation}	
where $P(y)$ is given by
\begin{equation}\label{6-3}
	P(y):=\int_{\mathbb{R}^3} \frac{v^{6-\sigma}(z)}{|y-z|^\sigma} dz, \qquad \forall \, y\in\mathbb{R}^{3}.
\end{equation}	

The proof of the integral representation formula \eqref{6-2} is quite similar to Lemma \ref{lem0}, so we omit the details here.

\smallskip

By the IE system \eqref{6-2}, it is easy to see that if $u\not\equiv0$ or $v\not\equiv0$, then we must have $u>0$ and $v>0$ in $\mathbb{R}^{3}$. Therefore, in what follows, we will consider positive solutions $(u,v)$ to the IE system \eqref{6-2} instead of the PDE system \eqref{PDEH}.

\smallskip

For arbitrary $x_0\in \mathbb{R}^3$ and $\lambda>0$, we define the Kelvin transforms for $u,v$ and $P$ by
\begin{equation}\label{6-4}
	u_{x_0,\lambda}(x)=\frac{\lambda^{2}}{|x-x_0|^2}u(x^{x_0,\lambda}),  \qquad
	v_{x_0,\lambda}(x)=\frac{\lambda}{|x-x_0|}v(x^{x_0,\lambda}), \qquad \forall \, x\in\mathbb{R}^{3}\setminus\{x_{0}\},
\end{equation}
\begin{equation}\label{6-4'}
  P_{x_0,\lambda}(x)=\frac{\lambda^{\sigma}}{|x-x_0|^\sigma}P(x^{x_0,\lambda}), \qquad \forall \, x\in\mathbb{R}^{3}\setminus\{x_{0}\},
\end{equation}
where $x^{x_0,\lambda}=\frac{\lambda^2(x-x_0)}{|x-x_0|^2}+x_0$. By \eqref{6-2}, \eqref{6-4}, \eqref{6-4'} and straightforward calculations, one can verify that
\begin{equation}\label{6-5}
	\begin{split}
		u(x)&=\frac{1}{2\pi^2}\int_{\mathbb{R}^3} \frac{P(y) v^{4-\sigma}(y)}{|x-y|^{2}}\mathrm{d}y\\
		&=\frac{1}{2\pi^2}\int_{B_\lambda(x_0)} \frac{P(y) v^{4-\sigma}(y)}{|x-y|^{2}}\mathrm{d}y+\frac{1}{2\pi^2}\int_{B_\lambda(x_0)} \frac{P(y^{x_0,\lambda})v^{4-\sigma}(y^{x_0,\lambda})}{{\left|\frac{|y-x_0|(x-x_0)}{\lambda}-\frac{\lambda (y-x_0)}{|y-x_0|}\right|}^{2}} {\left( \frac{\lambda}{|y-x_0|}\right)}^{4} \mathrm{d}y  \\
		&=\frac{1}{2\pi^2}\int_{B_\lambda(x_0)} \frac{P(y) v^{4-\sigma}(y)}{|x-y|^{2}}\mathrm{d}y+\frac{1}{2\pi^2}\int_{B_\lambda(x_0)} \frac{P_{x_0,\lambda}(y)v_{x_0,\lambda}^{4-\sigma}(y)}{{\left|\frac{|y-x_0|(x-x_0)}{\lambda}-\frac{\lambda (y-x_0)}{|y-x_0|}\right|}^{2}}\mathrm{d}y,
	\end{split}
\end{equation}
and
\begin{equation}\label{6-6}
	\begin{split}
		u_{x_0,\lambda}(x)&=\frac{\lambda^{2}}{2\pi^2|x-x_0|^2}\int_{\mathbb{R}^3} \frac{P(y) v^{4-\sigma}(y)}{|x^{x_0,\lambda}-y|^{2}}\mathrm{d}y \\
		&=\frac{1}{2\pi^2}\int_{B_\lambda(x_0)} \frac{P(y) v^{4-\sigma}(y)}{\left|\frac{\lambda (x-x_0)}{|x-x_0|}-\frac{|x-x_0|(y-x_0)}{\lambda}\right|^{2}}\mathrm{d}y \\ &\quad +\frac{1}{2\pi^2}\int_{B_\lambda(x_0)} \frac{P(y^{x_0,\lambda})v^{4-\sigma}(y^{x_0,\lambda})}{|x-y|^{2}} {\left( \frac{\lambda}{|y-x_0|}\right)}^{4} \mathrm{d}y  \\
		&=\frac{1}{2\pi^2}\int_{B_\lambda(x_0)} \frac{P(y) v^{4-\sigma}(y)}{\left|\frac{\lambda (x-x_0)}{|x-x_0|}-\frac{|x-x_0|(y-x_0)}{\lambda}\right|^{2}}\mathrm{d}y+\frac{1}{2\pi^2}\int_{B_\lambda(x_0)} \frac{P_{x_0,\lambda}(y)v_{x_0,\lambda}^{4-\sigma}(y)}{{|x-y|}^{2}}\mathrm{d}y.
	\end{split}
\end{equation}
Similarly, by direct calculations, we have
\begin{equation}\label{6-7}
	v(x)=\frac{1}{4\pi}\int_{B_\lambda(x_0)} \frac{u^{\frac{5}{2}}(y)}{|x-y|}\mathrm{d}y+\frac{1}{4\pi}\int_{B_\lambda(x_0)} \frac{u_{x_0,\lambda}^{\frac{5}{2}}(y)}{{\left|\frac{|y-x_0|(x-x_0)}{\lambda}-\frac{\lambda (y-x_0)}{|y-x_0|}\right|}}\mathrm{d}y,
\end{equation}	
and	
\begin{equation}\label{6-8}
	v_{x_0,\lambda}(x)=\frac{1}{4\pi}\int_{B_\lambda(x_0)} \frac{u^{\frac{5}{2}}(y)}{\left|\frac{\lambda (x-x_0)}{|x-x_0|}-\frac{|x-x_0|(y-x_0)}{\lambda}\right|}\mathrm{d}y+\frac{1}{4\pi}\int_{B_\lambda(x_0)} \frac{u_{x_0,\lambda}^{\frac{5}{2}}(y)}{{|x-y|}}\mathrm{d}y.
\end{equation}	

Define $\omega^u_{x_0,\lambda}(x):=u_{x_0,\lambda}(x)-u(x)$ and $\omega^v_{x_0,\lambda}(x):=v_{x_0,\lambda}(x)-v(x)$ for $x\in \mathbb{R}^{3}\setminus\{x_0\}$. The moving sphere process can be divided into two steps.

\medskip

Step 1. Start moving the sphere $S_{\lambda}(x_{0})$ from near $\lambda=0$.

\smallskip

We will show that, for $\lambda>0$ sufficiently small, there holds
\begin{equation}\label{aim}
	\omega^u_{x_0,\lambda}(x)\geq 0, \quad\, \omega^v_{x_0,\lambda}(x)\geq 0, \qquad \forall \, x\in B_\lambda(x_0)\setminus\{x_0\}.
\end{equation}
Obviously, it suffices to prove that $B_{u,\lambda}^-(x_0)=B_{v,\lambda}^-(x_0)=\emptyset$ for $\lambda>0$ small, where $B_{u,\lambda}^-(x_0):=\left\{x\in B_\lambda(x_0)\setminus\{x_0\}|\, \omega^u_{x_0,\lambda}(x)<0\right\}$ and $B_{v,\lambda}^-(x_0):=\left\{x\in B_\lambda(x_0)\setminus\{x_0\}|\, \omega^v_{x_0,\lambda}(x)<0\right\}$.

\smallskip

By \eqref{6-5} and \eqref{6-6}, for any $x\in B_\lambda(x_0)\setminus\{x_{0}\}$, we derive
\begin{equation}\label{6-9}
	\begin{split}
		&\omega^u_{x_0,\lambda}(x)=u_{x_0,\lambda}(x)-u(x)\\
		=&\frac{1}{2\pi^2}\int_{B_\lambda}\Bigg[\frac{1}{|x-y|^{2}}-\frac{1}{\left|\frac{\lambda (x-x_0)}{|x-x_0|}-\frac{|x-x_0|(y-x_0)}{\lambda}\right|^{2}}\Bigg]\left[P_{x_0,\lambda}(y)v_{x_0,\lambda}^{4-\sigma}(y)-P(y) v^{4-\sigma}(y)\right]\mathrm{d}y \\
		\geq& \frac{1}{2\pi^2}\int_{Q_\lambda} \Bigg[\frac{1}{|x-y|^{2}}-\frac{1}{\left|\frac{\lambda (x-x_0)}{|x-x_0|}-\frac{|x-x_0|(y-x_0)}{\lambda}\right|^{2}}\Bigg]\left[P_{x_0,\lambda}(y)v_{x_0,\lambda}^{4-\sigma}(y)-P(y) v^{4-\sigma}(y)\right]\mathrm{d}y \\
		=&\frac{1}{2\pi^2}\int_{Q_\lambda(x_0)} \Bigg[\frac{1}{|x-y|^{2}}-\frac{1}{\left|\frac{\lambda (x-x_0)}{|x-x_0|}-\frac{|x-x_0|(y-x_0)}{\lambda}\right|^{2}}\Bigg]P(y)\left(v_{x_0,\lambda}^{4-\sigma}(y)- v^{4-\sigma}(y)\right)\mathrm{d}y\\
		& +\frac{1}{2\pi^2}\int_{Q_\lambda(x_0)}\Bigg[\frac{1}{|x-y|^{2}}-\frac{1}{\left|\frac{\lambda (x-x_0)}{|x-x_0|}-\frac{|x-x_0|(y-x_0)}{\lambda}\right|^{2}}\Bigg]\left(P_{x_0,\lambda}(y)- P(y)\right)v_{x_0,\lambda}^{4-\sigma}(y)\mathrm{d}y,
	\end{split}
\end{equation}
where the subset $Q_\lambda(x_0)\subseteq B_\lambda(x_0)\setminus\{x_{0}\}$ is defined by
\begin{equation}\label{6-10}
	Q_\lambda(x_0):=\left\{x\in B_\lambda(x_0)\setminus\{x_{0}\}|\,P_{x_{0},\lambda}(x)v_{x_0,\lambda}^{4-\sigma}(x)<P(x)v^{4-\sigma}(x)\right\}.
\end{equation}
We have used the identity ${\left|\frac{|y-x_0|(x-x_0)}{\lambda}-\frac{\lambda (y-x_0)}{|y-x_0|}\right|}=\left|\frac{\lambda (x-x_0)}{|x-x_0|}-\frac{|x-x_0|(y-x_0)}{\lambda}\right|$ and the fact that $|x-y|<\left|\frac{\lambda (x-x_0)}{|x-x_0|}-\frac{|x-x_0|(y-x_0)}{\lambda}\right|$ for any $x,y\in B_\lambda(x_0)\setminus\{x_{0}\}$.

\smallskip

Through direct calculations, one has, for any $y\in B_\lambda(x_0)\setminus\{x_{0}\}$,
\begin{equation}\label{6-11}
	\begin{split}
		&\quad P_{x_0,\lambda}(y)-P(y)\\
		&=\int_{B_\lambda(x_0)}\Bigg[\frac{1}{|y-z|^{\sigma}}-\frac{1}{\left|\frac{\lambda (y-x_0)}{|y- x_0|}-\frac{|y-x_0|(z-x_0)}{\lambda}\right|^{\sigma}}\Bigg]\left(v_{x_0,\lambda}^{6-\sigma}(z)- v^{6-\sigma}(z)\right)dz.
	\end{split}
\end{equation}
Then, we infer from \eqref{6-9}, \eqref{6-11} and mean value theorem that, for any $x\in B_{u,\lambda}^{-}(x_0)$,
\begin{equation}\label{6-12}
	\begin{split}
		0&>\omega^u_{x_0,\lambda}(x)\\
		&\geq \frac{1}{2\pi^2}\int_{Q_\lambda(x_0)} \Bigg[\frac{1}{|x-y|^{2}}-\frac{1}{\left|\frac{\lambda (x-x_0)}{|x-x_0|}-\frac{|x-x_0|(y-x_0)}{\lambda}\right|^{2}}\Bigg]P(y)\left(v_{x_0,\lambda}^{4-\sigma}(y)-v^{4-\sigma}(y)\right)\mathrm{d}y\\
		&\quad +\frac{1}{2\pi^2}\int_{Q_\lambda(x_0)} \Bigg[\frac{1}{|x-y|^{2}}-\frac{1}{\left|\frac{\lambda (x-x_0)}{|x-x_0|}-\frac{|x-x_0|(y-x_0)}{\lambda}\right|^{2}}\Bigg]\left(P_{x_0,\lambda}(y)- P(y)\right)v_{x_0,\lambda}^{4-\sigma}(y)\mathrm{d}y\\
		&\geq \frac{4-\sigma}{2\pi^2}\int_{Q_\lambda\cap B_{v,\lambda}^-(x_0)}\Bigg[\frac{1}{|x-y|^{2}}-\frac{1}{\left|\frac{\lambda (x-x_0)}{|x-x_0|}-\frac{|x-x_0|(y-x_0)}{\lambda}\right|^{2}}\Bigg]P(y)v^{3-\sigma}(y)\omega^v_{x_0,\lambda}(y)\mathrm{d}y \\
		&\quad +\frac{6-\sigma}{2\pi^2}\int_{Q_\lambda(x_0)}\Bigg[\frac{1}{|x-y|^{2}}-\frac{1}{\left|\frac{\lambda (x-x_0)}{|x-x_0|}-\frac{|x-x_0|(y-x_0)}{\lambda}\right|^{2}}\Bigg]v_{x_0,\lambda}^{4-\sigma}(y)\\
		&\quad\quad \times \Bigg[\int_{B_{v,\lambda}^-(x_0)}\Bigg(\frac{1}{|y-z|^{\sigma}}-\frac{1}{\left|\frac{\lambda (y-x_0)}{|y-x_0|}-\frac{|y-x_0|(z-x_0)}{\lambda}\right|^{\sigma}}\Bigg)v^{5-\sigma}(z)\omega^v_{x_0,\lambda}(z)dz\Bigg]\mathrm{d}y\\
		&\geq \int_{Q_\lambda(x_0)\cap B_{v,\lambda}^-(x_0)}\frac{(4-\sigma)}{2\pi^2|x-y|^{2}}P(y)v^{3-\sigma}(y)\omega^v_{x_0,\lambda}(y)\mathrm{d}y \\
		&\quad +\int_{Q_\lambda(x_0)}\frac{(6-\sigma)v_{x_0,\lambda}^{4-\sigma}(y)}{2\pi^2|x-y|^{2}} \left[\int_{B_{v,\lambda}^-(x_0)}\frac{v^{5-\sigma}(z)\omega^v_{x_0,\lambda}(z)}{|y-z|^{\sigma}}dz\right]\mathrm{d}y
	\end{split}
\end{equation}
Similarly, by \eqref{6-7}, \eqref{6-8} and mean value theorem, one can deduce that for any $x\in B_{v,\lambda}^{-}(x_0)$,
\begin{equation}\label{6-13}
	\begin{split}
		0&>\omega^v_{x_{0},\lambda}(x)\\
		&\geq \frac{5}{8\pi}\int_{B_{\lambda}(x_0)}\Bigg[\frac{1}{|x-y|}- \frac{1}{\left|\frac{\lambda (x-x_0)}{|x-x_0|}-\frac{|x-x_0|(y-x_0)}{\lambda}\right|}\Bigg]u^{\frac{3}{2}}(y)\omega_{x_0,\lambda}^{u}(y)\mathrm{d}y \\
		&\geq \frac{5}{8\pi}\int_{B^-_{u,\lambda}(x_0)} \frac{u^{\frac{3}{2}}(y)\omega_{x_0,\lambda}^{u}(y)}{|x-y|} \mathrm{d}y.
	\end{split}
\end{equation}

Define $W_{x_0,\lambda}(y):=\int_{B_{v,\lambda}^-(x_0)}\frac{v^{5-\sigma}(z)\omega^v_{x_0,\lambda}(z)}{|y-z|^{\sigma}}dz$ for any $y\in Q_{\lambda}(x_{0})$. By \eqref{6-12}, \eqref{6-13}, Hardy-Littlewood-Sobolev inequality and H\"{o}lder inequality, we have, for any $q>\max\left\{3,\frac{3}{\sigma}\right\}$, it holds that
\begin{equation}\label{6-15}
	\begin{split}
		&\quad \|\omega^u_{x_0,\lambda}\|_{L^{q}(B_{u,\lambda}^-(x_0))} \\
		&\leq C{\left\|P(y)v^{3-\sigma}(y)\omega^v_{x_0,\lambda}(y)\right\|}_{L^{\frac{3q}{3+q}}\left(Q_{\lambda}(x_{0})\cap B_{v,\lambda}^-(x_0)\right)}+C{\left\|W_{x_0,\lambda} v_{x_0,\lambda}^{4-\sigma}\right\|}_{L^{\frac{3q}{3+q}}(Q_\lambda(x_0))}\\
		&\leq C{\left\|P v^{3-\sigma}\right\|}_{L^{3}( B_{v,\lambda}^-(x_0)) } {\left\|\omega^v_{x_0,\lambda}\right\|}_{L^{q}( B_{v,\lambda}^-(x_0))}+C{\left\|v_{x_0,\lambda}^{4-\sigma}\right\|}_{L^{3}(Q_\lambda(x_0))} {\left\|W_{x_0,\lambda} \right\|}_{L^{q}(Q_\lambda(x_0))}\\
		&\leq C{\left\|P v^{3-\sigma}\right\|}_{L^{3}( B_{v,\lambda}^-(x_0)) } {\left\|\omega^v_{x_0,\lambda}\right\|}_{L^{q}( B_{v,\lambda}^{-})}+C{\left\|v_{x_0,\lambda}^{4-\sigma}\right\|}_{L^{3}(Q_\lambda(x_0))} {\left\|v^{5-\sigma} \omega^v_{x_0,\lambda} \right\|}_{L^{\frac{3q}{3+(3-\sigma)q}}(B_{v,\lambda}^{-})}\\
		&\leq C{\left\|P v^{3-\sigma}\right\|}_{L^{3}( B_{v,\lambda}^-(x_0)) } {\left\|\omega^v_{x_0,\lambda}\right\|}_{L^{q}( B_{v,\lambda}^-(x_0))} \\
&\quad +C{\left\|v_{x_0,\lambda}^{4-\sigma}\right\|}_{L^{3}(Q_\lambda(x_0))} {\left\|v^{5-\sigma} \right\|}_{L^{\frac{3}{3-\sigma}}(B_{v,\lambda}^-(x_0))} {\left\|\omega^v_{x_0,\lambda} \right\|}_{L^{q}(B_{v,\lambda}^-(x_0))}\\
		&\leq C_1\left({\left\|P v^{3-\sigma}\right\|}_{L^{3}( B_{v,\lambda}^-(x_0)) } +{\left\|v_{x_0,\lambda}^{4-\sigma}\right\|}_{L^{3}( Q_\lambda(x_0)) } {\left\|v^{5-\sigma} \right\|}_{L^{\frac{3}{3-\sigma}}(B_{v,\lambda}^-(x_0))}\right){\left\|\omega^v_{x_0,\lambda}\right\|}_{L^{q}( B_{v,\lambda}^-(x_0))},
	\end{split}
\end{equation}
and
\begin{equation}\label{6-16}
	\begin{split}
		\|\omega^v_{x_0,\lambda}\|_{L^{q}(B_{v,\lambda}^-(x_0))} \leq C{\left\|u^{\frac{3}{2}}\omega_{x_0,\lambda} ^u\right\|}_{L^{\frac{3q}{3+2q}}( B_{u,\lambda}^-(x_0))}\leq C_2{\left\|u^{\frac{3}{2}}\right\|}_{L^{\frac{3}{2}}(B_{u,\lambda}^-(x_0))}{\left\|\omega^u_{x_0,\lambda} \right\|}_{L^{q}( B_{u,\lambda}^-(x_0))}.
	\end{split}
\end{equation}

In order to continue, we need the upper bound of $v_{x_0,\lambda}^{4-\sigma}$ in $ Q_\lambda(x_0)$, which will be given by the lower bound of $P_{x_0,\lambda}$ in view of the definition of $ Q_\lambda(x_0)$. From the definition of $P(y)$, one can get the following lower bound:
\begin{equation}\label{6-17}
	P(y)\geq \frac{C}{|y-x_{0}|^{\sigma}}\int_{|z-x_{0}|\leq\frac{1}{2}}v^{6-\sigma}(z)dz=:\frac{C}{|y-x_{0}|^{\sigma}}, \qquad \forall \,\, |y-x_{0}|\geq1,
\end{equation}
where $C>0$ is a positive constant depending on $x_0$ and $v$. Thanks to  \eqref{6-17}, we can obtain the following lower bounds for $P_{x_0,\lambda}$ in $B_\lambda(x_0)\setminus\{x_0\}$. We first consider $0<\lambda<1$. Note that if $0<|y-x_0|\leq\lambda^2$, then $|y^{x_{0},\lambda}-x_0|\geq1$. Then we infer from \eqref{6-17} that
\begin{equation}\label{6-18}
	P_{x_0,\lambda}(y)={\left(\frac{\lambda}{|y-x_0|}\right)}^{\sigma}P(y^{x_0,\lambda})\geq \frac{C}{\lambda^{\sigma}}, \qquad \forall \,\, 0<|y-x_0|\leq\lambda^2.
\end{equation}
Now suppose $\lambda^2\leq|y-x_0|<\lambda$, then $|y^{x_{0},\lambda}-x_0|\leq1$, and we have
\begin{equation}\label{6-19}
	P_{x_0,\lambda}(y)={\left(\frac{\lambda}{|y-x_0|}\right)}^{\sigma}P(y^{x_0,\lambda})\geq\left(\min_{\overline{B_1(x_0)}} P\right){\left(\frac{\lambda}{|y-x_0|}\right)}^{\sigma}\geq\min_{\overline{B_1(x_0)}}P=:C>0.
\end{equation}
For $\lambda\geq1$, it holds $|y^{x_{0},\lambda}-x_0|>\lambda\geq1$ for any $y\in B_\lambda(x_0)\setminus\{x_0\}$, and hence \eqref{6-17} yields
\begin{equation}\label{6-20}
	P_{x_0,\lambda}(y)={\left(\frac{\lambda}{|y-x_0|}\right)}^{\sigma}P(y^{x_{0},\lambda})\geq\frac{C}{\lambda^{\sigma}}, \qquad \forall \,\, y\in B_\lambda(x_0)\setminus\{x_0\}.
\end{equation}

By the definition of $Q_{\lambda}(x_{0})$, the continuity of $v$ and $P$ in $\mathbb{R}^{3}$, the lower bound estimates \eqref{6-18}, \eqref{6-19} and \eqref{6-20}, we obtain the following upper bounds: for any $x\in Q_{\lambda}(x_{0})$,
\begin{equation}\label{6-21}
	v_{x_0,\lambda}^{4-\sigma}(x)<\frac{P(x)v^{4-\sigma}(x)}{P_{x_{0},\lambda}(x)}\leq C(1+\lambda^{\sigma})\left(\max_{\overline{B_{1}(x_{0})}}Pv^{4-\sigma}\right)=:C(1+\lambda^{\sigma}), \qquad \text{if} \,\, 0<\lambda<1;
\end{equation}
\begin{equation}\label{6-22}
	v_{x_0,\lambda}^{4-\sigma}(x)<\frac{P(x)v^{4-\sigma}(x)}{P_{x_{0},\lambda}(x)}\leq C(1+\lambda^{\sigma})\left(\max_{\overline{B_{\lambda}(x_{0})}}Pv^{4-\sigma}\right)=:C_{\lambda}(1+\lambda^{\sigma}), \qquad \text{if} \,\, \lambda\geq1.
\end{equation}
From the local boundedness of $u$, $v$ and $P$ in $\mathbb{R}^{3}$, the upper bound \eqref{6-21} of $v_{x_0,\lambda}^{4-\sigma}$ in $Q_{\lambda}(x_{0})$, we deduce that
\begin{equation}\label{6-23}
	\left({\left\|P v^{3-\sigma}\right\|}_{L^{3}( B_{v,\lambda}^-(x_0)) } +{\left\|v_{x_0,\lambda}^{4-\sigma}\right\|}_{L^{3}( Q_\lambda(x_0)) } {\left\|v^{5-\sigma} \right\|}_{L^{\frac{3}{3-\sigma}}(B_{v,\lambda}^-(x_0))}\right){\left\|u^{\frac{3}{2}}\right\|}_{L^{\frac{3}{2}}(B_{u,\lambda}^-(x_0))}\rightarrow 0,
\end{equation}
as $\lambda\rightarrow0$. Therefore, there exists $0<\epsilon_0<1$ small enough such that, for all $0<\lambda<\epsilon_0$,
\begin{equation}\label{6-24}
	C_1 C_2\left[{\left\|P v^{3-\sigma}\right\|}_{L^{3}( B_{v,\lambda}^-(x_0)) } +{\left\|v_{x_0,\lambda}^{4-\sigma}\right\|}_{L^{3}( Q_\lambda(x_0)) } {\left\|v^{5-\sigma} \right\|}_{L^{\frac{3}{3-\sigma}}(B_{v,\lambda}^{-})}\right]{\left\|u^{\frac{3}{2}}\right\|}_{L^{\frac{3}{2}}(B_{u,\lambda}^{-})}<\frac{1}{2},
\end{equation}
where $C_1$ and $C_2$ are the constants in the last inequality of \eqref{6-15} and \eqref{6-16} respectively. Immediately, we conclude from \eqref{6-15} and \eqref{6-16} that, for any $0<\lambda<\epsilon_0$,
\begin{equation}\label{a14}
  \|\omega^u_{x_0,\lambda}\|_{L^{q}(B_{u,\lambda}^{-}(x_0))}=\|\omega^v_{x_0,\lambda}\|_{L^{q}(B_{v,\lambda}^{-}(x_0))}=0,
\end{equation}
which implies that both $B_{u,\lambda}^-(x_0)$ and $B_{v,\lambda}^-(x_0)$ have measure $0$. Then,  due to the continuity of $\omega^u_{x_0,\lambda}$ and $\omega^v_{x_0,\lambda}$ in $B_\lambda(x_0)\setminus\{x_0\}$, we must have $B_{u,\lambda}^-(x_0)=B_{v,\lambda}^-(x_0)=\emptyset$ for any $0<\lambda<\epsilon_0$. Hence we have derived \eqref{aim} for all $0<\lambda<\epsilon_0$.

\medskip

Step 2. Moving the sphere $S_{\lambda}$ outward until the limiting position.

\smallskip

Step 1 provides a starting point to carry out the method of moving spheres for arbitrarily given center $x_0\in \mathbb{R}^{3}$. Next, we will continuously increase the radius $\lambda$ as long as \eqref{aim} holds. For arbitrarily given center $x_0$, the critical scale $\lambda_{x_0}$ is defined by
\begin{equation}\label{6-25}
	\lambda_{x_0}:=\sup\left\{\lambda>0\mid\,\omega^u_{x_0,\mu}\geq 0 \,\, \text{and}\,\,\omega^v_{x_0,\mu}\geq 0 \,\,\text{in} \,\, B_\mu(x_0)\setminus\{x_0\}, \,\, \forall \,0<\mu\leq\lambda\right\}.
\end{equation}
Then it follows from Step 1 that $0<\lambda_{x_{0}}\leq+\infty$ for any $x_{0}\in\mathbb{R}^{3}$.

\smallskip

It seems that the critical scale $\lambda_{x_0}$ may depend on the center $x_0$. However, we have the following crucial Lemma on the synchronism of $\lambda_{x_{0}}$ with respect to $x_{0}$.
\begin{lem} \label{lem6-1}
One of the following two assertions holds, that is, either\\
\emph{(A)} For every $x_0\in\mathbb{R}^{3}$, the corresponding critical scale $0<\lambda_{x_0}<+\infty$, \\ or \\
\emph{(B)} For every $x_0\in\mathbb{R}^{3}$, the corresponding critical scale $\lambda_{x_0}=+\infty$.
\end{lem}

In order to prove Lemma \ref{lem6-1}, we need the following proposition.
\begin{prop}\label{prop6-2}
	If $\lambda_{x_0}<+\infty$, then $u_{x_0,\lambda_{x_0}}(x)\equiv u(x)$ and $v_{x_0,\lambda_{x_0}}(x)\equiv v(x)$ for any $x\in B_{\lambda_{x_0}}(x_0) \setminus\{x_0\}$.
\end{prop}
\begin{proof}[Proof of Proposition \ref{prop6-2}]
We will prove Proposition \ref{prop6-2} by contradiction arguments. By the definition of $\lambda_{x_0}$, we know that $\omega^u_{x_{0},\lambda_{x_0}}\geq 0$ and $\omega^v_{x_{0},\lambda_{x_0}}\geq 0$ in $B_{\lambda_{x_0}}(x_0)\setminus\{x_0\}$.

Suppose on the contrary that Proposition \ref{prop6-2} is false. We first assume that $\omega^u_{x_0, \lambda_{x_0}}\not\equiv 0$ in $B_{\lambda_{x_0}}(x_0) \setminus\{x_0\}$, then there exists $\bar{x}\in B_{\lambda_{x_0}}(x_0) \setminus\{x_0\}$ and $\delta>0$ such that $B_\delta(\bar x)\subset B_{\lambda_{x_0}}(x_0) \setminus\{x_0\}$ and $\omega^u_{x_{0},\lambda_{x_0}}>C>0$ in $B_\delta(\bar x)$. Then, we derive by \eqref{6-13} that, for any $x\in B_{\lambda_{x_0}}(x_0)\setminus\{x_0\}$,
	\begin{equation}\label{6-26}
		\begin{split}
			\omega^v_{x_0, \lambda_{x_0}}(x)&\geq \frac{5}{8\pi}\int_{B_{\lambda_{x_0}}(x_0)}\left( \frac{1}{|x-y|}- \frac{1}{\left|\frac{\lambda_{x_0} (x-x_0)}{|x-x_0|}-\frac{|x-x_0|(y-x_0)}{\lambda_{x_0}}\right|}\right)u^{\frac{3}{2}}(y)\omega_{x_0, \lambda_{x_0}}^u(y)\mathrm{d}y \\
			&\geq \frac{5 }{8\pi}\int_{B_{\delta}(\bar x)}\left( \frac{1}{|x-y|}- \frac{1}{\left|\frac{\lambda_{x_0} (x-x_0)}{|x-x_0|}-\frac{|x-x_0|(y-x_0)}{\lambda_{x_0}}\right|}\right)u^{\frac{3}{2}}(y)\omega_{x_0, \lambda_{x_0}}^u(y)\mathrm{d}y \\
			&>0.
		\end{split}
	\end{equation}
	Combining this with \eqref{6-11} and \eqref{6-12} imply that
	\begin{equation}\label{6-27}
		\omega^u_{x_0, \lambda_{x_0}}(x)>0, \qquad \forall \, x\in B_{\lambda_{x_0}}(x_0)\setminus\{x_0\}.
	\end{equation}
	Similarly, if $\omega^v_{x_0,\lambda_{x_0}}\not\equiv 0$ one can also deduce that  from  \eqref{6-11}, \eqref{6-12} and \eqref{6-13} that
\begin{equation}\label{a15}
  \omega^u_{x_0, \lambda_{x_0}}(x)>0, \quad\, \omega^v_{x_0, \lambda_{x_0}}(x)>0, \qquad \forall \, x\in B_{\lambda_{x_0}}(x_0)\setminus\{x_0\}.
\end{equation}

	Let $\eta_1>0$ be sufficiently small, which will be determined later. Define the narrow region
	\begin{equation}\label{6-28}
		A_{\eta_1}:=\{x\in \mathbb{R}^3 \,\mid\,0<|x-x_0|<\eta_1 \,\,\, \text{or} \,\,\, \lambda_{x_0}-\eta_1<|x-x_0|<\lambda_{x_0}\}.
	\end{equation}	
	Since $\omega^u_{x_{0},\lambda_{x_0}}$ and $\omega^v_{x_{0},\lambda_{x_0}}$ are continuous in $\mathbb{R}^{3}\setminus\{x_0\}$ and $A_{\eta_1}^c:=\left(B_{\lambda_{x_0}}(x_0)\setminus\{x_{0}\}\right)\setminus A_{\eta_1}$ is a compact subset, there exists a $C_0>0$ such that
	\begin{equation}\label{4-11}
		\omega^u_{x_0,\lambda_{x_0}}(x)>C_0, \,\quad \omega^v_{x_0, \lambda_{x_0}}(x)>C_0, \quad\quad \forall \,\, x\in A_{\eta_1}^c.
	\end{equation}
	By continuity, we can choose $\eta_2>0$ (depending on $\eta_{1}$) sufficiently small such that, for any $\lambda\in[\lambda_{x_0},\,\lambda_{x_0}+\eta_2]$,
	\begin{equation}\label{6-29}
		\omega^u_{x_0,\lambda}(x)>\frac{C_0}{2}, \quad\, \omega^v_{x_0,\lambda}(x)>\frac{C_0}{2},\quad\quad \forall \,\, x\in A_{\eta_1}^c.
	\end{equation}
	Hence we must have, for any $\lambda\in [\lambda_{x_{0}},\,\lambda_{x_{0}}+\eta_2]$,
	\begin{equation}\label{6-30}
\begin{aligned}
		B_{u,\lambda}^-(x_0)\cup B_{v,\lambda}^-(x_0)&\subset \left(B_{\lambda}(x_0)\setminus\{x_{0}\}\right)\setminus A^{c}_{\eta_1} \\
&=\left\{x\in\mathbb{R}^{3}\,\mid\,0<|x-x_0|<\eta_1 \,\, \text{or} \,\, \lambda_{x_0}-\eta_1<|x-x_0|<\lambda\right\}.
\end{aligned}
	\end{equation}
	By \eqref{6-30}, the continuity of $u$, $v$ and $P$, the the upper bounds \eqref{6-21} and \eqref{6-22} of $v_{x_{0},\lambda}^{4-\sigma}$ in $B_{\lambda}(x_0)\setminus\{x_0\}$, we can choose $\eta_1$ sufficiently small (and $\eta_2$ more smaller if necessary) such that, for any $\lambda\in[\lambda_{x_{0}},\,\lambda_{x_{0}}+\eta_2]$,
	\begin{equation}\label{6-31}
		C_1 C_2\left[{\left\|P v^{3-\sigma}\right\|}_{L^{3}( B_{v,\lambda}^-(x_0)) } +{\left\|v_{x_0,\lambda}^{4-\sigma}\right\|}_{L^{3}( Q_\lambda(x_0)) } {\left\|v^{5-\sigma} \right\|}_{L^{\frac{3}{3-\sigma}}(B_{v,\lambda}^{-})}\right]{\left\|u^{\frac{3}{2}}\right\|}_{L^{\frac{3}{2}}(B_{u,\lambda}^-(x_0))}<\frac{1}{2},
	\end{equation}
	where $C_1$ and $C_2$ are the constants in the last inequality of \eqref{6-15} and \eqref{6-16} respectively. Immediately, we conclude from \eqref{6-15} and \eqref{6-16} that
\begin{equation}\label{a16}
  \|\omega^u_{x_0, \lambda}\|_{L^{q}(B_{u,\lambda}^{-}(x_0))}=\|\omega^v_{x_0, \lambda}\|_{L^{q}(B_{v,\lambda}^{-}(x_0))}=0,
\end{equation}
and hence $B_{u, \lambda}^-(x_0)=B_{v, \lambda}^-(x_0)=\emptyset$ for any $\lambda\in[\lambda_{x_{0}},\,\lambda_{x_{0}}+\eta_2]$. Consequently, we obtain that, for any $\lambda\in[\lambda_{x_{0}},\,\lambda_{x_{0}}+\eta_2]$,
	\begin{equation}\label{6-32}
		\omega^u_{x_0, \lambda}(x)\geq 0, \quad\, \omega^v_{x_0,\lambda}(x)\geq 0, \quad\quad \forall \,\, x\in B_{\lambda}(x_0)\setminus\{x_0\},
	\end{equation}
	which contradicts the definition of the critical scale $\lambda_{x_{0}}$. This completes our proof of Proposition \ref{prop6-2}.	
\end{proof}

Now we are ready to prove Lemma \ref{lem6-1}.
\begin{proof}[Proof of Lemma \ref{lem6-1}]
	 If there exists a $x_0\in \mathbb{R}^{3}$ such that $\lambda_{x_0}=+\infty$, we have, for any $\lambda>0$,
	\begin{equation}\label{6-33}
		\omega^{u}_{x_0, \lambda}(x)\geq 0, \qquad \forall \,\, x\in B_\lambda(x_0)\setminus\{x_0\},
	\end{equation}
	which implies that, for any $\lambda>0$,
	\begin{equation}\label{6-34}
		u(x)\geq u_{x_{0},\lambda}(x), \qquad \forall \,\, |x-x_0|>\lambda.
	\end{equation}
	Then, due to the arbitrariness of $\lambda>0$, \eqref{6-34} yields that
	\begin{equation}\label{6-35}
		\lim_{|x|\rightarrow+\infty}|x|^{2}u(x)=+\infty.
	\end{equation}
	However, if we assume there exists another point $z_0\in \mathbb{R}^{3}$ such that $\lambda_{z_0}<+\infty$, then by Proposition \ref{prop6-2}, we have
	\begin{equation}\label{6-36}
		u_{z_{0},\lambda_{z_0}}(x)=u(x), \qquad \forall \,\, x\in\mathbb{R}^3\setminus\{z_0\}.
	\end{equation}
	The above identity immediately implies that
	\begin{equation}\label{6-37}
		\lim_{|x|\rightarrow+\infty}|x|^{2}u(x)=\lambda_{z_{0}}^{2}u(z_{0})<+\infty,
	\end{equation}
	which contradicts \eqref{6-35}. This finishes our proof of Lemma \ref{lem6-1}.
\end{proof}

Now we are to complete our proof of Theorem \ref{thm1}. Suppose that, for every $x_0\in\mathbb{R}^{3}$, the critical scale $\lambda_{x_0}=+\infty$. Then, by conclusion (ii) in Lemma \ref{lem10} and the integrability $\int_{\mathbb{R}^{3}}\frac{v^{6-\sigma}(x)}{|x|^{\sigma}}\mathrm{d}y<+\infty$, we conclude that $v\equiv 0$ and hence $u\equiv 0$, which contradicts $u>0$ and $v>0$ in $\mathbb{R}^{3}$. Therefore, we must have, for every $x_0\in\mathbb{R}^3$, the critical scale $\lambda_{x_0}<+\infty$. Then by Proposition \ref{prop6-2} and conclusion (i) in Lemma \ref{lem10}, we have
\begin{equation}\label{6-38}
	u(x)=C\frac{\mu}{1+\mu^{2}|x-\bar{x}|^2}, \qquad 	v(x)=C'{\left(\frac{\mu}{1+\mu^{2}|x-\bar{x}|^2}\right)}^{\frac{1}{2}}
\end{equation}
for some $C$, $C'$, $\mu>0$ and $\bar{x}\in\mathbb{R}^3$. Consequently, we have proved  that any nontrivial nonnegative solutions $(u,v)$ to IE system \eqref{6-2} and PDE system \eqref{PDEH} must have the form \eqref{6-38} for some $C>0$, $C'>0$, $\mu>0$ and $\bar{x}\in\mathbb{R}^3$.

\smallskip

Furthermore, from (37) in Lemma 4.1 in \cite{DFHQW}, we have the following integral formula:
\begin{equation}\label{formula}
	\int_{\mathbb{R}^{n}}\frac{1}{|x-y|^{2\gamma}}\left(\frac{1}{1+|y|^{2}}\right)^{n-\gamma}\mathrm{d}y=I(\gamma)\left(\frac{1}{1+|x|^{2}}\right)^{\gamma}
\end{equation}
for any $n\geq1$ and $0<\gamma<\frac{n}{2}$, where $I(\gamma):=\frac{\pi^{\frac{n}{2}}\Gamma\left(\frac{n-2\gamma}{2}\right)}{\Gamma(n-\gamma)}$. By the IE system \eqref{6-2}, the integral formula \eqref{formula} with $n=3$ and direct calculations, we deduce that the constants $C$ and $C'$ in \eqref{6-38} can be calculated accurately by
\begin{equation}\label{6-39}
	C=\left(\frac{2\times 3^{2(5-\sigma)}}{I\left(\frac{\sigma}{2}\right)}\right)^{\frac{1}{24-5\sigma}} \qquad \text{and} \qquad C'=\left(\frac{3\times 2^{\frac{5}{2}}}{I\left(\frac{\sigma}{2}\right)^{\frac{5}{2}}}\right)^{\frac{1}{24-5\sigma}},
\end{equation}
where $I\left(\frac{\sigma}{2}\right)=\frac{\pi^{\frac{3}{2}}\Gamma\left(\frac{3-\sigma}{2}\right)}{\Gamma\left(3-\frac{\sigma}{2}\right)}$. This concludes our proof of Theorem \ref{thm1}.



\begin{thebibliography}{99}

\bibitem{AGSY} C. O. Alves, F. Gao, M. Squassina and M. Yang, {\it Singularly perturbed critical Choquard equations}, J. Differential Equations, \textbf{263} (2017), no. 7, 3943-3988.

\bibitem{Be} J. Bertoin, {\it L\'{e}vy Processes}, Cambridge Tracts in Mathematics, \textbf{121}, Cambridge University Press, Cambridge, 1996.

\bibitem{BKN} K. Bogdan, T. Kulczycki and A. Nowak, {\it Gradient estimates for harmonic and $q$-harmonic functions of symmetric stable processes}, Illinois J. Math., \textbf{46} (2002), 541-556.

\bibitem{Branson} T. P. Branson, {\it Group representations arising from Lorentz conformal geometry}, J. Funct. Anal., \textbf{74} (1987), 199-291.

\bibitem{Branson1} T. P. Branson, {\it Sharp inequality, the functional determinant and the complementary series}, Trans. Amer. Math. Soc., \textbf{347} (1995), 3671-3742.

\bibitem{BF} H. Brezis and F. Merle, {\it Uniform estimates and blow-up behavior for solutions of $-\Delta u=V(x)e^{u}$ in two dimensions}, Commun. Partial Differential Equations, \textbf{16} (1991), 1223-1253.

\bibitem{CT} X. Cabr\'{e} and J. Tan, {\it Positive solutions of nonlinear problems involving the square root of the Laplacian}, Adv. Math., \textbf{224} (2010), 2052-2093.

\bibitem{CGS} L. Caffarelli, B. Gidas and J. Spruck, {\it Asymptotic symmetry and local behavior of semilinear elliptic equations with critical Sobolev growth}, Comm. Pure Appl. Math., \textbf{42} (1989), 271-297.

\bibitem{CS} L. Caffarelli and L. Silvestre, {\it An extension problem related to the fractional Laplacian}, Commun. Partial Differential Equations, \textbf{32} (2007), no. 7-9, 1245-1260.

\bibitem{CV} L. Caffarelli and L. Vasseur, {\it Drift diffusion equations with fractional diffusion and the quasi-geostrophic equation}, Annals of Math., \textbf{171} (2010), no. 3, 1903-1930.

\bibitem{C} D. Cao, {\it Nontrivial solution of semilinear elliptic equation with critical exponent in $\mathbb{R}^{2}$}, Comm. Partial Differential Equations, \textbf{17} (1992), no. 3-4, 407-435.

\bibitem{CD} D. Cao and W. Dai, {\it Classification of nonnegative solutions to a bi-harmonic equation with Hartree type nonlinearity}, Proc. Royal Soc. Edinburgh-A: Math., \textbf{149} (2019), 979-994.

\bibitem{CDQ0} D. Cao, W. Dai and G. Qin, {\it Super poly-harmonic properties, Liouville theorems and classification of nonnegative solutions to equations involving higher-order fractional Laplacians}, Trans. Amer. Math. Soc., \textbf{374} (2021), no. 7, 4781-4813.

\bibitem{CDZ} D. Cao, W. Dai and Y. Zhang, {\it Existence and symmetry of solutions to 2-D Schr\"{o}dinger-Newton equations}, Dyn. Partial Differ. Equ., \textbf{18} (2021), no. 2, 113-156.

\bibitem{CL3} D. Cao and H. Li, {\it High energy solutions of the Choquard equation}, Disc. Cont. Dyn. Syst. - A, \textbf{38} (2018), no. 6, 3023-3032.

\bibitem{CQ} D. Cao and G. Qin, {\it Liouville type theorems for fractional and higher-order fractional systems}, Discrete Contin. Dyn. Syst.-A, \textbf{41} (2021), no. 5, 2269-2283.

\bibitem{CC} J. Case and S.-Y. A. Chang, {\it On fractional GJMS operators}, Comm. Pure Appl. Math., \textbf{69} (2016), no. 6, 1017-1061.

\bibitem{CG} S.-Y. A. Chang and M. Gonz\'{a}lez, {\it Fractional Laplacian in conformal geometry}, Adv. Math., \textbf{226} (2011), no. 2, 1410-1432.

\bibitem{CY} S.-Y. A. Chang and P. C. Yang, {\it On uniqueness of solutions of $n$-th order differential equations in conformal geometry}, Math. Res. Lett., \textbf{4} (1997), 91-102.

\bibitem{CDQ} W. Chen, W. Dai and G. Qin, {\it Liouville type theorems, a priori estimates and existence of solutions for critical order Hardy-H\'{e}non equations in $\mathbb{R}^n$}, preprint, submitted for publication, arXiv: 1808.06609.

\bibitem{CL} W. Chen and C. Li, {\it Classification of solutions of some nonlinear elliptic equations}, Duke Math. J., \textbf{63} (1991), no. 3, 615-622.

\bibitem{CL1} W. Chen and C. Li, {\it On Nirenberg and related problems - a necessary and sufficient condition}, Comm. Pure Appl. Math., \textbf{48} (1995), 657-667.

\bibitem{CL0} W. Chen and C. Li, {\it Moving planes, moving spheres, and a priori estimates}, J. Differential Equations, \textbf{195} (2003), no. 1, 1-13.

\bibitem{CL2} W. Chen and C. Li, {\it Methods on nonlinear elliptic equations}, AIMS Series on Differential Equations \& Dynamical Systems, \textbf{4}, American Institute of Mathematical Sciences (AIMS), Springfield, MO, 2010, xii+299 pp, ISBN: 978-1-60133-006-2; 1-60133-006-5.

\bibitem{CLL} W. Chen, C. Li and Y. Li, {\it A direct method of moving planes for the fractional Laplacian}, Adv. Math., \textbf{308} (2017), 404-437.

\bibitem{CLO} W. Chen, C. Li and B. Ou, {\it Classification of solutions for an integral equation}, Comm. Pure Appl. Math., \textbf{59} (2006), 330-343.

\bibitem{CLM} W. Chen, Y. Li and P. Ma, {\it The fractional Laplacian}, World Scientific Publishing Co. Pte. Ltd., Hackensack, NJ, [2020], \copyright 2020, 331 pp, ISBN: [9789813223998]; [9789813224001]; [9789813224018].

\bibitem{CLZ} W. Chen, Y. Li and R. Zhang, {\it A direct method of moving spheres on fractional order equations}, J. Funct. Anal., \textbf{272} (2017), no. 10, 4131-4157.

\bibitem{CW} S. Cingolani and T. Weth, {\it On the planar Schr\"{o}dinger-Poisson system}, Ann. Inst. H. Poincar\'{e} Anal. Non Lin\'{e}aire, \textbf{33} (2016), no. 1, 169-197.

\bibitem{Co} P. Constantin, {\it Euler equations, Navier-Stokes equations and turbulence, in Mathematical Foundation of Turbulent Viscous Flows}, Vol. 1871 of Lecture Notes in Math., 1-43, Springer, Berlin, 2006.

\bibitem{DFHQW} W. Dai, Y. Fang, J. Huang, Y. Qin and B. Wang, {\it Regularity and classification of solutions to static Hartree equations involving fractional Laplacians}, Discrete and Continuous Dynamical Systems - A, \textbf{39} (2019), no. 3, 1389-1403.

\bibitem{DFQ} W. Dai, Y. Fang and G. Qin, {\it Classification of positive solutions to fractional order Hartree equations via a direct method of moving planes}, J. Differential Equations, \textbf{265} (2018), 2044-2063.

\bibitem{DHL} W. Dai, Y. Hu and Z. Liu, {\it Sharp reversed Hardy-Littlewood-Sobolev inequality with extended kernel}, to appear in Studia Math., 32 pp, arXiv: 2006.03760.

\bibitem{DL} W. Dai and Z. Liu, {\it Classification of nonnegative solutions to static Schr\"{o}dinger-Hartree and Schr\"{o}dinger-Maxwell equations with combined nonlinearities}, Calc. Var. Partial Differential Equations, \textbf{58} (2019), no. 4, Paper No. 156, 24 pp.

\bibitem{DLQ} W. Dai, Z. Liu and G. Qin, {\it Classification of nonnegative solutions to static Schr\"{o}dinger-Hartree-Maxwell type equations}, SIAM J. Math. Anal., \textbf{53} (2021), no. 2, 1379-1410.

\bibitem{DQ} W. Dai and G. Qin, {\it Classification of nonnegative classical solutions to third-order equations}, Adv. Math., \textbf{328} (2018), 822-857.

\bibitem{DQ0} W. Dai and G. Qin, {\it Liouville type theorems for fractional and higher order H\'{e}non-Hardy type equations via the method of scaling spheres}, Int. Math. Res. Not. IMRN, 2022, 70 pp, DOI: 10.1093/imrn/rnac079.

\bibitem{DQ3} W. Dai and G. Qin, {\it Liouville type theorem for critical order H\'{e}non-Lane-Emden type equations on a half space and its applications}, Journal of Functional Analysis, \textbf{281} (2021), no. 10, Paper No. 109227, 37 pp.

\bibitem{DQ4} W. Dai and G. Qin, {\it Liouville type theorems for elliptic equations with Dirichlet conditions in exterior domains}, J. Differential Equations, \textbf{269} (2020), 7231-7252.

\bibitem{DGZ} J. Dou, Q. Guo and M. Zhu, {\it Subcritical approach to sharp Hardy-Littlewood-Sobolev type inequalities on the upper half space}, Adv. Math., \textbf{312} (2017), 1-45.

\bibitem{DZ} J. Dou and M. Zhu, {\it Sharp Hardy-Littlewood-Sobolev inequality on the upper half space}, Int. Math. Res. Not. IMRN, 2015, no. 3, 651-687.

\bibitem{Fall} M. M. Fall, {\it Entire $s$-harmonic functions are affine}, Proc. Amer. Math. Soc., \textbf{144} (2016), 2587-2592.

\bibitem{Farina} A. Farina, {\it Liouville-type theorems for elliptic problems}, Handbook of Differential Equations: Stationary Partial Differential Equations, \textbf{4} (2007), 61-116.

\bibitem{FK} R. L. Frank and T. K\"{o}nig, {\it Classification of positive singular solutions to a nonlinear biharmonic equation with critical exponent}, Anal. \& PDE, \textbf{12} (2019), no. 4, 1101-1113.

\bibitem{FKT} R. L. Frank, T. K\"{o}nig and H. Tang, {\it Classification of solutions of an equation related to a conformal $\log$ Sobolev inequality}, Adv. Math., \textbf{375} (2020), 107395, 27 pp.

\bibitem{FLS} R. L. Frank, E. Lenzmann and L. Silvestre, {\it Uniqueness of radial solutions for the fractional Laplacian}, Comm. Pure Appl. Math., \textbf{69} (2016), no. 9, 1671-1726.

\bibitem{FL} J. Frohlich, E. Lenzmann, {\it Mean-field limit of quantum bose gases and nonlinear Hartree equation}, in: Sminaire E. D. P. (2003-2004), Expos nXVIII. 26p.

\bibitem{FL1} R. L. Frank and E. H. Lieb, {\it Inversion positivity and the sharp Hardy-Littlewood-Sobolev inequality}, Calc. Var. \& Partial Differential Equations, \textbf{39} (2010), 85-99.

\bibitem{FL} R. L. Frank and E. H. Lieb, {\it A new, rearrangement-free proof of the sharp Hardy-Littlewood-Sobolev inequality}, in Spectral Theory, Function Spaces and Inequalities, Oper. Theory Adv. Appl., \textbf{219}, Springer Basel, Basel, Switzerland, 2012, 55-67.

\bibitem{GNN1} B. Gidas, W. Ni and L. Nirenberg, {\it Symmetry and related properties via maximum principle}, Commun. Math. Phys., \textbf{68} (1979), 209-243.

\bibitem{Gra} C. Graham, {\it Conformal powers of the Laplacian via stereographic projection}, SIGMA Symmetry Integrability Geom. Methods Appl., \textbf{3} (2007), Paper 121, 4 pp.

\bibitem{GJMS} C. Graham, R. Jenne, L. Mason and G. Sparling, {\it Conformally invariant powers of the Laplacian. I. Existence}, J. London Math. Soc., \textbf{46} (1992), no. 3, 557-565.

\bibitem{JLX} Q. Jin, Y. Y. Li and H. Xu, {\it Symmetry and Asymmetry: The Method of Moving Spheres}, Adv. Differential Equations, \textbf{13} (2007), no. 7, 601-640.

\bibitem{JLX1} T. Jin, Y. Y. Li and J. Xiong, {\it On a fractional Nirenberg problem, part I: blow up analysis and compactness of solutions}, J. Eur. Math. Soc., \textbf{16} (2014), no. 6, 1111-1171.

\bibitem{Juhl} A. Juhl, {\it Explicit formulas for GJMS-operators and $Q$-curvatures}, Geom. Funct. Anal., \textbf{23} (2013), no. 4, 1278-1370.

\bibitem{K} T. Kulczycki, {\it Properties of Green function of symmetric stable processes}, Probability and Mathematical Statistics, \textbf{17} (1997), 339-364.

\bibitem{Lei} Y. Lei, {\it Qualitative analysis for the Hartree-type equations}, SIAM J. Math. Anal., \textbf{45} (2013), 388-406.

\bibitem{LMZ} D. Li, C. Miao and X. Zhang, {\it The focusing energy-critical Hartree equation}, J. Diff. Equations, \textbf{246} (2009), 1139-1163.

\bibitem{Lieb} E. H. Lieb, {\it Sharp constants in the Hardy-Littlewood-Sobolev and related inequalities}, Ann. of Math. (2), \textbf{118} (1983), 349-374.

\bibitem{LS} E. Lieb and B. Simon, {\it The Hartree-Fock theory for Coulomb systems}, Comm. Math. Phys., \textbf{53} (1977), 185-194.

\bibitem{Lin} C. S. Lin, {\it A classification of solutions of a conformally invariant fourth order equation in $\mathbb{R}^{n}$}, Comment. Math. Helv., \textbf{73} (1998), 206-231.

\bibitem{Liu} S. Liu, {\it Regularity, symmetry, and uniqueness of some integral type quasilinear equations}, Nonlinear Anal., \textbf{71} (2009), 1796-1806.

\bibitem{Li} Y. Y. Li, {\it Remark on some conformally invariant integral equations: the method of moving spheres}, J. European Math. Soc., \textbf{6} (2004), 153-180.

\bibitem{LZ1} Y. Y. Li and L. Zhang, {\it Liouville type theorems and Harnack type inequalities for semilinear elliptic equations}, J. Anal. Math, \textbf{90} (2003), 27-87.

\bibitem{LZ} Y. Y. Li and M. Zhu, {\it Uniqueness theorems through the method of moving spheres}, Duke Math. J., \textbf{80} (1995), 383-417.

\bibitem{MZ} L. Ma and L. Zhao, {\it Classification of positive solitary solutions of the nonlinear Choquard equation}, Arch. Rational Mech. Anal., \textbf{195} (2010), no. 2, 455-467.

\bibitem{MS} V. Moroz and J. Van Schaftingen, {\it Groundstates of nonlinear Choquard equations: existence, qualitative properties and decay asymptotics}, J. Funct. Anal., \textbf{265} (2013), no. 2, 153-184.

\bibitem{MS1} V. Moroz and J. Van Schaftingen, {\it Existence of groundstates for a class of nonlinear Choquard equations}, Trans. Amer. Math. Soc., \textbf{367} (2015), no. 9, 6557-6579.

\bibitem{N} Qu$\acute{\hat{o}}$c Anh Ng\^{o}, {\it Classification of entire solutions of $(-\Delta)^{N}u+u^{-(4N-1)}=0$ with exact linear growth at infinity in $\mathbb{R}^{2N-1}$}, Proc. Amer. Math. Soc., \textbf{146} (2018), no. 6, 2585-2600.

\bibitem{NN} Qu$\acute{\hat{o}}$c Anh Ng\^{o} and V. H. Nguyen, {\it Sharp reversed Hardy-Littlewood-Sobolev inequality on $\mathbb{R}^{n}$}, Israel J. Math., \textbf{220} (2017), no. 1, 189-223.

\bibitem{Pa} P. Padilla, {\it On some nonlinear elliptic equations}, Thesis, Courant Institute, 1994.

\bibitem{P} S. Paneitz, {\it A quartic conformally covariant differential operator for arbitrary pseudo-Riemannian manifolds}, preprint, available at http://wwww.emis.de/journals, 1983.

\bibitem{Serrin} J. Serrin, {\it A symmetry problem in potential theory}, Arch. Rational Mech. Anal., \textbf{43} (1971), 304-318.

\bibitem{S} L. Silvestre, {\it Regularity of the obstacle problem for a fractional power of the Laplace operator}, Comm. Pure Appl. Math., \textbf{60} (2007), 67-112.

\bibitem{WX} J. Wei and X. Xu, {\it Classification of solutions of higher order conformally invariant equations}, Math. Ann., \textbf{313} (1999), no. 2, 207-228.

\bibitem{Xu} X. Xu, {\it Exact solutions of nonlinear conformally invariant integral equations in $\mathbb{R}^{3}$}, Adv. Math., \textbf{194} (2005), 485-503.

\bibitem{Yu} X. Yu, {\it Classification of solutions for some elliptic system}, preprint, 38 pp, 2021.

\bibitem{Z} N. Zhu, {\it Classification of solutions of a conformally invariant third order equation in $\mathbb{R}^{3}$}, Commun. Partial Differential Equations, \textbf{29} (2004), 1755-1782.

\end{thebibliography}
\end{document}